\documentclass[a4paper]{aart}
\usepackage[centertags]{amsmath}
\usepackage{amsmath}
\usepackage[dutch,british]{babel}
\usepackage{amsfonts}
\usepackage{amssymb}
\usepackage{amsthm}

\usepackage{newlfont}
\usepackage{color}

\usepackage{authblk}

 \usepackage{bbm}
\include{amsthm_sc}

\usepackage{stmaryrd}
\usepackage{mathrsfs}
\usepackage{pstricks,pst-node}

\usepackage{fancyhdr}
\usepackage{graphicx}
\usepackage{fancybox}
\usepackage{geometry}
\usepackage{setspace}
\usepackage{cleveref}
\newcommand{\mfc}{m_{fc}}

\renewcommand{\v}{\mathrm{v}}
\renewcommand{\a}{\mathrm{a}}
\renewcommand{\b}{\mathrm{b}}

\theoremstyle{plain} 
\newtheorem{theorem}{Theorem}[section]
\newtheorem*{theorem*}{Theorem}
\newtheorem{lemma}[theorem]{Lemma}
\newtheorem*{lemma*}{Lemma}
\newtheorem{corollary}[theorem]{Corollary}
\newtheorem*{corollary*}{Corollary}

\newtheorem*{proposition*}{Proposition}
\newtheorem{definition}[theorem]{Definition}
\newtheorem*{definition*}{Definition}

\theoremstyle{definition} 

\newtheorem*{example*}{Example}
\newtheorem{remark}[theorem]{Remark}

\newtheorem*{remark*}{Remark}
\newtheorem*{remarks*}{Remarks}

 \newtheorem{assumption}[theorem]{Assumption}

\numberwithin{equation}{section}
\numberwithin{figure}{section}

\newcommand{\ol}[1]{\overline{#1} \!\,} 

\usepackage{amsmath} 
\usepackage{amssymb}
\usepackage{amsthm}
\usepackage{amsxtra}
\usepackage{bbm}
\usepackage{mathrsfs}
\usepackage{ifthen}

\newcommand{\deq}{\mathrel{\mathop:}=}
\newcommand{\e}[1]{\mathrm{e}^{#1}}
\newcommand{\R} {\mathbb{R}}
\newcommand{\C} {\mathbb{C}}
\newcommand{\N} {\mathbb{N}}

\newcommand{\E} {\mathbb{E}}
\newcommand{\T} {\mathbb{T}}
\newcommand{\adj}{^*} 

\DeclareMathOperator{\Tr}{Tr}

\DeclareMathOperator{\supp}{supp}

\DeclareMathOperator{\re}{\mathrm{Re}}
\DeclareMathOperator{\im}{\mathrm{Im}}

\newcommand{\caB}{{\mathcal B}}

\newcommand{\caD}{{\mathcal D}}
\newcommand{\caE}{{\mathcal E}}

\newcommand{\caG}{{\mathcal G}}

\newcommand{\caL}{{\mathcal L}}

\newcommand{\caO}{{\mathcal O}}
\newcommand{\caP}{{\mathcal P}}

\newcommand{\caY}{{\mathcal Y}}

\newcommand{\bbP}{{\mathbb P}}

\newcommand{\opunit}{\text{1}\kern-0.22em\text{l}}

\newcommand{\fra}{{\mathfrak a}}

\newcommand{\frn}{{\mathfrak n}}

\newcommand{\frX}{{\mathfrak X}}

\newcommand{\bsu}{{\boldsymbol u}}

\renewcommand{\d}{{\mathrm d}}

\newcommand{\beq}{ \begin{equation} }
\newcommand{\eeq}{ \end{equation} }

\newcommand{\baq}{ \begin{eqnarray} }
\newcommand{\eaq}{ \end{eqnarray} }
\newcommand{\bet}{ \begin{theorem} }
\newcommand{\eet}{ \end{theorem} }

\newcommand{\lone}{\mathbbm{1}}

  \let\ga=\gamma 
\let\ka=\kappa \let\la=\lambda

\newcommand{\ben}{\begin{arabicenumerate}}
\newcommand{\een}{\end{arabicenumerate}}

\newcommand{\dd}{\mathrm{d}}
\newcommand{\ii}{\mathrm{i}}

\title{Local deformed semicircle law and complete delocalization for Wigner matrices with random potential}
\vspace{15pt} \normalsize

\author[1]{Ji Oon Lee \thanks{Partially supported by Basic Science Research Program through the National Research Foundation of Korea Grant 2011-0013474}}
\author[2]{Kevin Schnelli}

\affil[1]{\it Department of Mathematical Sciences, KAIST \authorcr \it Daejeon 305-701, Republic of Korea \authorcr \it jioon.lee@kaist.edu}

\affil[2]{\it Department of Mathematics, Harvard University \authorcr \it Cambridge, MA 02138, USA \authorcr \it skevin@math.harvard.edu}

\begin{document}

\maketitle

\begin{abstract}
We consider Hermitian random matrices of the form $H = W + \lambda V$, where $W$ is a Wigner matrix and $V$ a diagonal random matrix independent of $W$. We assume subexponential decay for the matrix entries of $W$ and we choose $\lambda \sim 1$ so that the eigenvalues of $W$ and $\lambda V$ are of the same order in the bulk of the spectrum. In this paper, we prove for a large class of diagonal matrices $V$ that the local deformed semicircle law holds for $H$, which is an analogous result to the local semicircle law for Wigner matrices. We also prove complete delocalization of eigenvectors and other results about the positions of eigenvalues.
\end{abstract}

\vspace{15mm}

{\small

\vspace{5mm}


\vspace{5mm}
}

\section{Introduction}

Consider large matrices whose entries are random variables. Famous examples of such matrices are Wigner matrices: a Wigner matrix is an $N\times N$ real or complex matrix $W=(w_{ij})$ whose entries are independent random variables with mean zero and variance $1/N$, subject to the symmetry constraint $w_{ij}=\overline{w}_{ji}$. The eigenvalues of Wigner matrices are highly correlated; the empirical density of eigenvalues converges to the Wigner semicircle law~\cite{Wigner} in the large $N$ limit. Under some additional moment assumptions on the entries this convergence also holds on very small scales: denoting by $G_W(z)=(W-z)^{-1}$, $z\in\C^+$, the resolvent or Green function of $W$, the convergence of the empirical eigenvalue distribution on scale~$\eta$ at an energy~$E\in\R$ is equivalent to the convergence of the averaged Green function $m_W(z)=N^{-1}\mathrm{Tr}\,G_W(z)$, $z=E+\ii\eta$. The convergence of $m_W(z)$ at the optimal scale $N^{-1}$, up to logarithmic corrections, the so-called {\it local 
semicircle law}, was established for Wigner matrices in a series of papers~\cite{ESY1,ESY2,ESY3}, where it was also shown that the eigenvectors of Wigner matrices are completely delocalized. The proof is based on a self-consistent equation for $m_W(z)$ and the continuity of the Green function~$G(z)$ in the spectral parameter $z$. In~\cite{EYY1,EYY2} convergence of Green function entries was established on optimal scales. Precise estimates on the averaged Green function $m_W(z)$ and on the eigenvalue locations are essential ingredients for proving bulk universality~\cite{ESY4,ESYY} and edge universality~\cite{EYY} for Wigner matrices.

Diagonal matrices with i.i.d.\ random entries are another example of random square matrices. Their eigenvalues are independent, hence uncorrelated, and their eigenvectors are localized. Physically, the diagonal matrix may represent an on-site random potential on a lattice system. Compared to the mean-field nature of Wigner matrices, which are in the weak disorder or the delocalization regime, the diagonal randomness also provides a good example of the strong disorder or localization regime. 

In this paper we consider the interpolation of the two, i.e., the $N \times N$ random matrix
\begin{align} \label{interpolation}
H = \lambda V + W\,,\qquad \lambda\in\R\,,
\end{align}
where $V$ is a real diagonal random matrix, or a `random potential', and $W$ is a real symmetric or complex Hermitian Wigner matrix independent of $V$. The matrix $V$ is properly normalized so that the typical eigenvalues of~$V$ and~$W$ are of the same order. (See Definition \ref{assumption wigner} for a precise statement.) The parameter $\lambda$ determines the relative strength of each part in this model. 

For $\lambda \sim 1$ the eigenvalue density is not solely determined by $V$ or $W$ in the limit $N\to\infty$, but can be described by a functional equation for the Stieltjes transforms of the limiting eigenvalue distributions of~$V$ and~$W$; see~\cite{P, PV}. In general, this limiting eigenvalue distribution, referred to as the {\it deformed semicircle law}, is different from the semicircle distribution. The equal strength of~$V$ and~$W$ makes it non-trivial to find the spectral properties of~$H$. We remark that there are some results related to this model~\cite{CDFF, K, BBCF}.

When $W$ belongs to the Gaussian Unitary Ensemble (GUE), $H$ is called the {\it deformed GUE}, and it can describe Dyson Brownian motion \cite{D} on the real line; see, e.g., \cite{J1}. There has been much important work on various scales of~$\lambda$: Related to symmetry-breaking, transition statistics for eigenvalues in the bulk, especially the nearest neighbor spacing, were studied in \cite{Pa, FN} for~$\lambda \sim N^{1/2}$. In this situation, the diagonal part $\lambda V$ controls the average density, while the GUE part induces fluctuation of eigenvalues. For $\lambda\lesssim 1$, it was shown in~\cite{S1} that universality of eigenvalue correlation functions holds in the bulk of the spectrum. Concerning the edge behaviour, it was shown in \cite{J2} that the transition from the Tracy-Widom to the standard Gaussian distribution occurs on the scale $\lambda \sim N^{-1/6}$. For $\lambda\ll N^{-1/6}$, the Tracy-Widom distribution for the edge eigenvalues was established in~\cite{S2}.

In this paper, we prove, for $\lambda \lesssim 1$ and a large class of random potentials, convergence of the empirical density of eigenvalues down to the optimal scale $1/N$, i.e., we show a {\it local deformed semicircle law} for the averaged Green function $m_H(z)=N^{-1} \Tr(H-z)^{-1}$ , $z=E+\ii\eta$, for all $\eta\gg N^{-1}$. Unlike in the Wigner case, the diagonal disorder of $V$ prevents the diagonal Green function entries from concentrating around $m_H(z)$ for $\lambda\not=0$. Following~\cite{EYY} we derive a self-consistent equation for $m_H(z)$, whose analysis requires a stability estimate that forces interesting conditions on $V$ and $\lambda$. As an intermediate result, we obtain a weak local deformed semicircle law for $m_H(z)$ and complete delocalization of the eigenvectors of $H$ up to the edge.
 In~\cite{EYY2} a `fluctuation average lemma' was proven that yielded optimal rigidity estimates on the location of the eigenvalues of $H$ in the bulk~\cite{EYY2} and up to the edge~\cite{EYY}. Combining the weak deformed semicircle law with the `fluctuation average lemma'~\cite{EYY2} we obtain convergence of $m_H(z)$ on the optimal scale. However, the self-averaging mechanism of the Wigner matrix $W$ in the bulk is only observed after the leading fluctuations stemming from $V$ are subtracted. For example, for $\lambda\sim 1$, the rigidity of eigenvalue location is weaker than in the Wigner case, but we show that the eigenvalue spacing is rigid in the bulk on intermediate scales.

The paper is organized as follows: In Section~\ref{Definition and Results}, we introduce the precise definition and assumptions of the model, state the main results and give a short outline of the proofs. Our assumptions on $\lambda V$ mainly depend on the behaviour of the deformed semicircle law as described in Lemmas~\ref{lemma assumptions 1} and~\ref{squareroot lemma 2} (see also Lemma~\ref{lemma.lee1} in Section~\ref{Weak Deformed Semicircle Law}). For similar results on the deformed semicircle law; see~\cite{B,S2}. In Section \ref{Weak Deformed Semicircle Law}, we prove a weak local (deformed) semicircle law and complete delocalization of eigenvectors. The proof of the local deformed semicircle law follows closely the proof of the weak local semicircle law for sparse random matrices given in~\cite{EKYY1}. In Section~\ref{strong deformed semicircle Law}, we give a proof of the average fluctuation lemma (Lemma \ref{lemma.21}). The proof is inspired by~\cite{EKY}, where fluctuation averages are considered 
for generalized Wigner and random band matrices. The combination of the fluctuation average lemma with the weak local (deformed) semicircle law, yields a proof of the strong local deformed semicircle law as in~\cite{EYY,EKYY1}. In Section~\ref{Identifying the leading corrections}, we identify the leading corrections due to the random diagonal part to the strong local semicircle law on scale~$N^{-1/2}$, in the bulk of the spectrum. In Section~\ref{Density of states}, we establish, following the Helffer-Sj\"ostrand argument given in~\cite{ERSY}, estimates on the density of states and the rigidity of eigenvalues. Using the results obtained in Section~\ref{Identifying the leading corrections}, we obtain estimates on the rigidity of the eigenvalue spacing on intermediate scales in the bulk of the spectrum. Technical details about the square root behaviour and the stability bounds for the deformed semicircle law are given in the Appendix. \\

{\it Acknowledgements:} We thank Horng-Tzer Yau for suggesting this problem to us and for
numerous helpful discussions. We are grateful to Antti Knowles for
discussions and many useful remarks. We also thank Alex Bloemendal,
Paul Bourgade, L\'aszl\'o Erd\H{o}s and an anonymous referee for helpful comments.

\section{Definition and Results} \label{Definition and Results}
In this section, we define our model and state our main results.

\subsection{Free convolution}\label{freeadditiveconvolution}
As first shown in~\cite{P} the limiting spectral distribution of the interpolating model~\eqref{interpolation} is given by the {\it (additive) free convolution measure} of the limiting distribution of the entries of $\lambda V$ and $\mu_{sc}$, the semicircular measure. In a more general setting, the free convolution measure, $\mu_1\boxplus\mu_2$, of two probability measures $\mu_1$ and $\mu_2$, is defined as the distribution of the sum of two freely independent non-commutative random variables, having distributions $\mu_1$, $\mu_2$ respectively; we refer to~\cite{VDN,NS,HP,AGZ}. The (additive) free convolution may also be described in terms of the Stieltjes transform: Let $\mu$ be a probability measure on $\R$, then we define the Stieltjes transform of $\mu$ by
\begin{align}
 m_{\mu}(z)\deq\int_{\R}\frac{\dd\mu(x)}{x-z}\,,\quad\quad z\in\C^+\,.
\end{align}
Note that $m_{\mu}(z)$ is an analytic function in the upper half plane, satisfying $\lim_{y\to\infty}\ii y m_{\mu}(\ii y)=1$. 
As shown in~\cite{V,B1}, the free convolution has the following property: Denote by $m_{\mu_1}$, $m_{\mu_2}$, $m_{\mu_1\boxplus\mu_2}$, the Stieltjes transforms of $\mu_1$, $\mu_2$, $\mu_1\boxplus\mu_2$, respectively. Then there exist two analytic functions $\omega_1$, $\omega_2$, from $\C^+$ to $\C^+$, satisfying $\lim_{y\rightarrow\infty}\,\omega_i(\ii y)/\ii y=1$, ($i=1,2$), such that
\begin{align}\label{system 1}
 m_{\mu_1\boxplus\mu_2}(z)&=m_{\mu_1}(\omega_1(z))=m_{\mu_2}(\omega_2(z))\,,\nonumber\\
 \omega_1(z)+\omega_2(z)&=z-\frac{1}{m_{\mu_1\boxplus\mu_2}(z)}\,,
\end{align}
for $z\in\C^+$. The functions $\omega_i$ are referred to as subordination functions. Note that~\eqref{system 1} also shows that $\mu_1\boxplus\mu_2=\mu_2\boxplus\mu_1$. It was pointed out in~\cite{CG, BeB07} that the system~\eqref{system 1} may be used as an alternative definition of the free convolution. In particular, given $\mu_1$, $\mu_2$, the system~\eqref{system 1} has a unique solution $(m_{\mu_1\boxplus\mu_2},\omega_1,\omega_2)$.

The system~\eqref{system 1} has been used in~\cite{PV} to exploit the limiting eigenvalue distributions for random matrices of the form $A+UBU\adj$, with $A$, $B$ deterministic or random $N\times N$ matrices and $U$ an $N\times N$ random Haar unitary matrix. Free probability theory turned out to be a natural setting for studying global laws for such ensembles; see, e.g.,~\cite{VDN,AGZ}. For more recent treatments, including local laws, we refer to~\cite{K,CDFF,BBCF}.

In case we choose the measure $\mu_2$ as the standard semicircular law $\dd\mu_{sc}(E)=\frac{1}{2\pi}\sqrt{(4-E^2)_+}\dd E$, a simple computation reveals that the Stieltjes transform, $m_{\mu_{sc}}\equiv m_{sc}$, satisfies
\begin{align*}
 m_{sc}(z)=-\frac{1}{z+m_{sc}(z)}\,,\quad\quad z\in\C^+\,.
\end{align*}
Using this information, we can reduce the system~\eqref{system 1}, to the self-consistent equation
\begin{align}\label{lemma.11}
m_{fc}(z)=\int\frac{\dd\mu(x)}{x-z-m_{fc}(z)}\,,\quad\quad z\in\C^+\,,
\end{align}
$\im m_{fc}(z)\ge 0$, for $z\in\C^+$, with $\lim_{y\to\infty}\ii y\, m_{{fc}}(\ii y)=1$, where we have abbreviated $\mu\equiv\mu_1$. Equation~\eqref{lemma.11} is often called the {\it Pastur relation}. A slightly modified version of the functional Equation~\eqref{lemma.11} is the starting point of the analysis in~\cite{P} and also of the present paper; see~\eqref{eq111}.

The (unique) solution of~\eqref{lemma.11} has first been studied in details in~\cite{B}. In particular, it has been shown that $\limsup_{\eta\searrow 0} \im m_{fc}(E+\ii\eta)<\infty$, $E\in\R$, and hence the free convolution measure $\mu_{fc}\equiv\mu\boxplus\mu_{sc}$ is absolutely continuous (for simplicity we denote the density also with $\mu_{fc}$) and we conclude from the Stieltjes inversion formula that
\begin{align*}
 \mu_{fc}(E)=\lim_{\eta\searrow 0}\frac{1}{\pi}\im m_{fc}(E+\ii \eta)\,,\quad\quad E\in\R\,.
\end{align*}
Moreover, it was shown in~\cite{B} that the density $\mu_{fc}$ is analytic in the interior of the support of $\mu_{fc}$. We refer to, e.g.,~\cite{BBG} for further results on the regularity of the free convolution measure.

\subsection{Assumptions}
In this section, we define the model~\eqref{interpolation} in details and list our main assumptions.
\subsubsection{Definition of the model}
\begin{definition}\label{assumption wigner}
Let $W$ be an $N\times N$ random matrix, whose entries, $(w_{ij})$, are independent, up to the symmetry constraint $w_{ij}=\ol{w}_{ji}$, centered, complex random variables with variance $N^{-1}$ and subexponential decay, i.e.,
\begin{align}\label{eq.C0}
 \mathbb{P}\left (\sqrt{N} |w_{ij}|>x\right)\le C_0\,\e{-x^{1/\theta}}\,,
\end{align}
for some positive constants $C_0$ and $\theta>1$. In particular,
\begin{align}
 \E w_{ij}=0\,, \qquad \E|w_{ij}|^p\le C\frac{(\theta p)^{\theta p}}{N^{p/2}}\,, \qquad (p\ge3)\,,
\end{align}
and,  
\begin{align}\label{wigner complex}
\E w_{ii}^2 = \frac{1}{N}\,, \qquad \E |w_{ij}|^2 = \frac{1}{N}\,, \qquad \E w_{ij}^2 = 0\,, \qquad (i \neq j)\,.
\end{align}

\end{definition}
\begin{remark}
We remark that all our methods also apply to {\it symmetric} Wigner matrices, i.e., when $(w_{ij})$ are centered, real random variables with variance $N^{-1}$ and subexponential decay. In this case,~\eqref{wigner complex} gets replaced by
\begin{align}
  \E w_{ii}^2 = \frac{2}{N}\,, \qquad \E w_{ij}^2 = \frac{1}{N}\,, \qquad (i \neq j)\,.
\end{align}

\end{remark}

Let $V=(v_i)$ be an $N\times N$ diagonal random matrix, whose entries $(v_i)$ are real, centered, i.i.d.\ random variables, independent of $W=(w_{ij})$, with law $\mu$. More assumptions on $\mu$ will be stated below. For $\lambda\in\R$, we consider the random matrix
\begin{align}\label{thematrix}
 H=(h_{ij})\deq\lambda V+W\,.
\end{align}
In the next sections, we will choose $\mu$, such that $\supp\mu=[-1,1]$, but we observe that varying $\lambda$ is equivalent to changing the support of $\mu$.

We define the resolvent, or \emph{Green function}, $G(z)$, and the averaged Green function, $m(z)$, of $H$ by
\begin{align*}
 G(z)=(G_{ij}(z))\deq \frac{1}{\lambda V+W-z}\,,\quad\quad m(z)\deq\frac{1}{N}\Tr G(z)\,,\quad\quad z\in\mathbb{C}^{+}\,.
\end{align*} 
Frequently, we abbreviate $G\equiv G(z)$, $m\equiv m(z)$, etc..

\subsubsection{Free convolution}\label{our free convolution}
Following the discussion in Subsection~\ref{freeadditiveconvolution}, we define $m^{\lambda}_{fc}$ as the solution to
\begin{align}\label{eq111}
 m^{\lambda}_{fc}(z)=\int\frac{\dd\mu(v)}{\lambda v-z-m^{\lambda}_{fc}(z)}\,,\quad\quad z\in\C^{+}\,,
\end{align}
with $\im\mfc^{\lambda}(z)\ge0$, $z\in\C^+$. We denote by $\mu^{\lambda}_{fc}$ the corresponding probability measure. For simplicity, we discard the superscript $\lambda$ from our notation. Let us list some easy examples:
\begin{itemize}
\item[$i.$] Choosing $\mu=\delta_1$, one directly sees that $\mu_{fc}$ is a semicircle law of radius $2$ centered at $\lambda$. 
\item[$ii.$] For the choice $\mu=\frac{1}{2}(\delta_{-1}+\delta_{1})$,~\eqref{eq111} reduces to a cubic equation and the support of the measure $\mu_{fc}$ can be inferred from a simple analysis of the discriminant of that equation. As it turns out, the support of $\mu$ consists of a single interval for $\lambda\le 1$ and of two intervals for $\lambda>1$. For simplicity, we will exclude the possibility of $\mu_{fc}$ having support on several disjoint intervals in the following. However, some of our results can be generalized to this setting.
\item[$iii.$] If $\mu$ is the standard Gaussian measure, no closed expression for $\mu_{fc}$ exists, but the moments of~$\mu_{fc}$ can be computed recursively; see~\cite{BDJ}. Moreover, the density of $\mu_{fc}$ is a smooth function with Gaussian tails. Although the Gaussian case is important, we will not deal with measures of unbounded support, but comment on the Gaussian case in Remarks~\ref{remark gaussian} and~\ref{remark gaussian 2}.
\end{itemize}

\subsubsection{Assumptions on $\lambda V$ and $\mu$}
We state our assumptions on $\mu$ the distributions of the entries $(v_i)$ of $V$. From now on, we choose $\mu$ with $\mu=\supp[-1,1]$. Depending on the size of the `perturbation' parameter $\lambda$, we have to distinguish two cases: For $|\lambda|\le 1$, we will assume the following:
\begin{assumption}{[Small $\lambda$]}\label{assumption mu}\label{assumption v}
The entries of the diagonal matrix $V=(v_i)$ are centered, real, i.i.d.\ random variables, independent of $W=(w_{ij})$. For $|\lambda|\le1$, we assume that the distribution of $(v_i)$ has a continuous density $\mu(v)$, such that $\mu(v)>0$, $v\in(-1,1)$, and $\mu(v)=0$, $v\not\in[-1,1]$. 
\end{assumption}
This assumption ensures that the deformed semicircle law $\mu_{fc}$ is supported on a single interval $[L_1,L_2]$, with a square root behaviour at the edges. More precisely, we have the following result:

\begin{lemma}\label{lemma assumptions 1}
 Let $\lambda\le1$ and assume that $\mu$ satisfies Assumption~\ref{assumption v}. Then there are $-\infty<L_1<0<L_2<\infty$, such that $\supp\,\mu_{fc}=[L_1,L_2]$. Moreover, denoting by $\kappa_E$ the distance to the endpoints of the support of $\mu_{fc}$, i.e.,
\begin{align}\label{kappa}
 \kappa_E\deq \min\{|E-L_1|,|E-L_2|\}\,,\quad\quad E\in\R\,,
\end{align}
there exists $C\ge1$ such that
\begin{align}\label{squarerrot behaviour}
 C^{-1}\sqrt{\kappa_E}\le\mu_{fc}(E)\le C\sqrt{\kappa_E}\,,\quad\quad E\in[L_1,L_2]\,.
\end{align}
\end{lemma}
In a slightly different setting this lemma has been proven in~\cite{S2}; see also~\cite{B,ON}. In the Appendix we explain how to adopt the proof in~\cite{S2} to our setting.
\begin{remark}
Above, we have chosen $\lambda$ to be independent of $N$. However, we may choose $\lambda=CN^{-\delta}$, for some constants $C$ and $\delta>0$. In this setting~{\it all} our results hold true as well, in particular, in all the bounds one may simply replace $\lambda$ by $CN^{-\delta}$.
\end{remark}

\begin{remark}
Determining the endpoints $L_1$ and $L_2$ of the support of $\mu_{fc}$ explicitly is, in general, not possible, since it involves solving an implicit equation. However, for $\lambda$ sufficiently small, one can show that $L_1=-2\sqrt{1+\lambda^2}+\caO(\lambda^3)$ and $L_2=2\sqrt{1+\lambda^2}+\caO(\lambda^3)$. Also the measure $\mu_{fc}$ is $\caO(\lambda)$-close to the semicircular measure of radius $2\sqrt{1+\lambda^2}$ in an appropriate distance, but we refrain from going into the details of this `perturbative approach'.
\end{remark}

For $|\lambda|>1$, we have to strengthen the above assumptions, since the square root behaviour at the endpoint of the support of $\mu_{fc}$ may fail, for $\lambda$ large enough. We call a probability measure $\mu$ a {\it Jacobi} measure if it is given by  a density of the form
\begin{align}\label{Jacobi measure}
 \mu(v)=Z^{-1}(1+v)^{\a}(1-v)^{\b} d(v)\lone_{[-1,1]}(v)\,,
\end{align}
where $d\in C^{1}([-1,1])$, with $d(v)>0$, $v\in[-1,1]$, $-1<\a,\b<\infty$, and $Z$ a normalization constant.

We have the following result.
\begin{lemma}\label{lemma assumptions 2a}
 Let $\mu$ be a centered Jacobi measure; see~\eqref{Jacobi measure}. Then, for any $\lambda\in\R$,  there are $-\infty<L_1<0<L_2<\infty$, such that $\supp\,\mu_{fc}=[L_1,L_2]$. Moreover,
\begin{itemize}\label{squareroot lemma 2}
 \item[$(1)$] for $-1<\a,\b\le 1$, for any $\lambda\in\R$, $\mu_{fc}$ has the square root behaviour~\eqref{squarerrot behaviour};

\item[$(2)$] for $1<\a,\b<\infty$, there exists $\lambda_1\equiv{\lambda_1(\mu)}>1$ and $\lambda_2\equiv{\lambda_2(\mu)}>1$ such that,
\begin{itemize}
\item[$(2a)$] for $|\lambda|<\lambda_1$, $|\lambda|<\lambda_2$, $\mu_{fc}$ has the square root behaviour at both endpoints;
\item[$(2b)$] for $|\lambda|<\lambda_1$, $|\lambda|>\lambda_2$, $\mu_{fc}$ has the square root behaviour at the lower endpoint of the support (i.e., for $E\in[L_1,0])$, but there is $C\ge1$, such that
\begin{align}
 C^{-1}(L_2-E)^{\b}&\le  \mu_{fc}(E)\le C(L_2-E)^{\b}\,,\quad\quad E\in[0,L_2]\,.\label{no squareroot}
\end{align}
Analogue statements hold for $|\lambda|>\lambda_1$, $|\lambda|<\lambda_2$, etc..
\end{itemize}
\end{itemize}
\end{lemma}
The proof of the lemma is given in the Appendix.

For our methods to work, we have to exclude situation (2b) of Lemma~\ref{squareroot lemma 2}. For $|\lambda|>1$, we will thus assume:

\begin{assumption}{[Large $\lambda$]}\label{assumption mu'}\label{assumption v'}
The entries of the diagonal matrix $V=(v_i)$ are centered, real, i.i.d.\ random variables, independent of $W=(w_{ij})$. For $|\lambda|>1$, we assume that the distribution of the $(v_i)$ is given by a centered Jacobi measure, and $\lambda$ and $\a,\b$ are chosen as in $(1)$ or $(2a)$ of Lemma~\ref{squareroot lemma 2}.

\end{assumption}

\subsubsection{Notations and Conventions} 
To state our main results, we need some more notations and conventions. For high probability estimates we use two parameters $\xi\equiv\xi_N$ and $\varphi\equiv\varphi_N$:   We assume that
\begin{align}\label{eq.xi}
 a_0<\xi\le A_0\log\log N\,,\quad\quad \varphi=(\log N)^C\,,
\end{align}
for some fixed constants $a_0>2$, $A_0\ge10$, $C\ge 1$. These constants are chosen such that large deviation estimates in Lemma~\ref{lemma.LDE} hold. They only depend on $\theta$ and $C_0$ in~\eqref{eq.C0} and will be kept fixed in the following. 

\begin{definition}
For $\nu>0$, we say an event $\Omega$ has $(\xi,\nu)$-high probability, if
\begin{align*}
 \mathbb{P}(\Omega^c)\le\e{-\nu(\log N)^{\xi}}\,,
\end{align*}
for $N$ sufficiently large.

Similarly, for a given event $\Omega_0$ we say an event $\Omega$ holds with $(\xi,\nu)$-high probability on $\Omega_0$, if
\begin{align*}
 \mathbb{P}(\Omega_0\cap\Omega^c)\le\e{-\nu(\log N)^{\xi}}\,,
\end{align*}
for $N$ sufficiently large.
\end{definition}

For brevity, we occasionally say an event holds with high probability, when we mean $(\xi,\nu)$-high probability. We do not keep track of the explicit value of $\nu$ in the following, allowing $\nu$ to decrease from line to line such that $\nu>0$. From our proof it becomes apparent that such reductions occur only finitely many times.

We use the symbols $\caO(\,\cdot\,)$ and $o(\,\cdot\,)$ for the standard big-O and little-o notation. The notations $\caO\,,\,o$, $\ll$, $\gg$, usually refer to the limit $N\to \infty$. Here $a\ll b$ means $a=o(b)$. We use $c$ and $C$ to denote positive constants that do not depend on $N$. Their value may change from line to line. Finally, we write $a\sim b$, if there is $C\ge1$ such that $C^{-1}|b|\le |a|\le C |b|$, and, occasionally, we write for $N$-dependent quantities $a_N\lesssim b_N$, if there exist constants $C,c>0$ such that $|a_N|\le C(\varphi_N)^{c\xi}|b_N|$.

\subsection{Results}\label{section results} 
In this subsection we state our main results. The presentation of our results follows the one in~\cite{EKYY1}.

 Since we choose the measure $\mu$ to be centered, we may assume that $\lambda\ge 0$, without loss of generality in the following. Fix some $\lambda_0>0$, then we assume that the perturbation parameter $\lambda$ is in the domain
\begin{align}\label{domain of lambda}
 \caD_{\lambda_0}\deq\{\lambda\in\R^+\,:\,\lambda\le\lambda_0\}\,.
\end{align}
Here $\lambda_0$ is an arbitrary constant, but recall that in case $\lambda_0>1$, Assumption~\ref{assumption v'} may not be satisfied for $\a,\b>1$.

We define the spectral parameter $z=E+\ii\eta$, with $E\in\R$ and $\eta>0$. Let $E_0\ge3+\lambda_0$ and define the domain
\begin{align}\label{eq.DL}
 \caD_L\deq\{z=E+\ii \eta\in\C\,:\,| E|\le E_0\,, (\varphi_N)^{L}\le N\eta\le 3N\}\,,
\end{align}
with $L\equiv L(N)$, such that $L\ge 12 \xi$. Here, we chose $E_0$ bigger than $3+\lambda$, since we know that the spectrum of~$W$ lies in the set $\{E\in\R\,:\,|E|\le 3\}$ with high probability. Thus spectral perturbation theory implies that the spectrum of $H$ is contained in $\{E\in\R\,:\,|E|\le 3+\lambda\}$, with high probability.

Recall the definition of $\kappa_E$, the distance to the endpoints of the support of $\mu_{fc}$, in~\eqref{kappa}. In the following, we often abbreviate $\kappa\equiv \kappa_E$.

\subsubsection{Local Laws}
\begin{theorem}\emph{[Strong local law]}\label{thm.strong}
Let $H=\lambda V+W$, where $W$ satisfies the assumptions in Definition~\ref{assumption wigner} and~$\lambda V$ satisfies Assumption~\ref{assumption v} or Assumption~\ref{assumption v'}. Let
\begin{align}\label{eq: strong xi}
 \xi=\frac{A_0+o(1)}{2}\log\log N\,.
\end{align}
Then there are constants $\nu>0$ and $c_1$, depending on the constants ~$\theta$~and~$C_0$ in~\eqref{eq.C0},~$\lambda_0$ in~\eqref{domain of lambda},~$E_0$ in~\eqref{eq.DL}, $A_0$ in~\eqref{eq: strong xi}, and the measure~$\mu$, such that for $L\ge 40\xi$, the events
\begin{align}\label{strong1}
 \bigcap_{\substack{z\in \caD_L\\ \lambda\in\caD_{\lambda_0}}}\left\lbrace\left| m(z)-m_{fc}(z) \right| \le (\varphi_N)^{c_1\xi}\left(\min\left\{\frac{\lambda^{1/2}}{N^{1/4}},\frac{\lambda}{\sqrt{\kappa+\eta}}\frac{1}{\sqrt{N}}\right\}+\frac{1}{N\eta}\right)\right\rbrace
\end{align}
and
\begin{align}\label{strong2}
\bigcap_{\substack{z\in \caD_L\\ \lambda\in\caD_{\lambda_0}}}\left\lbrace\max_{i\not=j}|G_{ij}(z)| \le (\varphi_N)^{c_1\xi}\left(\sqrt{\frac{\im m_{fc}(z)}{N\eta}}+\frac{1}{N\eta}\right)\right\rbrace
\end{align}
both have $(\xi,\nu)$-high probability.
\end{theorem}
\begin{remark}\label{remark gaussian}
 If we choose the entries of $V=(v_i)$ to be independent standard Gaussian random variables, we have the following result: Under the same assumptions as in Theorem~\ref{thm.strong} (except Assumption~\ref{assumption v} or~\ref{assumption v'}) and with similar constants, the events

\begin{align}\label{strongbulk1}
\bigcap_{\substack{z\in \caD_L\\ \lambda\in\caD_{\lambda_0}}}\left\lbrace\left| m(z)-m_{fc}(z) \right| \le(\varphi_N)^{c_1\xi}\left(\frac{\lambda}{\sqrt{N}}+\frac{1}{N\eta}\right)\right\rbrace
\end{align}
and
\begin{align}\label{strongbulk2}
\bigcap_{{\substack{z\in \caD_L\\ \lambda\in\caD_{\lambda_0}}}}\left\lbrace\max_{i\not=j}|G_{ij}(z)| \le(\varphi_N)^{c_1\xi}\left(\sqrt{\frac{\im m_{fc}(z)}{N\eta}}+\frac{1}{N\eta}\right)\right\rbrace
\end{align}
both have $(\xi,\nu)$-high probability. Note, however, that the result only applies to the compact domain $\caD_L$ of the spectral parameter $z$ (i.e., for some fixed $E_0$), but the limiting spectrum of $H=\lambda V+W$ has unbounded support. The proof of the estimates~\eqref{strongbulk1} and~\eqref{strongbulk2} is similar to the proof of Theorem~\ref{thm.strong} and we refrain from stating it explicitly.
\end{remark}
For $\lambda=0$, we have $m_{fc}=m_{sc}$, where $m_{sc}$ is the Stieltjes transform of the standard semicircle law. In this case stronger estimates have been obtained; see, e.g.,~\cite{E}. Roughly speaking, in this situation we have the high probability bounds
\begin{align}\label{estimates on semicircle}
 |m(z)-m_{sc}(z)|\lesssim\frac{ 1}{N\eta}\quad\quad\textrm{ and }\quad\quad |G_{ij}(z)-\delta_{ij}m(z)|\lesssim \sqrt{\frac{\im m_{sc}(z)}{N\eta}}+\frac{1}{N\eta}\,, \qquad(\lambda=0)\,,
\end{align}
(up to logarithmic corrections), within the range of admitted parameters. 

This suggests that the bound on $G_{ij}(z)$, ($i\not=j$), in~\eqref{strong2} is optimal. However, for $\lambda\not=0$, $G_{ii}(z)$ strongly depends on $v_i$, $i\in\{1,\ldots,N\}$, and the diagonal resolvent entries do not concentrate round their mean $m(z)$. This becomes apparent from Schur's complement formula (see, e.g.,~\eqref{eq.148}) and one easily establishes that \mbox{$|G_{ii}(z)-m(z)|\le C\lambda+o(1)$}, with high probability. 

Comparing the estimate on $m-m_{fc}$ in~\eqref{strong1} with the corresponding estimate in~\eqref{estimates on semicircle}, one may suspect that the leading correction terms in~\eqref{strong1} stem from fluctuations of the random variables $(v_i)$. The next theorem asserts that this is indeed true, at least in the bulk of the spectrum: There are random variables, $\zeta_0\equiv\zeta_0^N(z)$, which depend on the random variables $(v_i)$, but are independent of the random variables $(w_{ij})$, such that \mbox{$|m(z)-m_{fc}(z)-\zeta_0(z)|\lesssim (N\eta)^{-1}$} with high probability in the bulk of the spectrum; see~\eqref{thm xi0 equation}. Concerning the spectral edge, we remark that the estimate in~\eqref{strong1} is optimal for $\lambda\ll N^{-1/6}$, but it is not known whether~$\lambda^{1/2}N^{-1/4}$  is the optimal rate for $\lambda\gg N^{-1/6}$.

 To state our next result, we define the domain
\begin{align*}
 \caB_L\deq \caD_L\cap\{ z=E+\ii\eta\in\C\,:\, \sqrt{\kappa_E+\eta}\ge(\varphi_N)^{L\xi}N^{-1/4}\}\,.
\end{align*}
We have the following result:

\begin{theorem}\label{thm xi0}
Let $H=\lambda V+W$, where $W$ satisfies the assumptions in Definition~\ref{assumption wigner} and $\lambda V$ satisfies Assumption~\ref{assumption v} or~\ref{assumption v'}. Then, for any $z\in\caD_L$, $\lambda\in\caD_{\lambda_0}$, there exist random variables $\zeta_0(z)\equiv\zeta_0^N(z)$, depending only on $(v_i)$ such that, with the same constants as in Theorem~\ref{thm.strong}, the event
\begin{align}\label{thm xi0 equation}
 \bigcap_{{\substack{z\in \caB_L\\ \lambda\in\caD_{\lambda_0}}}}\left\lbrace\left| m(z)-m_{fc}(z) -\zeta_0(z)\right| \le (\varphi_N)^{c_1\xi}\frac{1}{N\eta}\right\rbrace
\end{align}
has $(\xi,\nu)$-high probability. The random variables $\zeta_0(z)$ have the following property: The event
\begin{align}\label{thm xi0 equation 2}
 \bigcap_{{\substack{z\in \caD_L\\ \lambda\in\caD_{\lambda_0}}}}\left\lbrace|\zeta_0(z)| \le (\varphi_N)^{c_1\xi}\min\left\{\frac{\lambda^{1/2}}{N^{1/4}},\frac{\lambda}{\sqrt{\kappa+\eta}}\frac{1}{\sqrt{N}}\right\}\right\rbrace
\end{align}
has $(\xi,\nu)$-high probability.
\end{theorem}

\begin{remark}
 The estimates in~\eqref{thm xi0 equation} and~\eqref{thm xi0 equation 2} need some explanation: Choosing $E$ in the bulk of the spectrum, i.e., $\kappa_E\ge\varkappa$, for some $\varkappa>0$, we have
\begin{align}\label{remarks about xi0}
 |m(z)-m_{fc}(z)-\zeta_0(z)|\le  (\varphi_N)^{c_1\xi}\frac{1}{N\eta}\,,\quad\quad |m(z)-m_{fc}(z)|\le (\varphi_N)^{c_1\xi}\left(\frac{\lambda}{\sqrt{N}}+\frac{1}{N\eta}\right)\,,
\end{align}
with high probability and the estimate seems to be optimal. In particular, on microscopic scales, $\eta\ll  N^{-1/2}$, the local fluctuations stem from the Wigner matrix $W$, whereas on intermediate scales, $\eta\sim N^{-1/2}$, the fluctuations due to the Wigner matrix are of the same size as the fluctuations due to the diagonal matrix $V$. Finally, on macroscopic scales, $\eta\sim 1$, the fluctuations are dominated by the matrix $V$.
\end{remark}

\begin{remark}\label{remark on widetildexi0}
 The random variable $\zeta_0\equiv \zeta_0(z)$ is defined as the solution to a quadratic equation; see~\eqref{definition of zeta0} below. In the bulk of the spectrum, we can approximate $\zeta_0$ by
\begin{align}\label{preview widetildex0}
 \widetilde{\zeta}_0(z)\deq \left(1-\int\frac{\dd\mu(v)}{(\lambda v-z-m_{fc}(z))^2} \right)^{-1}\left(\frac{1}{N}\sum_{i=1}^N\frac{1}{\lambda v_i-z-\mfc(z)}-\int\frac{\dd\mu(v)}{\lambda v-z-\mfc(z)} \right)\,,\quad z\in\caD_L\,,
\end{align}
that only depends on $(v_i)$. Note that $\widetilde{\zeta}_0(z)$ embodies, up to a deterministic ($z$-dependent) prefactor, the expected fluctuation corresponding to the $\lambda N^{-1/2}$ term in \eqref{strong1}. For $\kappa_E\ge\varkappa$, for some fixed $\varkappa>0$, we will argue in Subsection~\ref{definition of widetildezeta0}, that $|\zeta_0(z)-\widetilde{\zeta}_0(z)|\ll (N\eta)^{-1}$ and we may thus replace~$\zeta_0$ in~\eqref{remarks about xi0} by~$\widetilde{\zeta}_0$ without changing the bounds; c.f., Lemma~\ref{lemma bound on widetildexi0}.
\end{remark}

\subsubsection{Eigenvector delocalization}
Next, let $\mu_1\le\ldots\le \mu_N$ denote the eigenvalues of $H=\lambda V+W$, and let $\bsu_1,\ldots,\bsu_N$ denote the associated eigenvectors. We use the notation $\bsu_{\alpha}=(u_{\alpha}(i))_{i=1}^N$ for the vector components. All eigenvectors are $\ell^2$-normalized.
The next theorem asserts that, with high probability, all eigenvectors of $H=\lambda V+W$ are completely delocalized:
\begin{theorem}{\emph{[Eigenvector delocalization]}}\label{thm: eigenvector delocalization} 
 Assume that $H=\lambda V+W$ satisfies the assumptions in Definition~\ref{assumption wigner} and Assumption~\ref{assumption mu} or~\ref{assumption mu'}. Then there is a constant $\nu>0$, depending on~$\theta$~and~$C_0$ in~\eqref{eq.C0},~$\lambda_0$ in~\eqref{domain of lambda},~$E_0$ in~\eqref{eq.DL}, $A_0$ in~\eqref{eq: strong xi}, and the measure~$\mu$, such that, for any $\xi$ satisfying~\eqref{eq.xi}, we have
\begin{align*}
 \max_{1\le\alpha\le N}\max_{1\le i\le N}|u_{\alpha}(i)|\le \frac{(\varphi_N)^{4\xi}}{\sqrt{N}}\,,
\end{align*}
 with $(\xi,\nu)$-high probability.
\end{theorem}
\begin{remark}\label{remark gaussian 2} In case the entries of $V=(v_i)$ are independent Gaussian random variables, the situation is more subtle. For any finite $E_0$, there exists a constant $c_{E_0}$, independent of $N$, and a constant $\nu$, depending on $A_0$, $E_0$, $\theta$ and $C_0$,  such that, for any $\xi$ satisfying~\eqref{eq.xi}, the following holds: Let $\alpha\in\{1,\ldots,N\}$ be such that the eigenvalue $\mu_{\alpha}$ satisfies $|\mu_{\alpha}|\le E_0$. Then we have
\begin{align}\label{gaussian eigenvectors}
 \max_{1\le i\le N}|u_{\alpha}(i)|\le c_{E_0}\frac{(\varphi_N)^{4\xi}}{\sqrt{N}}\,,
\end{align}
with $(\xi,\nu)$-high probability. However, $c_{E_0}\to\infty$ and $\nu\to 0$, as $E_0\to\infty$.
\end{remark}

\subsubsection{Density of States}
Next, we state our main results about the {\it local density of states} of $H=\lambda V+W$. For $E_1<E_2$, we define the counting functions
\begin{align}\label{the counting functions}
\frn(E_1,E_2)\deq \frac{1}{N}|\{ \alpha\,:\, E_1<\mu_{\alpha}\le E_2\}|\,,\quad\quad n_{fc}(E_1,E_2)\deq\int_{E_1}^{E_2} \dd x \, \rho_{fc}(x)\,,
\end{align}
where we denote by $\rho_{fc}$ the density of the free convolution measure $\mu_{fc}$.
\begin{theorem}{\emph{[Local density of states]}}\label{theorem density of states}
 Let $H=\lambda V+W$, where $W$ satisfies the assumptions in Definition~\ref{assumption wigner} and $\lambda V$ satisfies Assumption~\ref{assumption mu} or~\ref{assumption mu'}. Let $\xi$ satisfy~\eqref{eq: strong xi}. Then there are constants $\nu>0$ and $c$, depending on~$\theta$~and~$C_0$ in~\eqref{eq.C0},~$\lambda_0$ in~\eqref{domain of lambda},~$E_0$ in~\eqref{eq.DL}, $A_0$ in~\eqref{eq: strong xi}, and the measure~$\mu$, such that, for $L\ge 40\xi$, the following holds: For any $E_1$, $E_2$, satisfying $-E_0\le E_1<E_2\le E_0$, $E_2>E_1+(\varphi_N)^L N^{-1}$, and any $\lambda\in\caD_{\lambda_0}$, the estimate
\begin{align}\label{result density of states 1}
 |\frn(E_1,E_2)-n_{fc}(E_1,E_2)|\le (\varphi_N)^{c\xi}\left(\frac{1}{N}+\frac{\lambda(E_2-E_1)}{\sqrt{\kappa+(E_2-E_1)}}\frac{1}{\sqrt{N}}\right)\,,
\end{align}
holds with $(\xi,\nu)$-high probability. 

Moreover, let $\varkappa>0$. Then, there exists a constant $C_{\varkappa}$, depending only on $\varkappa$, such that, for any $E_1,E_2$, satisfying $L_1+\varkappa\le E_1<E_2\le L_2-\varkappa$,  $E_2>E_1+(\varphi_N)^L N^{-1}$, and any $\lambda\in\caD_{\lambda_0}$, the estimate
\begin{align}\label{results density of states 2}
 |\frn(E_1,E_2)-n_{fc}(E_1,E_2)|\le C_{\varkappa}(\varphi_N)^{c\xi}\left(\frac{1}{N}+\frac{\lambda^{2}(E_2-E_1)^2}{\sqrt{N}}\right)\,,
\end{align}
 holds with $(\xi,\nu)$-high probability.
\end{theorem}
We remark, however, that the estimate in~\eqref{results density of states 2} deteriorates at the edge: $C_{\varkappa}\to\infty$, as $\varkappa\to 0$.

\subsubsection{Rigidity of eigenvalue spacing}
Recall that we denote by $\mu_1\le \mu_2\le\ldots\le \mu_N$, the eigenvalues of \mbox{$H=\lambda V+W$}. The estimates on the density of states in Theorem~\ref{theorem density of states} imply the following result on the rigidity of eigenvalue spacing, which establishes the relation between $|\mu_i - \mu_j|$ and $|i-j|$.

\begin{theorem} \label{rigidity of eigenvalue spacing}
Assume that $H=\lambda V+W$ satisfies the assumptions in Definition~\ref{assumption wigner} and Assumption~\ref{assumption mu} or~\ref{assumption mu'}. Consider $\mu_i < \mu_j$, with $i\ge \epsilon N$ and $j\le (1-\epsilon)N$, for some constant $\epsilon > 0$. Assume that $|i-j| \geq (\varphi_N)^{C' \xi}$, for some constant $C' > C$, where $C$ is the constant in \eqref{results density of states 2}. Then, there exist constants $C_1$, $C_2$, depending only on $\epsilon$, $\max \{ \rho_{fc}(x) : \mu_i \leq x \leq \mu_j \}$ and $\min \{ \rho_{fc}(x) : \mu_i \leq x \leq \mu_j \}$, such that the following holds: For $\xi$ and $\nu>0$, as in Theorem~\ref{theorem density of states}, the estimate
\begin{align*}
C_1 \frac{|i-j|}{N} \leq |\mu_i - \mu_j| \leq C_2 \frac{|i-j|}{N}\,,
\end{align*}
holds with $(\xi,\nu)$-high probability, for all $\lambda\in\caD_{\lambda_0}$. If we assume further that $|i-j| \leq (\varphi_N)^{c \xi} N^{1/2}$, for some constant $c>0$, then there exists a constant $K$ such that
\begin{align} \label{rigidity.spacing}
\left| |\mu_i - \mu_j| - \frac{|i-j|}{N \rho_{fc} (\mu_i)} \right| \leq(\varphi_N)^{K \xi}\frac{1}{N}\,,
\end{align}
with $(\xi,\nu)$-high probability, for all $\lambda\in\caD_{\lambda_0}$.
\end{theorem}

\begin{remark}
The estimate in~\eqref{rigidity.spacing} can be extended to $|i-j| \leq (\varphi_N)^{-c \xi} N^{3/4}$ in the following sense: There exists a constant $K$ such that, for some $\mu_i' \in [\mu_i, \mu_j]$,
\begin{align}\label{stronger estimate on eigenvalue spacing}
\left| |\mu_i - \mu_j| - \frac{|i-j|}{N \rho_{fc} (\mu_i')} \right| \leq(\varphi_N)^{K \xi}\frac{1}{N}\,,
\end{align}
with $(\xi,\nu)$-high probability. The estimate~\eqref{stronger estimate on eigenvalue spacing} easily follows from the proof of Theorem~\ref{rigidity of eigenvalue spacing}.
\end{remark}

\subsubsection{Integrated density of states and rigidity of eigenvalues}
Define the {\it integrated density of states} by
\begin{align*}
 \frn(E)\deq\frac{1}{N}|\{\alpha\,:\,\mu_{\alpha}\le E\}|\,.
\end{align*}
Similarly, we set
\begin{align*}
 n_{fc}(E)\deq\int_{-\infty}^E\rho_{fc}(x)\,\dd x\,,
\end{align*}
where $\rho_{fc}$ denotes the density of the free convolution measure $\mu_{fc}$. Then we have the following result:
\begin{theorem}\label{integrated density of states}
 Let $H=\lambda V+W$, where $W$ satisfies the assumptions in Definition~\ref{assumption wigner} and $\lambda V$ satisfies Assumption~\ref{assumption mu} or~\ref{assumption mu'}. Let~$\xi$ satisfy~\eqref{eq: strong xi}. Then there  are constants $\nu>0$ and $c$, depending on the constants ~$\theta$~and~$C_0$ in~\eqref{eq.C0},~$\lambda_0$ in~\eqref{domain of lambda},~$E_0$ in~\eqref{eq.DL}, $A_0$ in~\eqref{eq: strong xi}, and the measure~$\mu$, such that the event
\begin{align*}
 \bigcap_{\substack{E\in[-E_0,E_0]\\ \lambda\in\caD_{\lambda_0}}}\left\{|\frn(E)-n_{fc}(E)|\le (\varphi_N)^{c\xi}\left(\frac{1}{N}+\frac{\lambda^{3/2}}{N^{3/4}}+\frac{\lambda}{N^{5/6}}+\frac{\lambda\sqrt{\kappa_E}}{\sqrt{N}} \right)\right\}\,,
\end{align*}
has $(\xi,\nu)$-high probability.
\end{theorem}

Our last result concerns the rigidity of the eigenvalue location. We define the `classical' location of the eigenvalue $\mu_{\alpha}$ of $H$, $\gamma_\alpha$, by
\begin{align}\label{classical location}
 \int_{-\infty}^{\gamma_\alpha}\rho_{fc}(x)\dd x=\frac{ \alpha}{N}\,,\qquad \alpha\in\{1,\ldots,N\}\,,
\end{align}
where $\rho_{fc}$ is the density of the free convolution measure $\mu_{fc}$.
\begin{theorem}\label{rigidity of eigenvalues}
Let $H=\lambda V+W$, where $W$ satisfies the assumptions in Definition~\ref{assumption wigner} and $\lambda V$ satisfies Assumption~\ref{assumption mu} or~\ref{assumption mu'}. Let $\xi$ satisfy~\eqref{eq: strong xi}. Then there are constants $\nu>0$ and $c$, depending on the constants ~$\theta$~and~$C_0$ in~\eqref{eq.C0},~$\lambda_0$ in~\eqref{domain of lambda},~$E_0$ in~\eqref{eq.DL}, $A_0$ in~\eqref{eq: strong xi}, and the measure~$\mu$, such that 
\begin{align}\label{rigidity of eigenvalues equation}
 |\mu_{\alpha}-\gamma_{\alpha}|\le (\varphi_N)^{c\xi}\left( N^{-2/3}\left[\widehat\alpha^{-1/3}+\lone\left(\widehat\alpha\le(\varphi_N)^{c\xi}(1+\lambda^{3/2}N^{1/4}\right) \right]+ \lambda^{ 2} N^{-1/3}\widehat\alpha^{-2/3}+\lambda N^{-1/2}\right)\,,
\end{align}
with $(\xi,\nu)$-high probability, for all $\lambda\in\caD_{\lambda_0}$, where we have abbreviated $\widehat\alpha\deq\min\{\alpha,N-\alpha\}$.
\end{theorem}

\begin{remark}
Let us compare this rigidity result with the corresponding rigidity result for Wigner matrices $(\lambda=0)$: In the bulk of the spectrum, where $\alpha\sim N$, we obtain from~\eqref{rigidity of eigenvalues equation},
\begin{align}
  |\mu_{\alpha}-\gamma_{\alpha}|\le (\varphi_N)^{c\xi}\left(\frac{1+C\lambda}{N}+\frac{\lambda}{\sqrt{N}}\right)\,,
\end{align}
with $(\xi,\nu)$-high probability. Thus, for $\lambda\not=0$, the leading corrections in the rigidity estimate arise from fluctuations in the diagonal matrix $V$, and the eigenvalues do not satisfy as strong a rigidity estimate for their locations as in the Wigner case; see, e.g.,~\cite{ESY4,ESYY,ESY1,ESY2}. However, the eigenvalues satisfy a strong rigidity estimate on intermediate scales for their relative position or their spacing; see Theorem~\ref{rigidity of eigenvalue spacing} above. For $\lambda=0$, the rigidity of eigenvalue spacing is an immediate consequence of the rigidity of eigenvalue location.
\end{remark}

\begin{remark}
For $\lambda=0$, the model is known to exhibit bulk universality, (see, e.g.,~\cite{ESY4, ESYY, EYY1, EYY2, EY}), and it is easy to see that bulk universality holds true for the choice $\lambda=CN^{-\delta}$, with $\delta\ge1/2$. In case $W$ is a GUE matrix, bulk universality has been proved to hold for $\lambda$ order one; see~\cite{S1}. 

A corner stone in the proof of bulk universality for Wigner matrices in, e.g.,~\cite{EYY}, is the rigidity estimate
\begin{align}
 \frac{1}{N}\sum_{\alpha=1}^N\E|\mu_\alpha-\gamma_\alpha|^2\le C N^{-1-2\fra}\,,\qquad (\lambda=0)\,,
\end{align}
for some $\fra>0$, where $(\mu_\alpha)$ are the eigenvalues of $W$ and $(\gamma_\alpha)$ are the classical locations of the eigenvalues (with respect to the standard semicircle law). For $\lambda\not=0$, one can show that~\eqref{rigidity of eigenvalues equation} implies
\begin{align}\label{rigidity estimate for lambda order one}
 \frac{1}{N}\sum_{\alpha=1}^N\E|\mu_\alpha-\gamma_\alpha|^2\le C(\varphi_N)^{c\xi}\lambda^2N^{-1}+C N^{-1-2\fra}\,,
\end{align}
for some constants $C$, $c$ and $\fra>0$, where now $(\gamma_\alpha)$ denote the classical locations with respect to the deformed semicircle law. For $\lambda=CN^{-\delta}$, with $\delta>0$, it is thus conceivable that one can prove bulk universality following the lines of, e.g.,~\cite{EYY}, using local ergodicity of  Dyson Brownian motion and Green function comparison. Note that for $\lambda\ll 1$, the limiting eigenvalue distribution of $H$ is the semicircle law. For $\lambda$ order one, the proof of~\cite{EYY} seems not to be applicable without 
greater modifications. 

\end{remark}

\begin{remark}

When $V$ is a deterministic instead of a random diagonal matrix, one can still prove the results in this paper, provided that the limiting density of the eigenvalues of $V$ satisfies the required assumptions. Moreover, when $V$ is symmetric and $W$ is GUE or GOE, it is possible to prove the same results after diagonalizing $\lambda V + W$, due to the invariance of GUE (GOE) under unitary (orthogonal) conjugation; some knowledge on the convergence of empirical eigenvalue distribution is required, e.g., the analogue statement to~\eqref{mini lln} below. The extension to the more general case where $V$ is non-diagonal and $W$ is a general Wigner matrix requires a further investigation; the usual Lindeberg replacement strategy using moment matching conditions (e.g., \cite{TV1}) may not be sufficient, because the statements in this paper are stronger than those one gets from the moment matching conditions in the sense that the former hold with high probability 
while the latter hold after taking expectation.

\end{remark}

\subsection{Outline of proofs}
In this subsection, we briefly outline the proofs of the main results.

In Section~\ref{Weak Deformed Semicircle Law}, we derive a {\it weak local deformed semicircle law},~Theorem~\ref{thm.weak}. Following the lines of~\cite{EKYY1} we derive a {\it (weak) self-consistent equation} for $m-m_{fc}$.  The main differences with the Wigner case are: (1) The limiting eigenvalue distribution follows the deformed semicircle law instead of the semicircle law. For Wigner matrices, the stability of this self-consistent equation is obtained by elementary calculus using the exact form of $m_{sc}$. For the deformed model, the stability of the self-consistent Equation~\eqref{eq.172} follows from the ($\mu$- and $\lambda$-dependent) stability estimate~\eqref{stability bound}.  (2) When taking the normalized 
trace of the Green function, we average over the random variables $(v_i)$; see~\eqref{eq.173}. Eventually, we replace this average by its expected value. This replacement results in an error term that is, according to the CLT, of order $\lambda N^{-1/2}$ (up to logarithmic corrections); see~\eqref{mini lln}. These fluctuations should be compared with the other error terms. Among those error terms the dominating one, the~$Z_i$ defined in~\eqref{eq.146}, is of order $(N\eta)^{-1/2}$. Since $z\in\caD_L$, we can combine the error terms and our estimates in Theorem~\ref{thm.weak} are the same as the corresponding estimates in~\cite{EKYY1}. However, due to the order one diagonal entries~$(v_i)$, $G_{ii}$ is not self-averaging. In Section~\ref{Weak Deformed Semicircle Law}, we also prove Theorem~\ref{thm: eigenvector delocalization}: Delocalization of eigenvectors can be obtained as in~\cite{EYY1} as a corollary of the local deformed semicircle law, Theorem~\ref{thm.strong}.

In Section~\ref{strong deformed semicircle Law}, we prove the fluctuation average bound $|\frac{1}{N}\sum_i{Z_i}|\lesssim(N\eta)^{-1}$; see Lemma \ref{lemma.21}. The proof of this lemma is inspired by~\cite{EKY}, where more general fluctuation averages are considered for Wigner and random band matrices. Our treatment is in so far different as fluctuations of (diagonal elements of) Green functions are not self-averaging. Having established an optimal error bound on the average of $(Z_i)$, we have to keep track of the ($\eta$-independent) $\lambda N^{-1/2}$ fluctuation from the CLT alluded to above ($(N\eta)^{-1}$ cannot be compared with $\lambda N^{-1/2}$ on $\caD_L$). As in~\cite{EYY1}, we obtain a {\it (strong) self-consistent equation}, Equation~\eqref{self}, whose stability analysis yields a proof of the strong local deformed semicircle law.

In Section~\ref{Identifying the leading corrections}, we prove Theorem~\ref{thm xi0}. In~\eqref{definition of zeta0}, we define a random variable~$\zeta_0$, depending only on $(v_i)$, such that $|m(z)-m_{fc}(z)-\zeta_0(z)|$ is minimized. This can be achieved by defining $\zeta_0$ as the solution to the strong self-consistent equation with all error terms but the ones depending only on $(v_i)$ discarded. Defining $\zeta_0$ in this way yields an optimal bound on $|m(z)-m_{fc}-\zeta_0(z)|$ away from the spectral edges, or more precisely, for energies~$E$ satisfying $\kappa_E\gtrsim N^{-1/4}$.

In Section~\ref{local density of states}, we establish, using the Helffer-Sj\"{o}strand formula in the argument in~\cite{ERSY}, estimates on the density of states and the rigidity of eigenvalues. Using the results obtained in Section~\ref{Identifying the leading corrections} on $\zeta_0$, we obtain in Section~\ref{section bulk fluctuation} estimates on the rigidity of the eigenvalue spacing on intermediate scales in the bulk of the spectrum. More precisely, in the bulk of the spectrum we approximate $\zeta_0(z)$ by $\widetilde{\zeta}_0(z)$; see Remark~\ref{remark on widetildexi0} above for a definition. In Lemma~\ref{derivative zeta0}, we show that the $z$-dependent random variables $\widetilde{\zeta}_0(E+\ii\eta)$ is, for fixed $\eta$, a slowly varying function of $E$ in the bulk of the spectrum. This in turn can be used to get more precise estimates on the density of states in the bulk of the spectrum; see~\eqref{results density of states 2}. In Section~\ref{Eigenvalue spacing in the bulk}, we prove 
Theorem \ref{rigidity of eigenvalue spacing}, based on results on the local density of states. Using once more the random variables $\widetilde{\zeta}_0$, one can obtain stronger estimates on the eigenvalue spacing in the bulk of the spectrum; see~\eqref{stronger estimate on eigenvalue spacing}. Following~\cite{EKYY1}, we prove in Section~\ref{section: Integrated density of states and rigidity of eigenvalues} Theorems~\ref{integrated density of states} and~\ref{rigidity of eigenvalues}.

In the Appendix, we discuss properties of the deformed semicircular law, $\mu_{fc}$, in particular, we prove Lemma~\ref{lemma assumptions 1} ($\lambda\le 1$), Lemma~\ref{lemma assumptions 2a} ($\lambda>1$). In a slightly different setting Lemma~\ref{lemma assumptions 1} has been proven in~\cite{S1}, but we recall parts the proof, since it is used in the proof of Lemma~\ref{lemma assumptions 2a}. The first part of~\ref{lemma assumptions 2a}, can be proved in a similar way, but we have to impose some stronger assumptions on $\mu$; c.f.,~\eqref{jacobi measure}. The second part of Lemma~\ref{lemma assumptions 2a} shows that the deformed semicircle law may not have a square root behaviour at the edge. Our proof is based on elementary estimates, but we expect that the statement can be proven using methods of complex analysis.

\section{Weak Deformed Semicircle Law} \label{Weak Deformed Semicircle Law}

In this section, we prove a weaker form of the deformed semicircle law. This {\it weak deformed semicircle law} will be used to prove the {\it strong deformed law} in Theorem~\ref{thm.strong}. Moreover, complete delocalization of eigenvectors is a direct consequence of the weak law stated in the next theorem.

\begin{theorem}\emph{ [Weak deformed semicircle law]}\label{thm.weak}
 Let $H=\lambda V+W$ satisfy the assumptions in Definition~\ref{assumption wigner} and Assumption~\ref{assumption v} or~\ref{assumption v'}. Then there are constants $C$, $\nu>0$, depending on the constants ~$\theta$~and~$C_0$ in~\eqref{eq.C0},~$\lambda_0$ in~\eqref{domain of lambda},~$E_0$ in~\eqref{eq.DL}, $A_0$ in~\eqref{eq: strong xi}, and the measure~$\mu$, such that, for
\begin{align*}
 a_0\le \xi\le A_0 \log\log N\,,\quad\quad L\ge 12\xi\,,
\end{align*}
the event
\begin{align}\label{eq.weak1}
\bigcap_{\substack{z\in \caD_L\\ \lambda\in\caD_{\lambda_0}}}\left\{\max_{i\not=j}|G_{ij}(z)|\le C \frac{{(\varphi_N)}^{\xi}}{\sqrt{N\eta}}\right\}\,,
\end{align}
has $(\xi,\nu)$-high probability.

Denote by $\E^{v_i}$, the expectation with respect to the random variable $v_i$, $i\in\{1,\ldots,N\}$. Then the event
 \begin{align}\label{eq.weak2}
\bigcap_{\substack{z\in \caD_L\\ \lambda\in\caD_{\lambda_0}}}\left\{\max_{1\le i\le N}|\E^{v_i}G_{ii}(z)-m(z)|\le C(\varphi_N)^{\xi}\left(\frac{\lambda}{\sqrt{N}}+\frac{1}{(N\eta)^{1/3}}\right)\right\}\,,
\end{align}
has $(\xi,\nu)$-high probability.

Moreover, we have the weak local deformed semicircle law: The event
\begin{align}\label{eq.weak3}
\bigcap_{\substack{z\in \caD_L\\ \lambda\in\caD_{\lambda_0}}}\left\{|m(z)-m_{fc}(z)|\le C\frac{(\varphi_N)^{\xi}}{(N\eta)^{1/3}}\right\}\,,
\end{align}
has $(\xi,\nu)$-high probability.
\end{theorem}
The rest of the section is devoted to the proof of Theorems~\ref{thm.weak} and ~\ref{thm: eigenvector delocalization}. The proof follows closely the proof for Wigner matrices, see,~\cite{EKYY1} and~\cite{EYY}. We will always assume that $W$ satisfies the assumptions in Definition~\ref{assumption wigner} and that $\lambda V$ satisfies Assumption~\ref{assumption v} or~\ref{assumption v'}.

\subsection{Preliminaries}
\subsubsection{Some properties of $\mu_{fc}$ and $m_{fc}$}
The next lemma collects some useful properties of $m_{fc}$ under Assumptions~\ref{assumption v} or~\ref{assumption v'}.
\begin{lemma}\label{lemma.lee1}\label{lemma.12}
There exist $L_1<L_2$ such that the free convolution measure $\mu_{fc}$ has support $[L_1,L_2]$.  For all $z=E+\ii\eta\in \caD_L$, $\lambda\in\caD_{\lambda_0}$, the Stieltjes transform, $m_{fc}$, of $\mu_{fc}$ has the following properties:
\begin{itemize}
\item[i.] Let $\kappa\deq\min\{|E-L_1|, |E-L_2|\}$, then
\begin{align}\label{behaviour of mfc}
 \im m_{fc}(z)\sim\begin{cases} \sqrt{\kappa+\eta}\,,\quad& E\in[L_1,L_2]^{\phantom{c}}\,,\\
                   \frac{\eta}{\sqrt{\kappa+\eta}}\,,\quad &E\in[L_1,L_2]^c\,.
                  \end{cases}
\end{align}
\item[ii.] There exist constants $C,c>0$, depending on $\mu$, $E_0$ and $\lambda_0$, such that
\begin{align}\label{stability bound}
 c\le |\lambda -z-m_{fc}(z)|\le C\,.
\end{align}
\end{itemize}
\end{lemma}
We refer to~\eqref{stability bound} as `stability bound' and remark that a similar condition has already been used in~\cite{S2}. The proof of Lemma~\ref{lemma.lee1} is given in the Appendix.

\subsubsection{Minors}
\begin{definition}
Let $\T\subset\{1,\ldots,N\}$. Then we define $H^{(\T)}$ as the $(N-|\T|)\times(N-|\T|)$ minor of $H$ obtained by removing all columns and rows of $H$ indexed by $i\in\T$.  Note that we do not change the names of the indices of $H$ when defining $H^{(\T)}$.
More specifically, we define an operation $\pi_i$, $i\in\{1,\ldots,N\},$ on the probability space by
\begin{align*}
 (\pi_i(H))_{kl}\deq\lone(k\not=i)\lone(l\not=i)h_{kl}\,.
\end{align*}
Then, for $\T\subset\{1,\ldots,N\}$, we set $\pi_{\T}\deq\prod_{i\in\T}\pi_i$ and define
\begin{align*}
 H^{(\T)}\deq((\pi_{\T}(H)_{ij})_{i,j\not\in\T}\,.
\end{align*}
The Green functions $G^{(\T)}$, are defined in an obvious way using $H^{(\T)}$. Moreover, we use the shorthand notation
\begin{align*}
 \sum_{i}^{(\T)}\deq\sum_{\substack{i=1\\i\not\in\T}}^N\,\,,
\end{align*}
and abbreviate $(i)=(\{i\})$ and, similarly, $(\T i)=(\T\cup\{i\})$.
Finally, we set
\begin{align*}
 m^{(\T)}\deq\frac{1}{N}\sum_{i}^{(\T)}G_{ii}^{(\T)}\,.
\end{align*}
Here, we use the normalization $N^{-1}$, instead $(N-|\T|)^{-1}$, since it is more convenient for our computations.
\end{definition}

\subsubsection{Resolvent identities}
The next lemma collects the main identities between resolvent matrix elements of $H$ and $H^{(\T)}$.

\begin{lemma}
Let $H$ be an $N\times N$ matrix. Consider the Green function $G(z)\equiv G\deq(H-z)^{-1}$, $ z\in\C^+$. Then, for $i,j,k\in\{1,\ldots,N\}$, the following identities hold:
\begin{itemize}
 \item[-] {\it Schur complement/Feshbach formula:}\label{feshbach} For any $i$,
\begin{align}\label{schur} 
 G_{ii}=\frac{1}{h_{ii}-z-\sum_{m,n}^{(i)}{h_{im} G_{mn}^{(i)}}h_{ni}}\,.
\end{align}
\item[-] For $i\not=j$,
\begin{align}\label{res.1}
 G_{ij}=-G_{ii}G_{jj}^{(i)}(h_{ij}-\sum_{m,n}^{(ij)}h_{im}G_{mn}^{(ij)}h_{nj})\,.
\end{align}
\item[-] For $i,j\not=k$,
\begin{align}\label{res.2}
 G_{ij}=G_{ij}^{(k)}+\frac{G_{ik}G_{kj}}{G_{kk}}\,.
\end{align}

\item[-]{\it Ward identity:} For any $i$,
\begin{align}\label{ward}
 \sum_{n=1}^N|G_{in}|^2=\frac{1}{\eta}\im G_{ii}\,,
\end{align}
where $\eta=\im z$.
\end{itemize}
\end{lemma}
For a proof we refer to, e.g.,~\cite{EKYY1}. 

\subsubsection{Large deviation estimates}
We collect here some useful large deviation estimates for random variables with slowly decaying moments.
\begin{lemma}\label{lemma.LDE}
 Let $(a_i)$ and $(b_i)$ be centered and independent complex random variables with variance $\sigma^2$ and having subexponential decay
\begin{align*}
 \mathbb{P}\left(|a_i|\ge x\sigma\right)\le C_0\,\e{-x^{1/\theta}}\,,\qquad\mathbb{P}\left(|b_i|\ge x\sigma\right)\le C_0\,\e{-x^{1/\theta}}\,,
\end{align*}
for some positive constant $C_0$ and $\theta>1$. Let $A_i\in\C$ and $B_{ij}\in\C$. Then there exist constants $a_0>1$, $A_0\ge 10$ and $C\ge 1$, depending on $\theta$ and $C_0$, such that for $a_0\le\xi\le A_0\log\log N$, and $\varphi_N=(\log N)^{C}$,
\begin{align}
\mathbb{P}\left( \left|\sum_{i=1}^N A_ia_i   \right|\ge (\varphi_N)^{\xi}\sigma\left(\sum_{i=1}^N|A_i|^2\right)^{1/2}\right)&\le \e{-(\log N)^{\xi}}\,,\label{LDE1} \\[1mm]
\mathbb{P}\left(\left|\sum_{i=1}^N\ol{a}_i B_{ii}a_i-\sum_{i=1}^N \sigma^2 B_{ii}\right|\ge (\varphi_N)^{\xi}\sigma^2\left(\sum_{i=1}^N|B_{ii}|^2\right)^{1/2}\right)&\le \e{-(\log N)^{\xi}}\,,\label{LDE2}\\[1mm]
\mathbb{P}\left(\left|\sum_{i\not=j}^N\ol{a}_i B_{ij}a_j\right|\ge(\varphi_N)^{\xi}\sigma^2\left(\sum_{i\not=j}|B_{ij}|^2  \right)^{1/2}\right)&\le\e{-(\log N)^{\xi}}\label{LDE3}\,,\\[1mm]
\mathbb{P}\left(\left|\sum_{i,j=1}^N\ol{a}_i B_{ij}b_j\right|\ge(\varphi_N)^{\xi}\sigma^2\left(\sum_{i\not=j}|B_{ij}|^2  \right)^{1/2}\right)&\le\e{-(\log N)^{\xi}}\label{LDE4}\,,\end{align}
for $N$ sufficiently large.
\end{lemma}
We refer to~\cite{EYY1} for a proof.

\subsubsection{Schur complement formula} The proof of Theorem~\ref{thm.weak} starts with Schur's formula
\begin{align}\label{eq.145}
 G_{ii}=\frac{1}{h_{ii}-z-\sum_{k,l}^{(i)}h_{ik}G_{kl}^{(i)}h_{li}}\,,\quad \quad z\in \caD_L\,,
 \end{align}
where, for brevity, $G_{ij}\equiv G_{ij}(z)$. Define $\E_i$ to be the partial expectation with respect to the $i^{\textrm{th}}$-column/row of~$W$ and set
\begin{align}\label{eq.146}
Z_i\deq(\lone-\E_i)\sum^{(i)}_{k,l}h_{ik}G_{kl}^{(i)}h_{li}&=\sum_{k,l}^{(i)}(h_{ik}G^{(i)}_{kl}h_{li}-\frac{1}{N}\delta_{kl}G_{kl}^{(i)})\nonumber\\ 
&=\sum_{k}^{(i)}(|w_{ik}|^2-\frac{1}{N})G_{kk}^{(i)}+\sum_{k\not=l}^{(i)}w_{ik}G_{kl}^{(i)}w_{li}\,,
\end{align}
here we used $h_{ik}=w_{ik}+\lambda\delta_{ik}v_i$.
For a family of random variables $(F_1,\ldots, F_N)$ we introduce the notation
\begin{align}\label{eq.147}
 [F]\deq\frac{1}{N}\sum_{i=1}^N F_i\,.
\end{align}
Recalling the definition $m^{(i)}=\frac{1}{N}\Tr G^{(i)}=\frac{1}{N}\sum_{k}^{(i)}G_{kk}^{(i)}$, we obtain from Equations~\eqref{eq.145} and~\eqref{eq.146}
\begin{align}\label{eq.148}
G_{ii}&=\frac{1}{\lambda v_i+w_{ii}-z-m^{(i)}-Z_i}\nonumber\\
&=\frac{1}{\lambda v_i-z-m_{fc}-([\v]-\caY_i)}\,,
\end{align}
where 
\begin{align}\label{eq.149}
\v_i\deq G_{ii}-m_{fc}\,,\quad
 \caY_i\deq w_{ii}- Z_{i}-(m^{(i)}-m)\,.
\end{align}

Note the difference between $v_i$ and $w_{ii}$: Since we assumed that the (rescaled) entries of $W$ have subexponential decay, we have
\begin{align}\label{eq.1410}
|w_{ij}|\le C\frac{(\varphi_N)^{\xi}}{\sqrt{N}}\,,
\end{align}
with $(\xi,\nu)$-high probability, whereas $v_i=\caO(1)$, almost surely.

\begin{lemma}\label{lemma.13} There is a constant $C$ such that, for $z \in \caD_L $, $\lambda\in\caD_{\lambda_0}$ and $1\le i\le N$, we have
\begin{align}\label{eq.140}
 |m(z)-m^{(i)}(z)|\le \frac{C}{N\eta}\,.
\end{align}
\end{lemma}
\begin{proof}
 The claim follows from Cauchy's interlacing property of eigenvalues of $H$ and its minor $H^{(i)}$. For a detailed proof we refer to~\cite{E}.
\end{proof}

\subsection{A priori estimates on the domain $\Omega(z)$}
Define the z-dependent control quantities
\begin{align}\label{eq.151}
\Lambda_o\deq\max_{i\not= j}|G_{ij}|\,,\qquad \Lambda_d\deq\max_{i}|G_{ii}|\,,\qquad\Lambda\deq |m-m_{fc}|\,.
\end{align}
Note that these quantities also depend on $\lambda$, but we do not display this dependence, since, as we shall see, uniformity in $\lambda$ can always be achieved on the domain $\caD_{\lambda_0}$ using the stability bound~\eqref{stability bound}.

For $z\in \caD_L$, we define an event $\Omega(z)$ by
\begin{align}\label{eq.152}
\Omega(z)\deq\{\Lambda_o\le (\varphi_N)^{-2\xi}\}\cap \{\Lambda\le (\varphi_N)^{-2\xi}\}\,.
\end{align}
First, we check that we can bound the matrix elements of the Green function of the minor $H^{(i)}$ in terms of the matrix elements of the Green function of $H$. 
\begin{lemma}\label{lemma.111}
Let $z\in\caD_L$, $\lambda\in\caD_{\lambda_0}$. Then there are constants $C,c>0$ such that, for any \mbox{$\T\subset\{1,\ldots,N\}$} with $|\T|\le 10$, the following statements hold with high probability on $\Omega(z)$: 
\begin{itemize}
 \item[$i.$]
For any $i\not\in\T$,
\begin{align}\label{eq.153}
 c\le|G_{ii}^{(\T)}|\le C\,.
\end{align}
\item[$ii.$]For any $i,j\not\in\T$, $i\not=j$,
\begin{align}\label{eq.154}
c\Lambda_o\le|G_{ij}^{(\T)}|\le C\Lambda_o\,.
\end{align}
\item[$iii.$] 
\begin{align}\label{eq.155}
 |m-m^{(\T)}|\le C\Lambda_o^2\,.
\end{align}
\end{itemize}
Moreover, the constants $C$ and $c$ can be chosen uniformly in $z\in\caD_L$, $\lambda\in\caD_{\lambda_0}$.
\end{lemma}
\begin{proof}
Let $z\in \caD_L$ and $\lambda\in\caD_{\lambda_0}$. We will successively use ~\eqref{res.2}, i.e.,
\begin{align}\label{eq.156}
G_{ij}-G^{(k)}_{ij}=\frac{G_{ik}G_{kj}}{G_{kk}}\,.
\end{align}
Since we are working on $\Omega(z)$ we have $|G_{ij}|\le \Lambda_o\le  (\varphi_N)^{-2\xi}$, for $i\not=j$. Next, Equation~\eqref{eq.148} yields
\begin{align*}
\left |\frac{1}{G_{ii}}\right|= |z+m^{(i)}-\lambda v_i-w_{ii}-Z_i|\,.
\end{align*}
By the large deviation estimates~\eqref{LDE2},~\eqref{LDE3}, the Ward identity~\eqref{ward} and Inequality~\eqref{eq.140} we have
\begin{align}\label{eq.157}
|Z_i|\le C(\varphi_N)^{\xi}\left(\frac{1}{N^2}\sum_{k,l}^{(i)}|G^{(i)}_{kl}|^2\right)^{1/2}&\le C(\varphi_N)^{\xi}\sqrt{\frac{\im m^{(i)}}{N\eta}}\nonumber\\
&\le C(\varphi_N)^{\xi}\sqrt{\frac{\Lambda + \im m_{fc}} {N\eta}}+C(\varphi_N)^{\xi}\frac{1}{N\eta}\,,
\end{align}
with high probability on $\Omega(z)$. Since $\Lambda\le (\varphi_N)^{-2\xi}$ on $\Omega(z)$ and since $N\eta\ge (\varphi_N)^{12\xi}$, we conclude that $|Z_i|=o(1)$, with high probability on $\Omega(z)$, for $z\in \caD_L$ and $\lambda\in\caD_{\lambda_0}$. Finally, since 
\begin{align*}
\big|m^{(i)}\big|=\big|m_{fc}\big|+\caO\left(\frac{1}{N\eta}+(\varphi_N)^{-2\xi}\right)\,,
\end{align*}
on $\Omega(z)$, by~\eqref{eq.140}, we find 
\begin{align*}
  \left|\frac{1}{G_{ii}}\right|&= |\lambda v_i-z-m_{fc}|+o(1)\,,
\end{align*}
with high probability on $\Omega(z)$. This, together with the stability bound \eqref{stability bound}, proves the lower and upper bound on $G_{ii}$. Note that by~\eqref{stability bound}, we can choose the upper and lower bound on $G_{ii}$ to be uniform in $\lambda\in\caD_{\lambda_0}$, $z\in\caD_{L}$.

Statements $i$-$iii$ now follow by iterating~\eqref{eq.156}.
\end{proof}

Next, we define the control parameter $\Psi(z)$, for $z\in \caD_L$, by
\begin{align}\label{eq.158}
 \Psi(z)\deq (\varphi_N)^{\xi}\sqrt{\frac{\Lambda+\im m_{fc}}{N\eta}}\,,
\end{align}
 where $\Lambda=|m-m_{fc}|$. Again, we suppress the $\lambda$-dependence of $\Psi(z)$ from our notation. We will use $\Psi\equiv\Psi(z)$ to bound various quantities in the following:

\begin{lemma}\label{lemma.14}
Let $z\in\caD_L$, $\lambda\in\caD_{\lambda_0}$. Then there is a constant $C$ such that we have with $(\xi,\nu)$-high probability on~$\Omega(z)$:
\begin{align}
\Lambda_o&\le C\Psi\,,\label{eq.161}\\[1.5mm]
\max_i |Z_i|&\le C\Psi\,,\label{eq.162}\\
\max_i |\caY_i|&\le  C\Psi\,.\label{eq.163}
\end{align}
\end{lemma}
The constant $C$ can be chosen uniformly in $z\in\caD_L$, $\lambda\in\caD_{\lambda_0}$.
\begin{proof}
 We prove~\eqref{eq.161}. Let $i\not= j$, then by Equation~\eqref{res.1}, the large deviation estimates of Lemma~\ref{lemma.LDE} and Inequality~\eqref{eq.1410},
\begin{align*}
 |G_{ij}|\le C(|w_{ij}|+\sum_{k,l}^{(ij)}|w_{ik}G^{(ij)}_{kl}w_{lj}|)&\le C(\varphi_N)^{\xi}\left(\frac{1}{\sqrt{N}}+ \sqrt{\frac{1}{N^2} \sum_{k,l}^{(ij)}|G_{kl}^{(ij)}|^2} \right)\\
&=C(\varphi_N)^{\xi}\left(\frac{1}{\sqrt{N}}+\sqrt{\frac{\im m^{(ij)}}{N\eta}}\right)\,,
\end{align*}
with high probability, where we used in the last step the Ward identity~\eqref{ward}. Since $|m^{(ij)}-m|\le C\Lambda_o^2$, by Lemma~\ref{lemma.111}, we get
\begin{align*}
|G_{ij}|\le C\left(\frac{(\varphi_N)^{\xi}}{\sqrt{N}}+\Psi(z)\right)+C{\frac{(\varphi_N)^{\xi}}{\sqrt{N\eta}}}\Lambda_o\,, 
\end{align*}
with high probability. Since $\im m_{fc}(z)\ge C\eta$, by~\eqref{behaviour of mfc}, we can absorb the term $(\varphi_N)^{\xi}N^{-1/2}$ into the term $\Psi(z)$. Taking the maximum over $i\not=j$,  inequality~\eqref{eq.161} follows. The proofs for $Z_i$ and $\caY_i$ are similar.
\end{proof}

\subsection{Derivation of the weak self-consistent equation}
We now put Equation~\eqref{eq.148} into a form which admits an analysis of the average of the diagonal resolvent entries.
For $n\in\N$, define
\begin{align}\label{eq.171}
R_n(z)\deq\int\frac{\dd \mu(v)}{(\lambda v -z-m_{fc}(z))^n}\,,\quad\quad z\in \caD_L\,,\quad \lambda\in\caD_{\lambda_0}\,.
\end{align}
For any $n$, $R_n$ is bounded uniformly in $z$ and $\lambda$. This follows from the stability bound~\eqref{stability bound}. Note the special case $R_1=\mfc$. Recall the definitions $[\v]=\frac{1}{N}\sum_i G_{ii}-m_{fc}$ and $|\Lambda|=|m-\mfc|$.
\begin{lemma}\emph{[Weak self-consistent equation]}\label{lemma.15}
There is a constant $C$ such that, for all $z\in\caD_L$, $\lambda\in\caD_{\lambda_0}$, we have on $\Omega(z)$ with $(\xi,\nu)$-high probability
\begin{align}\label{eq.172}
 \left|(1-R_2)[\v]-R_3[\v]^2\right|\le C\Psi +C\frac{\Lambda^2}{\log N}\,.
\end{align}
\end{lemma}
\begin{proof}
Since $|\lambda v_i-z-m_{fc}|$ is bounded below by~\eqref{stability bound}, we can expand Equation~\eqref{eq.148} to second order in $([\v]-\caY_i)$,
\begin{align}\label{eq.173}
 \frac{1}{N}\sum_{i=1}^N G_{ii}=&\frac{1}{N}\sum_{i=1}^N\frac{1}{\lambda v_i-z-m_{fc}}+\frac{1}{N}\sum_{i=1}^N\frac{1}{(\lambda v_i-z-m_{fc})^2}([\v]-\caY_i)\nonumber\\ &+\frac{1}{N}\sum_{i=1}^N\frac{1}{(\lambda v_i-z-m_{fc})^3}([\v]-\caY_i)^2+ \caO(\Lambda^3)+\caO(\max_i|\caY_i|^3)\,,
\end{align}
where $ \caY_i=w_{ii}-Z_i-(m^{(i)}-m)$; see Equation~\eqref{eq.149}.

Next, we use the `law of large numbers' to replace the averages in the first two terms on the right side of~\eqref{eq.173} by their expectation:
It follows from the stability bound in~\eqref{stability bound} that the family of functions {$g_i\,:\,\caD_{\lambda_0}\times\caD_L\to\C^+$, $(\lambda,z)\mapsto (\lambda v_i-z-\mfc(z))^{-1}$} are jointly Lipschitz continuous with a constant depending only on $E_0$, $\lambda_0$ and $\mu$. Since the $(v_i)$ are i.i.d.\ random variables, McDiarmid's inequality implies that, for $\lambda\in\caD_{\lambda_0}$, $z\in\caD_L$, $n=1,2,3$, there is a constant $C'$, 
\begin{align}\label{mini lln}
 \left|\frac{1}{N}\sum_{i=1}^N\frac{1}{(\lambda v_i-z-m_{fc})^n}-\int\frac{\dd\mu(v)}{{(\lambda v-z-m_{fc})^n}}\right|\le\ C'\frac{\lambda(\varphi_N)^{\xi}}{\sqrt{N}}\,,
\end{align}
with $(\xi,\nu)$-high probability. Uniformity in $\lambda$, $z$ and $\nu$ can be established by a lattice argument: Choose a lattice $\caL\in\caD_{\lambda_0}\times\caD_L$, with $|\caL|\le CN^4$, such that for any $(\lambda, z)\in\caD_{\lambda_0}\times\caD_L$ there is $(\lambda',z')\in\caL$, with $|z-z'|\le N^{-2}$ and $|\lambda-\lambda'|\le N^{-2}$. Then~\eqref{mini lln} holds for all $(\lambda,z)\in\caL$ for some sufficiently large $C'$ and some sufficiently small $\nu>0$. Using the joint Lipschitz continuity of $(g_i)$, we conclude that there is a constant $C\ge C'$ such that the event
\begin{align} \label{lawlargenumber}
\bigcap_{n=1,2,3}\bigcap_{\substack{(\lambda,z)\in\caD_{\lambda_0}\times\caD_L }}\left\{\left|\frac{1}{N}\sum_{i=1}^N\frac{1}{(\lambda v_i-z-m_{fc})^n}-\int\frac{\dd\mu(v)}{{(\lambda v-z-m_{fc})^n}}\right|\le\ C\frac{\lambda(\varphi_N)^{\xi}}{\sqrt{N}}\right\}\,,
\end{align}
has $(\xi,\nu)$-high probability, for some $\nu>0$, depending on $E_0$, $\lambda_0$ and the distribution $\mu$.

Hence, we obtain from~\eqref{eq.173},
\begin{align*}
\frac{1}{N}\sum_{i=1}^N G_{ii}&=\int\frac{\dd\mu(v)}{\lambda v-z-m_{fc}}+R_2[\v]+R_3[\v]^{2} +\frac{1}{N}\sum_{i=1}^N\frac{1}{(\lambda v_i-z-m_{fc})^2}\caY_i\\ &\quad +\frac{1}{N}\sum_{i=1}^N\frac{1}{(\lambda v_i-z-m_{fc})^3}(\caY_i^2-2[\v]\caY_i)+ \caO(\Lambda^3)+\caO(\max_i|\caY_i|^3)+\caO\left(\frac{\lambda(\varphi_N)^{\xi}}{\sqrt{N}}\right)\,,
\end{align*}
with high probability on $\Omega(z)$, for $z\in \caD_L$ and $\lambda\in\caD_{\lambda_0}$. 
Recalling the functional Equation~\eqref{eq111} for $m_{fc}$, we obtain
\begin{align}\label{eq.149bis}
(1-R_2)[\v]&=R_3[\v]^2+\frac{1}{N}\sum_{i=1}^N\frac{1}{(\lambda v_i-z-m_{fc})^2}\caY_i +\frac{1}{N}\sum_{i=1}^N\frac{1}{(\lambda v_i-z-m_{fc})^3}(\caY_i^2-2[\v]\caY_i)\nonumber\\ &\quad\quad+ \caO(\Lambda^3)+\caO(\max_i|\caY_i|^3)+\caO\left(\frac{\lambda(\varphi_N)^{\xi}}{\sqrt{N}}\right)\,,
\end{align}
with high probability on $\Omega(z)$. Recalling that $|[\v]|=\Lambda$, we obtain
\begin{align*}
|2[\v]\caY_i| \le \left(\frac{\Lambda^2}{\log N}+(\log N)\max_i|\caY_i|^2  \right)\,,
\end{align*}
(the added factor $\log N$ will be useful below). Using the estimates in~\eqref{eq.162} and~\eqref{eq.163}, Equation~\eqref{eq.173} thus becomes
\begin{align*}
 (1-R_2)[\v]&=R_3[\v]^2+\caO\left( \frac{\Lambda^2}{\log N}   \right)+\caO\left(\frac{\lambda(\varphi_N)^{\xi}}{\sqrt{N}}+\Psi       \right)\,,
\end{align*}
which holds with high probability on $\Omega(z)$, $z\in \caD_L$ and $\lambda\in\caD_{\lambda_0}$. Next, observe that, since $\im\mfc(z)\ge C\eta$, we can absorb the third term on the right side of the above equation into the forth term. Finally, we note that we can choose the constants uniform in $z$ and $\lambda$.
\end{proof}
To conclude the proof of Theorems~\ref{thm.weak} we reason as follows. Assume, for simplicity, that $1-R_2(z)$, ($z=E+\ii\eta$), is bounded below (this holds true for $E$ in the bulk of the spectrum). Recalling that $|[\v]|=\Lambda$ and the definition of $\Psi(z)$, we are going to show that~\eqref{eq.172} implies

\begin{align*}
\Lambda\le C\Lambda^2+\caO\left(\frac{(\varphi_N)^{\xi}}{({N\eta})^{1/3}}  \right)\,,
\end{align*}
  with high probability on $\Omega(z)$. Hence, we obtain the following dichotomy: Either
\begin{align}\label{guineapig}
 \Lambda\le C\frac{(\varphi_N)^{\xi}}{(N\eta)^{1/3}}\,,\quad\quad{\mathrm{or}}\quad\quad \Lambda\ge c\,,
\end{align}
for some $N$-independent constant $c>0$, with high probability on $\Omega(z)$, $z\in \caD_L$, $\lambda\in\caD_{\lambda_0}$. Using the self-consistent equation~\eqref{eq.172}, we establish in the next section, that, for large $\eta$, i.e., $\eta\ge 2$, $\Lambda+\Lambda_o\le (\varphi_N)^{-2\xi}$, with high probability. In other words, $\Omega(z)$ holds with high probability, for $\im z\ge 2$. But then the first inequality in~\eqref{guineapig} must hold, for sufficiently large $N$, and we can reject the second inequality in~\eqref{guineapig} for such $\eta$. To extend this conclusion to all $\eta\ge(\varphi_N)^{L}N^{-1}$, we make use of the Lipschitz continuity of the resolvent mapping $z\mapsto G(z)$, which not only allows us to establish that $\Omega(z)$ holds with high probability for $\eta$ small, but also shows that~\eqref{guineapig} holds for small $\eta$. This continuity, or bootstrapping, argument is outlined in Section~\ref{proofofthm.weak}. This argument applies in a straightforward way in the bulk of 
the spectrum where we have $|1-R_2(z)|\ge c>0$. For~$z$ close to the spectral edge, $|1-R_2(z)|$ can become very small and a slightly modified version of the above dichotomy has to be applied (see Lemma~\ref{lemma.18}), but the bootstrapping method still applies.

\subsection{Initial estimates for large $\eta$}
To get the bootstrapping started, we need estimates on $\Lambda_o$ and $\Lambda$, for~$\eta\sim 1$.
\begin{lemma}\label{lemma.16}
 Let $\eta\ge2$. Then, for $z\in \caD_L$, $\lambda\in\caD_L$, we have
\begin{align}\label{eq.181}
 \Lambda_o+\Lambda\le  \frac{(\varphi_N)^{2\xi}}{\sqrt N}\,,
\end{align}
with $(\xi,\nu)$-high probability.
\end{lemma}
\begin{proof}
Let $\lambda\in\caD_L$. We fix $z\in\caD_L$, with $\eta\ge 2$. Then we have the following trivial estimates
\begin{align}\label{eq.182}
 |G_{ij}^{(\T)}|\le\frac{1}{\eta}\,,\quad |m^{(\T)}|\le\frac{1}{\eta}\,,\quad |m_{fc}|\le\frac{1}{\eta}\,,\quad |R_n|\le \left(\frac{1}{\eta}\right)^n\,,	
\end{align}
 for any $\T\subset\{1,\ldots, N\}$. 

We start with estimating $\Lambda_o$: From Equation~\eqref{res.1} we obtain using the large deviation estimates in Lemma~\ref{lemma.LDE}, that
\begin{align}\label{eq.183}
 |G_{ij}|\le C\left(\frac{(\varphi_N)^{\xi}}{\sqrt N}+(\varphi_N)^{\xi}\sqrt{\frac{m^{(ij)}}{N\eta}}\right)\le C\frac{(\varphi_N)^{\xi}}{\sqrt N}\,,
\end{align}
with high probability.

To bound $\Lambda$, we note that 
\begin{align*}
 |\caY_i|\le |Z_i|+|m^{(i)}-m|+|w_{ii}|\le C\frac{(\varphi_N)^{\xi}}{\sqrt N}\,,
\end{align*}
with high probability. The self-consistent equation~\eqref{eq.148} can be written as\small
 \begin{align}\label{eq.184}
[\v]&=\frac{1}{N}\sum_{i=1}^N\left[\frac{1}{\lambda v_i-z-m_{fc}-([\v]-\caY_i)}-\frac{1}{\lambda v_i-z-m_{fc}}\right]+\frac{1}{N}\sum_{i=1}^N\int\dd \mu(v)\left[\frac{1}{\lambda v_i-z-m_{fc}}-\frac{1}{\lambda v-z-m_{fc}}   \right]\,.
\end{align}\normalsize
The second term on the right side of the above equation is bounded by $C\frac{(\varphi_N)^{\xi}}{\sqrt{N}}$ with high probability, as follows from~\eqref{lawlargenumber}. To bound the other term, we rewrite it as
\begin{align*}
\frac{1}{N}\sum_{i=1}^N\frac{([\v]-\caY_i)}{(\lambda v_i-z-m_{fc}-([\v]-\caY_i))(\lambda v_i-z-m_{fc})}\,.
\end{align*}
Taking the imaginary part, we see that the denominators of the summands are with high probability larger in absolute value than
\begin{align*}
 \left(2-1+\caO\left(\frac{(\varphi_N)^{\xi}}{\sqrt{N}}\right)\right)2\ge \frac{3}{2}\,,
\end{align*}
for $\eta\ge2$. Thus, taking the maximum over $i$, we can bound the right side of~\eqref{eq.184} as
\begin{align*}
\Lambda= |[\v]|\le\frac{|[\v]|+\caO(\frac{(\varphi_N)^{\xi}}{\sqrt{N}})}{3/2}+\caO\left(\frac{(\varphi_N)^{\xi}}{\sqrt{N}}\right)\,,
\end{align*}
with high probability. This completes the estimate of $\Lambda$ and hence the proof.	
\end{proof}

\subsection{Proof of Theorem~\ref{thm.weak}}\label{proofofthm.weak}
We introduce the control parameters
\begin{align}\label{eq.191}
 \alpha(z)\equiv\alpha\deq |1-R_2|\,,\quad\quad \beta(z)\equiv\beta\deq\frac{(\varphi_N)^{2\xi/3}}{(N\eta)^{1/3}}\,.
\end{align}
Note that for any $z\in\caD_L$, we have $\beta\ll (\varphi_N)^{-3\xi}$. Also note that we have chosen $\beta$ to be independent of $\lambda$. 

\begin{lemma}\label{lemma.17}
For~$R_2$ and $R_3$, we have the following estimates:
\begin{itemize}
\item[i.]
There exists a constant $K>1$, depending only on $E_0$, $\lambda_0$ and $\mu$, such that, 
\begin{align}\label{eq.192}
 \frac{1}{K}\sqrt{\kappa+\eta}\le\alpha(z)\le K\sqrt{\kappa+\eta}\,,\quad\quad z\in\caD_L\,,\quad \lambda\in\caD_{\lambda_0}\,.
\end{align}
In particular, we have $\im m_{fc}(z)\le C_2 \alpha(z)$, for some $C_2 \ge1$.

\item[ii.]
There exists a constant $C_3$ such that $|R_3(z)| \leq C_3$ uniformly in $z\in \caD_L$ and $\lambda\in \caD_{\lambda_0}$. Moreover, there exist constants $c$ and $\epsilon_0$ such that $|R_3(z)| \geq c$ whenever $z\in \caD_L$ satisfies $|z-L_i| < \epsilon_0$, $i=1,2$.

\end{itemize}
\end{lemma}
The proof of this lemma is stated in the appendix; see Lemma~\ref{lem:stability_bound}.

Next, we fix $E$ and vary $\eta$ from 2 down to $(\varphi_N)^LN^{-1}$. Since $\sqrt{\kappa+\eta}$ is increasing and $\beta(E+\ii \eta)$ is decreasing in $\eta$, we conclude that the equation
\begin{align}\label{eq.193}
 \sqrt{\ka+\eta}=2U^2 K\beta(E+\ii \eta)
\end{align}
has a unique solution $\eta=\tilde{\eta}(U,E)$, for any $U>1$. Note that $\tilde{\eta}(U,E)\ll 1$.

\begin{lemma}\label{lemma.18}
There exists a constant $U_0$ such that, for any fixed $U\ge U_0$, there exists a constant $C_1(U)$, depending only on $U$, such that the following estimates hold for any $z\in \caD_L$:
\begin{align}
 \Lambda(z)&\le U\beta(z)\quad\textrm{or}\quad \Lambda(z)\ge\frac{\alpha(z)}{U}\,,&\textrm{if }\eta\ge\tilde\eta(U,E)\,,\label{dicho1}\\
\Lambda(z)&\le C_1(U)\beta(z)\,, &\textrm{if }\eta<\tilde\eta(U,E)\,,\label{dicho2}
\end{align}
on $\Omega(z)$, with $(\xi,\nu)$-high probability.
\end{lemma}
\begin{proof}
Fix $z\in \caD_L$. Since
\begin{align*}
 \Psi^2=(\varphi_N)^{2\xi}\frac{\Lambda+\im m_{fc}}{N\eta}=\caO(\beta^3\Lambda+\beta^3\alpha)\,,
\end{align*}
we can write the weak self-consistent Equation~\eqref{eq.172} as
\begin{align}\label{eq.194}
(1-R_2)[\v]= R_3 [\v]^2+\caO\left(\frac{\Lambda^2}{\log N}\right)+\caO\left(\frac{\lambda(\varphi_N)^{\xi}}{\sqrt N}+\sqrt{\beta^3\Lambda+\beta^3\alpha}\right)\,.
\end{align}
 Since $\sqrt{{\beta^3\Lambda+\beta^3\alpha}}\le \beta\sqrt{\beta\Lambda}+\beta\sqrt{\alpha\beta}\le C(\beta^2+\beta\alpha+\beta\Lambda)$ by Young's inequality, we obtain from ~\eqref{eq.194}
\begin{align}\label{eq.196}
 \left|(1-R_2)[\v]- R_3 [\v]^2\right|\le \caO\left(\frac{\Lambda^2}{\log N}\right)+ C\adj(\beta\Lambda+\alpha\beta+\beta^2)\,,
\end{align}
with high probability on $\Omega(z)$, for some $C\adj\ge1$. We set $U_0\deq 9(C\adj + C_3 +1)$, where $C_3$ is the constant in Lemma~\ref{lemma.17}. Depending on the size of $\beta$ relative to~$\alpha$, we estimate either $[\v]$ or $[\v]^2$ using the above inequality. We have to consider two cases:

{\it Case 1:} $\eta\ge\tilde\eta(U,E)$ (``Bulk estimate'') From~\eqref{eq.193} we find $\sqrt{\ka+\eta}\ge 2U^2K\beta(z)$ and hence, using~\eqref{eq.192} and the definition of $C\adj$,
\begin{align*}
 \beta\le\frac{\alpha}{2U^2}\le\frac{\alpha}{2C\adj} \le\alpha\,.
\end{align*}
Thus we find from~\eqref{eq.196} with high probability on $\Omega(z)$ that 
\begin{align*}
 \alpha\Lambda\le (|R_3|+1) \Lambda^2+C\adj(\beta\Lambda+\alpha\beta+\beta^2)\le (C_3 +1)\Lambda^2+\frac{\alpha\Lambda}{2}+2C\adj\alpha\beta\,.
\end{align*}
Hence, $\alpha\Lambda\le 2(C_3 +1)\Lambda^2+4C\adj\alpha\beta$. Thus, we either have $\alpha\Lambda/2\le 2(C_3 +1)\Lambda^2$ implying $\Lambda\ge \alpha/[4(C_3 +1)] \ge \alpha/U$ (recall that \mbox{$U\ge U_0=9(C\adj+C_3+1)$}), or $\alpha\Lambda/2\le 4 C\adj\alpha\beta$ implying $\Lambda\le 8C\adj\beta\le U\beta$. This proves~\eqref{dicho1}.

{\it Case 2:} $\eta\le\tilde\eta(U,E)$ (``Edge estimate''). Note that, when $\kappa \sim 1$, the left side of \eqref{eq.193} is of order 1, while the right side $2U^2 K\beta(E+\ii \eta) = o(1)$. (Recall that $\eta \geq (\varphi_N)^N N^{-1}$.) Thus, if $\eta\le\tilde\eta(U,E)$, then $\kappa < \epsilon_0$, where $\epsilon_0$ is the constant in Lemma \ref{lemma.17}. In particular, $|R_3| > c$ in this case.

From~\eqref{eq.192} and~\eqref{eq.193} we find $\alpha\le2U^2K^2\beta$. Thus from~\eqref{eq.196}, we find
\begin{align*}
c\Lambda^2\le2\alpha\Lambda+2C\adj(\beta\Lambda+\alpha\beta+\beta^2)\le C'\beta\Lambda+C'\beta^2\,,
\end{align*}
for some constant $C'$ depending on $U$. Inequality~\eqref{dicho2} follows.
\end{proof}

With Lemmas~\ref{lemma.17} and~\ref{lemma.18} at hand, we are prepared to start the continuity argument: We choose a decreasing sequence $(\eta_k)$, $k=1,\ldots,k_0$ satisfying $k_0\le C N^8$, $|\eta_k-\eta_{k+1}|\le N^{-8}$, $\eta_1=2$ and $\eta_{k_0}=(\varphi_N)^L N^{-1}$. For fixed $E\in[-E_0,E_0]$ we set $z_k=E+\ii\eta_k$. Recall Lemma~\ref{lemma.18}. We fix a $U\ge U_0$ throughout the remainder of this section.

One easily sees that, for large enough $N$,  $\eta_1\ge\eta(U,E)$, for any $E\in[-E_0,E_0]$. Therefore Lemma~\ref{lemma.16} implies that $\Omega(z_1)$ holds with high probability. This is the starting point of the continuity argument. The next lemma extends this result to all $k\le k_0$.
\begin{lemma}\label{lemma.19}
 Define the event
\begin{align}\label{eq.1A1}
 \Omega_k\deq\Omega(z_k)\cap\{\Lambda(z_k)\le C^{(k)}(U)\beta(z_k)\}\,,
\end{align}
where
\begin{align*}
 C^{(k)}(U)\deq\begin{cases} U &\textrm{ if }\eta_k\ge\tilde\eta(U,E)\,,\\
C_1(U) &\textrm{ if }\eta_k<\tilde\eta(U,E)\,.
               \end{cases}
\end{align*}
Then, there exists $\nu>0$, such that for any $\xi$, $1\le k\le k_0$,
\begin{align}\label{eq1.A2}
 \mathbb{P}(\Omega_k^c)\le 3k\, \e{-\nu(\log N)^{\xi}}\,.
\end{align}
\end{lemma}
\noindent Note that the estimates in this lemma are uniform in $\lambda\in\caD_{\lambda_0}$.
\begin{proof}
We proceed by induction on $k$. The case $k=1$ has just been proven. Hence, assume that~\eqref{eq1.A2} holds for some~$k\ge2$. Then
\begin{align*}
 \bbP(\Omega_{k+1}^{c})\le\bbP(\Omega_k\cap\Omega(z_{k+1})\cap\Omega_{k+1}^c)+\bbP(\Omega_k\cap(\Omega(z_{k+1}))^c)+\bbP(\Omega_k^c)=:B+A+\bbP(\Omega_k^c)\,,
\end{align*}
where we set
\begin{align*}
 A&\deq\bbP\big( \{\Omega_k\cap\{\Lambda>(\varphi_N)^{-2\xi}\}\} \cup \{\Omega_k\cap\{\Lambda_o>(\varphi_N)^{-2\xi} \}\} \big)\,,\\[1mm]
B&\deq\bbP\big(\Omega_k\cap\Omega(z_{k+1})\cap\{\Lambda(z_{k+1})>C^{(k+1)}(U)\beta(z_{k+1})\}\big)\,.
\end{align*}
We start by estimating $A$. Using the Lipschitz continuity of the resolvent map $z\mapsto G(z)$, $z\in\C^+$, we obtain
\begin{align*}
 |G_{ij}(z_{k+1})-G_{ij}(z_k)|\le |z_{k+1}-z_{k}|\sup_{z\in \caD_L} |G'_{ij}(z)|\le N^{-8}\sup_{z\in \caD_L}\frac{1}{(\im z)^2}\le N^{-6}\,.
\end{align*}
Thus $\Lambda(z_{k+1})\le \Lambda(z_k)+N^{-6}\le C\beta(z_k)\ll(\varphi_N)^{-2\xi}$ and
\begin{align*}
 \Lambda_o(z_{k+1})\le \Lambda_o(z_k)+N^{-6}\le C\Psi(z_k) \ll(\varphi_N)^{-2\xi}\,,
\end{align*}
with high probability on $\Omega(z_k)$, where we used Lemma~\ref{lemma.14}. Thus $A\le 2\,\e{-\nu(\log N)^{\xi}}$.

To bound $B$, suppose first that $\eta_k\ge \tilde\eta(U,E)$. Then, using the Lipschitz continuity of the resolvent map we find $|\Lambda(z_{k+1})-\Lambda(z_k)|\le N^{-6}$. Thus we find on $\Omega_k$ with high probability
\begin{align*}
 \Lambda(z_{k+1})\le\Lambda(z_{k})+N^{-6}\le U\beta(z_k)+N^{-6}\le \frac{3}{2} U\beta(z_{k+1})\,,
\end{align*}
where we used that $\beta$ is a deterministic decreasing function of $\eta$.   

Suppose next that $\eta_k>\eta_{k+1}\ge\tilde\eta(U,E)$. Then since $\frac{3}{2}U\beta<\alpha U^{-1}$, by Equation~\eqref{eq.193}, we find, in this case, $\Lambda(z_{k+1})<\alpha U^{-1}$. But the dichotomy of Equation~\eqref{dicho1} then implies on $\Omega_k\cap\Omega(z_{k+1})$ with high probability that $\Lambda(z_{k+1})\le U\beta(z_{k+1})$. If $\eta_{k+1}<\tilde\eta(U,E)$, the dichotomy immediately yields $\Lambda(z_{k+1})\le U\beta(z_{k+1})$. This shows that $B\le \e{-(\log N)^{\xi}}$ if $\eta_k\ge\tilde \eta(U,E)$.

If $\eta_k<\tilde\eta(U,E)$, then also $\eta_{k+1}< \tilde\eta(U,E)$ and hence Equation~\eqref{dicho2} gives $\Lambda(z_{k+1})\le C_1(U)\beta(z_{k+1})$. 

Thus, we have proven that, for all $k\le k_0$, $\bbP(\Omega_{k+1}^c)\le 3\e{-\nu(\log N)^{\xi}}+\bbP(\Omega_{k}^c)$. This concludes the proof of the lemma.
\end{proof}
To complete the proof of Theorem~\ref{thm.weak}, we need to extend the conclusion of the previous lemma to all $z\in\caD_L$. To accomplish this we use a simple lattice argument using the regularity of the Green function.

\begin{corollary}\label{cor.12}
 There exists constants $C$ and $\nu>0$, such that, for $\xi$ satisfying~\eqref{eq.xi},
\begin{align}\label{eq.1A3}
 \bbP\left[\bigcup_{\substack{z\in \caD_L \\ \lambda\in\caD_{\lambda_0}}}\Omega(z)^{c}\right]+\bbP\left[\bigcup_{\substack{z\in \caD_L \\ \lambda\in\caD_{\lambda_0}}}\left\{\Lambda(z)>C\beta(z)\right\}\right]\le\e{-\nu(\log N)^{\xi}}\,.
\end{align}

\end{corollary}

\begin{proof}
We choose a lattice $\caL\subset \caD_L$ with $|\caL|\le CN^6$ such that for any $z\in\caD_L$ there is a $z'\in\caL$ satisfying $|z-z'|\le N^{-3}$. Using the regularity of the Green function we have for $z,z'\in \caD_L$,
\begin{align}\label{eq.1A4}
 |G_{ij}(z)-G_{ij}(z')|\le\eta^{-2}|z-z'|\le\frac{1}{N}\,.
\end{align}
Lemma~\ref{lemma.19} yields
\begin{align}\label{eq.1A5}
 \bbP\left[\bigcap_{\substack{z'\in\caL \\ \lambda\in\caD_{\lambda_0}}}\left\{\Lambda(z')\le \frac{C}{2}\beta(z')\right\}\right]\ge 1-\e{-\nu(\log N)^{\xi}}\,,
\end{align}
for some constants $C$ and $\nu$. Hence, combining~\eqref{eq.1A4},~\eqref{eq.1A5} and $N^{-1}\le\beta(z)$, we get
\begin{align*}
 \bbP\left[\bigcup_{\substack{z\in\caD_L\\ \lambda\in\caD_{\lambda_0}}}\left\{\Lambda(z)>C\beta(z)   \right\}\right]\ge 1-\e{-\nu(\log N)^{\xi}}\,.
\end{align*}
The first term of~\eqref{eq.1A3} is estimated in a similar way.
\end{proof}
This proves ~\eqref{eq.weak3} of Theorem~\ref{thm.weak}. To prove~\eqref{eq.weak1}, we observe that~\eqref{eq.161}, ~\eqref{eq.191} and~\eqref{eq.1A3} imply that 
\begin{align*}
 \Lambda_o\le C\frac{{(\varphi_N)^{\xi}}}{\sqrt{N\eta }}\,,
\end{align*}
with high probability on $\Omega(z)$. Then~\eqref{eq.1A3} and a similar lattice argument as above yields~\eqref{eq.weak1}. To prove~\eqref{eq.weak2}, we note that~\eqref{eq.148} yields
\begin{align*}
 |\E^{v_i}G_{ii}-m|=\left|\int\frac{\dd\mu(v)}{\lambda v-z-m_{fc}(z)}-\frac{1}{N}\sum_{i=1}^N\frac{1}{\lambda v_i -z-m_{fc}}\right|+\caO([\v]+\max_i |\caY_i|)\,,
\end{align*}
with $(\xi,\nu)$-high probability on $\Omega(z)$, $z\in \caD_L$, $\lambda\in\caD_{\lambda_0}$. From the large deviation estimate in~\eqref{lawlargenumber} we find
\begin{align*}
 |\E^{v_i}G_{ii}-m|\le C \left(\frac{ \lambda (\varphi_N)^{\xi}}{\sqrt{N}}+\frac{(\varphi_N)^{\xi}}{(N\eta)^{1/3}}\right)\,,
\end{align*}
with high probability on $\Omega(z)$, and we can conclude the proof of~\eqref{eq.weak2} as above. This finishes the proof of Theorem~\ref{thm.weak}.

\subsection{Delocalization of eigenvectors}
Next, we show that the eigenvectors of $H$ are completely delocalized. We denote by $\bsu_{\alpha}$ the normed eigenvector to the eigenvalue $\mu_{\alpha}$ of $H=\lambda V+W$, i.e.,
\begin{align*}
 (\lambda V+W)\bsu_{\alpha}=\mu_{\alpha} \bsu_{\alpha}\,,
\end{align*}
such that $\|\bsu_{\alpha}\|^2_2=\sum_{i}|u_{\alpha}(i)|^2=1$, where $(u_{\alpha}(i))$ are the components of $\bsu_\alpha$.
\begin{proof}[Proof of Lemma~\ref{thm: eigenvector delocalization}]
We follow~\cite{EKYY1}. For $z\in \caD_L$ and $\lambda\in\caD_L$, we have 
\begin{align*}
 |G_{ii}(z)|\le\frac{1}{|\lambda v_i-z-m_{fc}+([\v]-\caY_i)|}\,.
\end{align*}
From the weak deformed semicircle law, Theorems~\ref{thm.weak}, we conclude that $ |[\v]-\caY_i|=o(1)$, with high probability. Since $|v_i-z-m_{fc}(z)|\ge c>0$ is bounded below uniformly in $\lambda\in\caD_{\lambda_0}$ and $z\in\caD_L$, by~\eqref{stability bound}, we have
\begin{align*}
 \max_i |G_{ii}(z)|\le C\,,
\end{align*}
with $(\xi,\nu)$-high probability, uniformly in $z\in \caD_L$ and $\lambda\in\caD_{\lambda_0}$. Set $\eta\deq(\varphi_N)^LN^{-1}$, $L\deq12\xi$. Then, by the spectral decomposition of $H$,
\begin{align*}
 C\ge \im G_{ii}(\mu_{\alpha}+\ii\eta)=\sum_{\beta=1}^N\frac{\eta|u_{\beta}(i)|^2}{(\mu_{\alpha}-\mu_{\beta})^2+\eta^2}\ge\frac{|u_{\alpha}(i)|^2}{\eta}\,,
\end{align*}
with $(\xi,\nu)$-high probability. This concludes the proof.
\end{proof}

\section{Fluctuation Lemma and Strong Deformed Semicircle Law} \label{strong deformed semicircle Law}

In this section, we prove a fluctuation Lemma (see Lemma~\ref{lemma.21} below) that, when combined with the weak local deformed law yields a proof of the strong local deformed law, i.e., Theorem~\ref{thm.strong}. 

Recall that we denote by $\E_i$ the partial expectation with respect to the $i^{\textrm{th}}$-column/row of the matrix~$W$. Set $Q_i\deq \lone-\E_i$. Roughly speaking, the main result of Subsection~\ref{subsection fluctuation lemma} asserts, assuming the conclusions of Theorem~\ref{thm.weak}, that we have 
\begin{align}\label{fluctuation average}
 \frac{1}{N}\sum_{i=1}^NQ_i\left(\frac{1}{G_{ii}}\right)\lesssim\frac{1}{N\eta}\,,
\end{align}
with high probability, up to logarithmic corrections. For a detailed study of fluctuation averages (for generalized Wigner- and band matrices) similar to~\eqref{fluctuation average} we refer to~\cite{EKY}, see also~\cite{EKYY3}, whose arguments we follow. The situation for the deformed ensembles considered here is in so far different as $Q_i(G_{ii})$ is of order $\lambda$, whereas $Q_i(G_{ii})\ll 1$ in the Wigner ensemble. Note, however, that $Q_i (G_{ii}^{-1})\lesssim (N\eta) ^{-1/2}$ for the deformed model studied here as well; see below.

Using the result of Subsection~\ref{subsection fluctuation lemma}, we derive in Subsection~\ref{subsection strong self-consistent equation} a `strong' self-consistent equation for $m-m_{fc}$. In Subsection~\ref{proof of the strong deforemd semicircle law}, we prove, following the arguments of~\cite{EKYY1}, Theorem~\ref{thm.strong}.

\subsection{Fluctuation lemma}\label{subsection fluctuation lemma}
Recall the notation $\Lambda=|m-m_{fc}|$. We set $Q_i\deq\lone-\E_{i}$, where $\E_{i}$ denotes the partial expectation with respect to the $i^{\mathrm{th}}$-column/row of the matrix $W$.
\begin{lemma}\label{lemma.21}
Suppose $\xi$ satisfies~\eqref{eq.xi} and let $L\ge 12\xi$. Let $\Xi$ be an event defined by requiring that the following holds on it: There are constants $C,c>0$ such that,
\begin{itemize}
 \item[$i.$] for all $z\in \caD_L$, $\lambda\in\caD_{\lambda_0}$,
\begin{align}\label{eq.21}
 \Lambda(z)\le \gamma(z)\,,\
\end{align}
where $\gamma$ is a deterministic function satisfying $\gamma(z)\le (\varphi_N)^{-2\xi}$;

\item[$ii.$] for all $z\in \caD_L$, $\lambda\in\caD_{\lambda_0}$,
\begin{align}\label{eq.22}
 \Lambda_o(z)\le C \Psi(z)\le C\Phi(z)\,,
\end{align}
where \[\Phi(z)^2\deq(\varphi_N)^{2\xi}\frac{\im m_{fc}(z)+\gamma(z)}{N\eta}\,\] is a deterministic control parameter;
\item[$iii.$] for all $z\in \caD_L$, $\lambda\in\caD_{\lambda_0}$ and any $i\in\{1,\ldots,N\}$, $c\le |G_{ii}(z)|\le C$ and
\begin{align}\label{eq.24}
\left|Q_i\left(\frac{1}{G_{ii}(z)}\right)\right|\le C \left(\frac{(\varphi_N)^{\xi}}{\sqrt{N}}+\Psi(z)\right)\,\le C\Phi(z)\,.
\end{align}

\end{itemize}
Assume that $\Xi$ holds with $(\xi,\nu)$-high probability, then there exist constants $C,c$, independent of $\lambda$ and $z$, such that, for $p\in\N$, even and satisfying $ p\le \nu(\log N)^{\xi-3/2}$,

\begin{align}\label{eq.25}
  \E\left|\frac{1}{N}\sum_{i=1}^NQ_i\left(\frac{1}{G_{ii}(z)}\right)\right|^{p}\le (Cp)^{5p}\left(\Phi(z)\right)^{2p}\,,
\end{align}
for all $z\in \caD_L$, $\lambda\in\caD_{\lambda_0}$.
\end{lemma}
For the proof of this lemma, we need the following two auxiliary results:

\begin{lemma}\label{lemma.2.2}
 Let the event $\Xi$ be defined as in Lemma~\ref{lemma.21}. Let $\xi$ satisfy~\eqref{eq.xi} and let $L>12\xi$.  Then there exists a constant $C$ such that, for $z\in \caD_L$, $\lambda\in\caD_{\lambda_0}$, the following holds: For any $\T\subset\{1,\ldots,N\}$, with $|\T|\le (\log N)^{\xi-1}$,
\begin{align*}
 \max_{i\not\in \T}|G_{ii}^{(\T)}(z)-G_{ii}(z)|\le C{|\T|}(\Lambda_o(z))^2\,,\qquad\quad \max_{\substack{i\not=j\\ i,j\not\in\T}}|G_{ij}^{(\T)}(z)|\le C \Lambda_o(z)\,,
\end{align*}
on $\Xi$. In particular, we have that $|G_{ii}^{(\T)}(z)|\ge c$, for some $c>0$, uniformly in $\T$ and $z\in \caD_L$, $\lambda\in\caD_{\lambda_0}$.
\end{lemma}

\begin{proof}
{ For simplicity we drop the $z$-dependence from the notation. For $l\in \N$, we set 
\begin{align*}
\Gamma_l \deq \max\left\{\left|G_{ij}^{(\T')}\right|\,:\,i,j\not\in\T'\,, i \neq j \,, |\T'|=l  \right\}, \quad \quad\widetilde\Gamma_l \deq \max\left\{\left|G_{ii}^{(\T')}- G_{ii}\right|\,:\,i\not\in\T'\,, |\T'|=l  \right\}.
\end{align*}
Equation~\eqref{res.2}, i.e., $ G_{ij}=G_{ij}^{(k)}+{G_{ik}G_{kj}}/{G_{kk}}$, implies that we have on $\Xi$
\begin{align*}
\Gamma_1 \le \Lambda_o +C\Lambda_o^2\ll (\varphi_N)^{-2\xi}\,, \quad \quad\widetilde\Gamma_1\le C\Lambda_o^2 \leq C(\Gamma_1)^2 \leq \Gamma_1 \ll (\varphi_N)^{-2\xi}\,.
\end{align*}
In particular, we have on $\Xi$ that $|G_{ii}^{(k)}|\ge |G_{ii}|- \widetilde\Gamma_1 \geq |G_{ii}|- 2\Gamma_1>0$, for any $k\not=i$ and $z\in \caD_L$, $\lambda\in\caD_{\lambda_0}$. Assume that there is a constant $C_0$ such that $|G_{ii}^{(\T')}|\ge |G_{ii}|- 2\Gamma_1\ge C_0^{-1}$ for any $\T'$ with $|\T'|\le l$,  $i\not\in \T'$, and $z\in \caD_L$, $\lambda\in\caD_{\lambda_0}$. Then Equation~\eqref{res.2} implies
\begin{align*}
 \Gamma_{l+1}\le \Gamma_l+C_0 \Gamma_l^2\,, \quad \quad\widetilde\Gamma_{l+1}\le \widetilde\Gamma_l+C_0 \Gamma_l^2\,,
\end{align*}
thence
\begin{align*}
 \Gamma_{l+1}\le \Gamma_1 + C_0\sum_{n=1}^l \Gamma_n^2 \,, \quad\quad \widetilde\Gamma_{l+1}\le \widetilde\Gamma_1 + C_0\sum_{n=1}^l \Gamma_n^2 \,.
\end{align*}
Thus, as long as $C_0l\Gamma_1\le 1/4$, we obtain by induction that 
$$
\Gamma_{l+1}\le 2\Gamma_1, \quad\quad\widetilde\Gamma_{l+1} \leq \widetilde\Gamma_1 + 4 C_0 l (\Gamma_1)^2 \leq 2\Gamma_1 \,,
$$
and $|G_{ii}^{(\T')}|\ge C_0^{-1}$, for any $i\not\in \T'$, $|\T'|=l+1$, $l\le (\log N)^{\xi-1}$. By induction on $l$, this proves the desired lemma.
}
\end{proof}

\begin{lemma}\label{lemma.2.3}
Let the event $\Xi$ be defined as in Lemma~\ref{lemma.21}. Let $\xi$ satisfy~\eqref{eq.xi} and let \mbox{$L\ge 12\xi$}. Assume that $\Xi$ has $(\xi,\nu)$-high probability. Then there is a constant $C$ such that for any $p,l\in\N$, with \mbox{$p,l\le (\log N)^{\xi-3/2}$}, and for any $z\in \caD_L$, $\lambda\in\caD_{\lambda_0}$, we have
\begin{align}\label{eq.26}
 \E\left|\frac{1}{G_{ii}^{(\T)}(z)}\right|^p\le C^p\,,
\end{align}
where $\T\subset\{1,\ldots,N\}$, with $|\T|\le l$, and $i\not\in\T$.
\end{lemma}
\begin{proof}
For simplicity we drop the $z$-dependence from our notation. By Lemma~\ref{lemma.2.2} we have $|G_{ii}^{(\T)}| \ge c$ on $\Xi$, for any $\T\not\ni i$ with $|\T|\le (\log N)^{\xi-1}$. On the complementary event~$\Xi^c$, we use Schur's complement formula~\eqref{schur},
\begin{align*}
 \frac{1}{G_{ii}^{(\T)}}=\lambda v_i+w_{ii}-z-\sum_{k,l}^{(i\T)}w_{ik}G^{(i\T)}_{kl}w_{li}\,,\quad\quad i\not\in\T\,.
\end{align*}
Then by Cauchy-Schwarz, the trivial bounds $|G_{ii}^{(\T)}|\le\eta^{-1}\le N$, $\E |h_{ij}|^{p}\le N^p$ and $\E |\la v_i|^p\le \lambda_0^p$, and the boundedness of $\caD_L$, we find
\begin{align*}
 \E\left|\frac{1}{G_{ii}^{(\T)}}\right|^p\lone(\Xi^c)\le \left[ \E\left|\frac{1}{G_{ii}^{(\T)}}\right|^{2p}\lone(\Xi^c) \right]^{1/2}\,\mathbb{P}(\Xi^c)^{1/2}\le (C+CN+CN^3)^p\mathbb{P}(\Xi^c)^{1/2}\le C^p\,,
\end{align*}
where we used that $\Xi$ has $(\xi,\nu)$-high probability and that $p\le (\log N)^{\xi-3/2}$.
\end{proof}

\begin{proof}[Proof of Lemma~\ref{lemma.21}]
For simplicity we drop the $z$-dependence from our notation.  We illustrate the idea of the proof for the simple case $p=2$:
\begin{align}\label{eq.27}
\E\left|\frac{1}{N}\sum_{i=1}^NQ_i\left(\frac{1}{G_{ii}}\right)\right|^{2}=\frac{1}{N^2}\sum_{i=1}^N \E\left|Q_i\left(\frac{1}{G_{ii}}\right)\right|^2+\frac{1}{N^2}\sum_{i\not=j}\E Q_i\ol{\left(\frac{1}{{G}_{ii}}\right)}Q_j\left(\frac{1}{G_{jj}}\right)\,.
\end{align}
The first term on the right side is bounded by
\begin{align}\label{eq.27bis}
\frac{1}{N^2}\sum_{i=1}^N \E\left|Q_i\left(\frac{1}{G_{ii}}\right)\right|^2\lone(\Xi)+\frac{1}{N^2}\sum_{i=1}^N \E\left|Q_i\left(\frac{1}{G_{ii}}\right)\right|^2\lone(\Xi^c)&\le\frac{C}{N}\Phi^2+\frac{o(1)}{N^2}\le
C\Phi^4\,,
\end{align}
where we used that $\Xi$ has $(\xi,\nu)$-high probability and that $N^{-1/2}\le C\Phi(z)$, since $\im\mfc(z)\ge C\eta$, $z\in\caD_L$.

To handle the second term on the right side of~\eqref{eq.27}, we use Equation~\eqref{res.2} to write
\begin{align}\label{eq.28}
 Q_j\left(\frac{1}{G_{jj}}\right)=Q_j\left(\frac{1}{G_{jj}^{(i)}}-\frac{G_{ij}G_{ji}}{G_{jj}G_{jj}^{(i)}G_{ii}}\right)\,,
\end{align}
 for $i\not=j$. Hence,
\begin{align*}
\E \,Q_i{\ol{\left(\frac{1}{G_{ii}}\right)}}Q_j\left(\frac{1}{G_{jj}}\right) &=\E\,{Q_i\ol{\left(\frac{1}{G_{ii}}\right)}Q_j\left(\frac{1}{G_{jj}^{(i)}}-\frac{G_{ij}G_{ji}}{G_{jj}G_{jj}^{(i)}G_{ii}} \right)}=-\E\,{Q_i\ol{\left(\frac{1}{G_{ii}}\right)}Q_j\frac{G_{ij}G_{ji}}{G_{jj}G_{jj}^{(i)}G_{ii}} }\,,
\end{align*}
where we used that $G_{jj}^{(i)}$ is independent of the entries in the $i^{\mathrm{th}}$-column/row of $W$, and that, for general random variables $A=A(W)$ and $B=B(W)$, $\E[ (Q_iA) B]=\E[B \E_iQ_i A]=0$ if $B$ is independent of the variables in the $i^{\mathrm{th}}$-column/row of $W$. Using Equation~\eqref{eq.28} once more and applying the same reasoning we obtain
\begin{align*}
\E\, Q_i\ol{\left(\frac{1}{G_{ii}}\right)}Q_j\left(\frac{G_{ij}G_{ji}}{G_{jj}G_{jj}^{(i)}G_{ii}}\right)&=\E \,Q_i\ol{\left(\frac{G_{ji}G_{ij}}{G_{ii}G_{ii}^{(j)}G_{jj}}\right)}Q_j\left(\frac{G_{ij}G_{ji}}{G_{jj}G_{jj}^{(i)}G_{ii}}\right)\,.
\end{align*}
Hence,
 \begin{align}\label{eq.29}
 \left|\frac{1}{N^2}\sum_{i\not=j}\E\, Q_i\ol{\left(\frac{1}{G_{ii}}\right)}Q_{j}\left(\frac{1}{G_{jj}}\right)\right|&\le\sup_{i\not=j}\E\left|Q_i\ol{\left(\frac{G_{ji}G_{ij}}{G_{ii}G_{ii}^{(j)}G_{jj}}\right)}Q_j\left(\frac{G_{ij}G_{ji}}{{G_{jj}G_{jj}^{(i)}G_{ii}}}\right)\right|\,.
 \end{align}
Using that, for a general random variable $A=A(W)$, $q\in\N$, 
\begin{align}\label{for general rv}
 \E|Q_{i}A|^q\le 2^{q-1}(\E|A|^q+\E|\E_i A|^q)\le 2^q\E|A|^q\,,
\end{align}
where we used Jensen's inequality for partial expectations, we obtain
\begin{align}\label{nouveau1}
 \left|\frac{1}{N^2}\sum_{i\not=j}\E\, Q_i\ol{\left(\frac{1}{G_{ii}}\right)}Q_{j}\left(\frac{1}{G_{jj}}\right)\right|&\le\sup_{i\not=j}C\left(\E\left|\frac{G_{ji}G_{ij}}{G_{ii}G_{ii}^{(j)}G_{jj}} \right|^2 \right)^{1/2}\left(\E\left|\frac{G_{ij}G_{ji}}{{G_{jj}G_{jj}^{(i)}G_{ii}}}\right|^2\right)^{1/2}\,.
\end{align}
Using that $|G_{ij}^{(\T)}|\le C\Phi$, ($i\not=j$), $|G_{ii}^{(\T)}|>c$, on $\Xi$, we have, for $i\not=j$,
\begin{align}\label{nouveau2}
 \E\left|\frac{G_{ji}G_{ij}}{G_{ii}G_{ii}^{(j)}G_{jj}} \right|^2 \lone(\Xi)\le C\Phi^4\,.
\end{align}
Using Lemma~\ref{lemma.2.3} and $|G_{ij}|\le \eta^{-1}\le N$, H\"older's inequality yields
\begin{align}\label{nouveau3}
 \E\left|\frac{G_{ji}G_{ij}}{G_{ii}G_{ii}^{(j)}G_{jj}} \right|^2 \lone(\Xi^c)\le C \,\mathbb{P}(\Xi^c)^{1/2} N^{4}\le C\Phi^4\,,
\end{align}
where we used that $\Xi$ has $(\xi,\nu)$-high probability, and the fact that $N^{-1/2}\le C\Phi(z)$, $z\in\caD_L$.

Combining the estimates~\eqref{eq.27bis}, \eqref{nouveau1},~\eqref{nouveau2} and~\eqref{nouveau3}, Inequality~\eqref{eq.25} follows for $p=2$.

Next, let $4\le p\le \nu(\log N)^{\xi-3/2}$ be even. Writing $p=2r$, we have
\begin{align}\label{eq.211}
\E\left|\frac{1}{N}\sum_{i=1}^N Q_i\left(\frac{1}{G_{ii}}\right)\right|^{2r}=\frac{1}{N^{2r}}\sum_{i_1,\ldots, i_{2r}}\E \prod_{j=1}^r{Q_{i_j}\ol{\left(\frac{1}{G_{i_ji_j}}\right)}}\prod_{j'=r+1}^{2r} Q_{i_{j'}}\left(\frac{1}{G_{i_{j'}i_{j'}}}\right)\,.
\end{align}
For simplicity, we first assume that we can replace the sum over the indices $\underline{i}\equiv(i_1,\ldots, i_{2r})$ by a truncated sum, where all indices are distinct, i.e., we consider
\begin{align}\label{eq.212-1}
 \frac{1}{N^{2r}}\sum_{\substack{i_1,\ldots, i_{2r} \\ \textrm{all distinct}}}\E \prod_{k=1}^rQ_{i_k}\ol{\left(\frac{1}{G_{i_ki_k}}\right)}\prod_{k'=r+1}^{2r} Q_{i_{k'}}\left(\frac{1}{G_{i_{k'}i_{k'}}}\right)\,.
\end{align}

As in the $p=2$ case, we make each factor of ${G_{ii}}$ in the above expression independent as of many summation indices as possible by an expansion procedure that uses the identities
\begin{align}\label{expand1}
G_{ij}^{(\T)}=G_{ij}^{(\T k)}+\frac{G_{ik}^{(\T)}G_{kj}^{(\T)}}{G_{kk}^{(\T)}}\,,
\end{align}
for $i,j,k\not\in\T$, $k\not=i,j$, and
\begin{align}\label{expand2}
 \frac{1}{G_{ii}^{(\T)}}=\frac{1}{G_{ii}^{(\T k)}}-\frac{G_{ik}^{(\T)}G_{ki}^{(\T)}}{G_{ii}^{(\T)}G_{ii}^{(\T k)}G_{kk}^{(\T)}}\,,
\end{align}
for $k\not\in\T$, $k\not=i$. 

The expansion procedure goes as follows: We start with expanding $ F_{i_1}\deq (G_{i_1i_1})^{-1}$ in~ \eqref{eq.212-1}. Using formula~\eqref{expand2}, where the choice of $k\in\{i_1,\ldots,i_{2r}\}\backslash\{{i_1}\}$ is immaterial, we can add to $(G_{i_1i_1})^{-1}$ one upper index $k$. This results in two terms, $(F_{i_1})_1\deq({G_{i_1i_1}^{(k)}})^{-1}$ and $(F_{i_1})_{0}\deq-G_{i_1k}G_{ki_1}/{G_{i_1i_1}G_{i_1i_1}^{(k)}G_{kk}}$. Using formula~\eqref{expand2} we can further expand $(F_{i_1})_1$ as $(F_{i_1})_{11}+(F_{i_1})_{10}$, where $(F_{i_1})_{11}=(G_{i_1i_1}^{(kl)})^{-1}$, for $l\in\{i_1,\ldots,i_{2r}\}\backslash\{i_1,k\}$ (again the choice of $l$ is immaterial), and $(F_{i_1})_{10}$ is a fraction with two off-diagonal resolvent entries in the numerator and three diagonal resolvent entries in the denominator. Similarly, we can split the term $(F_{i_1})_{0}=(F_{i_1})_{00}+(F_{i_1})_{01}$, where we applied~\eqref{expand1} or~\eqref{expand2} to one resolvent entry of $(F_{i_1})_0$, with an index $l\not=i_
1,k$. There is some arbitrariness in the choice of the resolvent entry used for the splitting that can, if desirable, be removed by choosing an ordering on the set of all resolvent entries $G_{ij}^{(\T)}$. We continue the splitting of the terms $(F_{i_1})_{\sigma}$, hereby generating terms indexed by sequences $\sigma$ of zeros and ones. 

The precise procedure is the following. Let $\caG$ denotes the set of monomials of resolvent entries of the form $G_{nm}^{(\T)}$, with $n\not=m$, $\T\subset\{i_1,\ldots,i_{2r}\}\backslash\{n,m\}$, and $1/G_{nn}^{(\T )}$, $\T\subset\{i_1,\ldots,i_{2r}\}\backslash\{n\}$. Given $F\in\caG$, the formulas~\eqref{expand1} and~\eqref{expand2} define an operation, $F\mapsto F_1\in\caG$, by adding an upper index, e.g., $G_{nm}^{(\T)}\mapsto G_{nm}^{(\T k)}$, and its complementary operation $F\mapsto F_0$, e.g., $G_{nm}^{(\T)}\mapsto {G_{nk}^{(\T)}G_{km}^{(\T)}}/{G_{kk}^{(\T)}}$, such that $F=(F)_0+(F)_1$. Composing these operations we generate from $F\equiv (F)_{\emptyset}\in\caG$, elements $(F)_\sigma\in\caG$, labeled by binary sequences $\sigma$. For these operations we use the notation $\sigma\mapsto\sigma0$ and $\sigma\mapsto\sigma1$. Given $F\equiv (F)_{\emptyset}\in\caG$, the recursive algorithm is as follows:
\begin{itemize}
 \item [$(A)$] {\it Stopping rules}
\begin{itemize}
\item[$(1)$] 
If all terms in $(F)_\sigma$ are maximally expanded, i.e., each resolvent entry in $(F)_\sigma$ is of the form
$G_{nm}^{(\T)}$ with $n,m\not\in\T$, $(\T nm)=\{i_1,\ldots,i_{2r}\}$;

\item[$(2)$] else if $(F)_\sigma$ contains at least $2p$ off-diagonal resolvent entries in the numerator;
\end{itemize}
we stop the expansion. 
\item[$(B)$] Else, we choose an arbitrary resolvent entry $G_{nm}^{(\T)}$ in $(F)_\sigma$. If $n=m$, we use~\eqref{expand2}, with some arbitrary $k\in\{i_1,\ldots,i_{2r}\}\backslash\{ (\T n)\}$, to split $(F)_\sigma=(F)_{\sigma0}+(F)_{\sigma1}$. If $n\not= m$, we use~\eqref{expand1}, with some arbitrary $k\in\{i_1,\ldots,i_{2r}\}\backslash\{(\T nm)\}$, to split  $(F)_\sigma=(F)_{\sigma0}+(F)_{\sigma1}$.
\end{itemize}
Below, we show that the stopping rules ensure that the recursive procedure is terminated after a finite number of steps. Choosing $F=(F_{i_1})=(G_{i_1i_1})^{-1}$, the above procedure yields
\begin{align}\label{eq.220}
  Q_{i_1}\ol{\left(\frac{1}{G_{i_1i_1}}\right)}Q_{i_2}\ol{\left(\frac{1}{G_{i_2i_2}}\right)}\ldots Q_{i_p}\left(\frac{1}{G_{i_pi_p}}\right)=\sum_{\sigma}Q_{i_1}\ol{(F_{i_1})}_\sigma Q_{i_2}\ol{\left(\frac{1}{{G_{i_2i_2}}}\right)}\ldots Q_{i_p}\left(\frac{1}{G_{i_pi_p}}\right)+R_{i_1}\,,
\end{align}
where the summation index $\sigma$ runs over a set of finite binary sequences (the number of terms in the sum is estimated below). The summands $(F_{i_1})_\sigma\in\caG$ are fractions with off-diagonal entries of $G$ in the numerator (except for the maximally expanded leading term $(G_{i_1i_1}^{(\T)})^{-1}$) and diagonal resolvent entries in the denominator. All these entries are maximally expanded in the summation indices. Each term in the rest term $R_{i_1}$, a fraction of resolvent entries, contains at least $2p$ off-diagonal resolvent entries in the numerator. 

We claim that the total number of terms generated by the above recursive procedure is bounded by $(Cp)^{2p}$, for some $p$-independent constant $C$. Indeed, the procedure described above generates a finite rooted binary tree, whose vertices are labeled by binary sequences $\sigma$. By the stopping rules $(1)$ and $(2)$, each term on the right side of~\eqref{eq.220} corresponds to a leaf node of this tree. Thus to get an upper bound on the number of terms in~\eqref{eq.220}, it is enough to estimate the depth of this tree. 

To estimate the depth of the tree, we estimate the maximal length of a generated sequences $\sigma$. We first observe that the number of off-diagonal resolvent entries is raised by one or two under the operation $\sigma\mapsto\sigma0$ (in case it is first applied to $F_{i_1}$, the number is raised by two). Hence, by stopping rule $(2)$, the leaf nodes are labeled by sequences $\sigma$ with less than $2p$ zeros in it. Also note that the operation $\sigma\mapsto\sigma0$ increases the number of resolvent entries by at most $4$, but the operation $\sigma\mapsto\sigma1$ does not change this number. Thus the total number of resolvent entries in a term $(F_{i_1})_\sigma$ is bounded by $8p+1$. Hence, a bound on the number of upper indices for a vertex is $(8p+1)p$. In other words, a sequence labeling a vertex has at most $(8p+1)p$ ones in it. Thus a sequence labeling a leaf node has a most $(8p+1)p$ ones and $2p$ zeros, therefore has length at most~$8p^2+3p$ and we conclude that the number of leaf nodes of the 
tree is bounded by
\begin{align*}
 \sum_{q=0}^{2p}\binom{8p^2+3p}{ q} \leq (2p+1) \frac{(11p^2)^{2p+1}}{(2p)!} \leq (Cp)^{4p} (2p)^{-2p} \le (Cp)^{2p}\,,
\end{align*}
for some constant $C$, independent of $p$. 

It follows that the right side of~\eqref{eq.220} contains at most $(Cp)^{2p}$ terms. In particular, the remainder $R_{i_1}$ contains at most $(Cp)^{2p}$ terms, each of which contains at least $2p$ off-diagonal matrix resolvent entries and less than $3p$ diagonal resolvent entries. By assumptions $ii$, Inequality~\eqref{for general rv}, Lemma~\ref{lemma.2.2} and Lemma~\ref{lemma.2.3}, the rest term $R_{i_1}$ satisfies
\begin{align*}
 \E|R_{i_1}|\le (Cp)^{2p}\Phi(z)^{2p}\,,
\end{align*}
for some sufficiently large $C$.

Next, we expand the term term $(G_{i_2i_2})^{-1}$ in~\eqref{eq.220}. We apply the same procedure to each `leaf node term'
\begin{align*}
 Q_{i_1}\ol{(F_{i_1})}_{\sigma}Q_{i_2}\ol{\left({\frac{1}{G_{i_2i_2}}}\right)}\ldots Q_{i_{p}}\left(\frac{1}{G_{i_{p}i_p}}\right)
\end{align*}
in~\eqref{eq.220}. Note that we do not expand the remainder term $R_{i_1}$ any further nor start a new expansion separately for $(G_{i_2i_2})^{-1}$ (this would yield an expansion with too many terms for our purposes). We also modify the stopping rule $(2)$ accordingly: We stop expanding a term in~\eqref{eq.220} whenever it contains at least $2p$ off-diagonal resolvent entries. Applying the algorithm $(A)$-$(B)$ to ~\eqref{eq.220} we find
\begin{align}\label{expand3}
 Q_{i_1}\ol{\left(\frac{1}{G_{i_1i_1}}\right)}Q_{i_2}\ol{\left(\frac{1}{G_{i_2i_2}}\right)}\ldots Q_{i_{2r}}\left(\frac{1}{G_{i_{2r}i_{2r}}}\right)=\sum_{\sigma_1,\sigma_2}Q_{i_1}\ol{(F_{i_1})}_{\sigma_1}Q_{i_2}{\ol{(F_{i_2})}_{\sigma_2}}\ldots Q_{i_{2r}}\left(\frac{1}{G_{i_{2r}i_{2r}}}\right)+R_{i_1}+R_{i_2}\,,
\end{align}
 where the remainder $R_{i_2}$ satisfies the same bound as $R_{i_1}$. The effect of the modified stopping rule $(2)$ is that the sequences $\sigma_1$ and $\sigma_2$ together contain in total at most $2p-1$ zeros.

 Expanding the remaining $2r-2$ factors of $(G_{ii})^{-1}$ in~\eqref{expand3}, we find
\begin{align}\label{eq.212}
\E \prod_{k=1}^rQ_{i_k}\ol{\left(\frac{1}{G_{i_ki_k}}\right)}\prod_{k'=r+1}^{2r} Q_{i_{k'}}\left(\frac{1}{G_{i_{k'}i_{k'}}}\right)=\sum_{\sigma_1,\ldots,\sigma_{2r}}\E\left[Q_{i_1}\ol{(F_{i_1})}_{\sigma_1}\cdots Q_{i_{2r}}{(F_{i_{2r}})}_{\sigma_{2r}}\right]+\E\mathcal{R}\,,
\end{align}
where the remainder $\mathcal{R}=\sum_{q=1}^p R_{i_q}$ satisfies
\begin{align}\label{bound on remainder}
 \E|\mathcal{R}|\le (Cp)^{2p}\Phi^{2p}\,,
\end{align}
for some sufficiently large $C$. It therefore suffices to consider only the first term on the right side of~\eqref{eq.212}, in which all monomials $(F_{i_k})_{\sigma_k}$ are maximally expanded and the summation runs over $2r$ binary sequences of finite length. Note that the total number of zeros in the array of sequences $\underline\sigma=(\sigma_1,\ldots,\sigma_{2r})$ is, by the modified stopping rule~$(2)$, at most~$2p-1$. It follows that the total number of terms in~\eqref{expand3} is less than $(Cp)^{3p}$. Indeed, this can be checked in the same way as is done above: A term in \eqref{expand3} corresponds to a leaf node on a rooted binary tree, whose vertices are labeled by $\underline{\sigma}$. The total number of zeros in~$\underline{\sigma}$ indexing a leaf node is bounded by~$2p$ and the number of ones is less than $(8p+1)p^2$. It follows that the total number of terms in the expansion of~\eqref{expand3} is bounded by $(Cp)^{3p}$ and we find
\begin{align}\label{eq.240}
 \left|\sum_{\sigma_1,\ldots,\sigma_{2r}}\E Q_{i_1}\ol{(F_{i_1})}_{\sigma_1}\cdots Q_{i_{2r}}(F_{i_{2r}})_{\sigma_{2r}}\right|\le (Cp)^{3p} \Phi^{p}\,.
\end{align}
Recall that, due to our simplification assumption all indices $(i_1,\ldots,i_p)$ are distinct. As in the case $p=2$ we now use the presence of the $Q$'s: First, we claim that, for any label $a\in\{1,\ldots,2r\}$,
\begin{align}\label{zeros in F}
 \left|(F_{i_a})_{\sigma_a}\right|\lone(\Xi)\le (C\Phi)^{1+\mathbf{0}(\sigma_a)}\,,
\end{align}
where $\mathbf{0}(\sigma_a)$ denotes the number of zeros in the sequence $\sigma_a$. For $\mathbf{0}(\sigma_a)=0$, this follows from hypothesis~$iii$. If $\mathbf{0}(\sigma_a)\ge 1$, the successive application of the operation $\sigma\mapsto\sigma0$ has generated at least $\mathbf{0}(\sigma_a)+1$ off-diagonal resolvent entries and at most $3\, \mathbf{0}(\sigma_a)+1$ diagonal resolvent entries.

Next, choose $(i_1,\ldots,i_{2r})$ and $(\sigma_1,\ldots,\sigma_{2r})$ in~\eqref{eq.240} such that 
\begin{align}\label{Q-term}
 \E\, Q_{i_1}{\ol{(F_{i_1})}_{\sigma_1}}\cdots Q_{i_{2r}}(F_{i_{2r}})_{\sigma_{2r}} \not=0\,.
\end{align}
 The key observation is the following:
\begin{itemize}
 \item[$(C)$]
Let $a\in\{1,\ldots,2r\}$, then there is a label $b\in\{1,\ldots,{2r}\}\backslash\{a\}$, such that the monomial $(F_{i_b})_{\sigma_{b}}$ contains an off-diagonal resolvent entry with $i_a$ as a lower index. We use the notation $b=\mathbf{l}(a)$, if $b$ is linked to $a$ in this sense.
\end{itemize}

  Indeed, assuming the contrary, we conclude that all monomials $(F_{i_c})_{\sigma_c}$ in~\eqref{Q-term}, but $(F_{i_a})_{\sigma_a}$, are independent of the random variables indexed by $i_a$. But due to the presence of the $Q_{i_a}$ this term has vanishing expectation. Note that this argument relies on the assumptions that all indices $(i_1,\ldots,i_{2r})$ are distinct.

Next, let $a\in\{1,\ldots,{2r}\}$ and denote by $\mathbf{l}_a\deq|\mathbf{l}^{-1}(\{a\})|$, the number of times the label $a$ is linked to some label $b$ in the sense of $(C)$. Then, 
\begin{align}\label{equation with |a|}
|(F_{i_a})_{\sigma_a}|\lone(\Xi)\le C^p\Phi^{1+\mathbf{l}_a}\,.
\end{align}
Indeed, for each label $c\in \mathbf{l}^{-1}(\{a\})$ we had to use at least once the operation $\sigma\mapsto\sigma0$ to get the lower index $i_a$. Hence, $\mathbf{0}(\sigma_{b})$, the number of zeros in $\sigma_{b}$, is at least $\mathbf{l}_a$. Inequality~\eqref{equation with |a|} follows from~\eqref{zeros in F}. Finally, noting that $\sum_{a}\mathbf{l}_a\ge p$ by $(C)$, we find that, for terms as in~\eqref{Q-term},
\begin{align}
 \left|\E\, Q_{i_1}{\ol{(F_{i_1})}_{\sigma_1}}\cdots Q_{i_{2r}}(F_{i_{2r}})_{\sigma_{2r}}\right|\le (C\Phi)^{2p}\,.
\end{align}

Combination with the bound ~\eqref{bound on remainder} on the remainder term, the estimate on the number of terms in~\eqref{eq.212}, we thus obtain
\begin{align*}
\left| \frac{1}{N^{2r}}\sum_{\substack{i_1,\ldots, i_{2r} \\ \textrm{all distinct}}}\E \prod_{k=1}^rQ_{i_k}\ol{\left(\frac{1}{G_{i_ki_k}}\right)}\prod_{k'=r+1}^{2r} Q_{i_{k'}}\left(\frac{1}{G_{i_{k'}i_{k'}}}\right)\right|\le (Cp)^{cp}\Phi^{2p}\,,
\end{align*}
for any $p\le \nu(\log N)^{\xi-3/2}$, under the simplifying assumption that all indeces are distinct in the sum.

To deal with the general case, we go back to~\eqref{eq.211}. Abbreviate $\underline{i}=(i_1,\ldots,i_{2r})$. Denote by $\caP_{2r}$ the set of partitions of $\{1,\ldots,2r\}$. Let $\Gamma(\underline{i})$ be the element of $\caP_{2r}$ defined by the equivalence relation $a\sim b$, if and only if $i_a=i_b$. Then we can write
\begin{align}\label{without restrictions}
 \E\left|\frac{1}{N}\sum_{i=1}^N  Q_i\left(\frac{1}{G_{ii}}\right)\right|^{2r}=\frac{1}{N^{2r}}\sum_{\Gamma\in\caP_{2r}}\sum_{i_1,\ldots,i_{2r}}\lone(\Gamma=\Gamma(\underline{i}))\,\E 	\,Q_{i_1}\ol{\left(\frac{1}{G_{i_1i_1}}\right)}\ldots Q_{i_{2r}}\left(\frac{1}{G_{i_ri_r}}\right)\,.
\end{align}
Fix now $\underline{i}$, and denote by $\Gamma\deq\Gamma(\underline{i})$, the partition induced by the equivalence relation $\sim$. For a label $a\in\{1,\ldots,2r\}$, we denote by $[a]$ the block of $a$ in $\Gamma$. Let $S(\Gamma)\deq\{ a\,:\ |[a]|=1\}\subset\{1,\ldots,2r\}$ denote the set of single labels and abbreviate by $s\deq|S(\Gamma)|$ its cardinality. We denote by $\underline{i}_{S(\Gamma)}\deq (i_a)_{a\in S}$, the summation indices associated with single labels. Notice that if $a$ is a single label (for some $\Gamma$), then there is exactly one $Q_{i_a}$ on the right side of~\eqref{without restrictions}. However, if $a$ is not a single label (for some $\Gamma$), $Q_{i_a}$ appears more than once on the right side of~\eqref{without restrictions}. 

Next, we expand the summands on the right side of~\eqref{without restrictions}, using the recursive procedure $(A)$-$(B)$, but we only expand in the single labels. More precisely, the recursive procedure is now defined as follows:
\begin{itemize}
 \item [$(A)$] {\it Stopping rules}
\begin{itemize}
\item[$(1')$] If all terms in $(F)_\sigma$ are maximally expanded in the single labels; a resolvent entry $G_{nm}^{(\T)}$, is maximally expanded in the single labels if $\underline{i}_S\subset (\T nm)$, $n,m\not\in\T$;
\item[$(2)$] else if $(F)_\sigma$ contains at least $2p$ off-diagonal resolvent entries in the numerator;

\end{itemize}
we stop the expansion. 
\item[$(B')$] Else, we choose an arbitrary resolvent entry $G_{nm}^{(\T)}$ in $(F)_\sigma$. If $n=m$, we use~\eqref{expand1}, with some arbitrary index $k\in\{\underline{i}_S\}\backslash\{ (\T n)\}$, to split $(F)_\sigma=(F)_{\sigma0}+(F)_{\sigma1}$. If $n\not= m$, we use~\eqref{expand2}, with some arbitrary $k\in\{\underline{i}_S\}\backslash\{(\T nm)\}$, to split  $(F)_\sigma=(F)_{\sigma0}+(F)_{\sigma1}$.
\end{itemize}
Applying this procedure to $(G_{i_1i_1})^{-1}$, we obtain a similar expansion as in~\eqref{eq.220}. Expanding the remaining factors of $(G_{i_ji_j})^{-1}$ as before (using only single labels), we obtain the analogue expression to~\eqref{eq.212}. The remainder terms can be estimated in the same way as before, simply by using the fact that each term in the remainder contains at least $2p$ off-diagonal resolvent entries. Also note that the bounds on the number of terms in the expansion still apply.  It therefore suffices to bound the summands in the first term on the right side of~\eqref{eq.212}, (now some of the indices may coincide). Recall that $s$ denotes the number of single labels in the fixed configuration $\underline{i}$. We claim that
\begin{align}\label{Q bound general}
\left| \E\, Q_{i_1}\ol{(F_{i_1})}_{\sigma_1}\cdots Q_{i_{2r}}(F_{i_{2r}})_{\sigma_{2r}}\right|\le C^{2p}\Phi^{p+s}\,.
\end{align}
This follows in a similar way as above, using the following observation:
\begin{itemize}
 \item[$(C')$]
Let $a\in S(\Gamma)$, then there is an label $b\in\{1,\ldots,{2r}\}\backslash\{a\}$, such that the monomial $(F_{i_b})_{\sigma_{b}}$ contains an off-diagonal resolvent entry with $i_a$ as a lower index. 
\end{itemize}
The bound~\eqref{Q bound general} now follows in the same way as above, by only considering single labels.

We now return to the sum in~\eqref{without restrictions}. We perform the summation by first fixing a partition $\Gamma\in\caP_{2r}$. Then
\begin{align}\label{sum over non label}
 \frac{1}{N^{2r}}\sum_{\underline{i}}\lone(\Gamma=\Gamma(\underline{i}))\le \left(\frac{1}{N}\right)^{2r-|\Gamma|}\le \left(\frac{1}{\sqrt{N}}\right)^{2r-s}\,.
\end{align}
Here we used that any block in the partition $\Gamma$ that is not associated to a single label, consists of at least two elements. Thus $|\Gamma|\le (2r+s)/2=r+s/2$. Now, using ${N}^{-1/2}\le C\Phi$, we find, combining~\eqref{sum over non label}, ~\eqref{Q bound general} and~\eqref{bound on remainder},
\begin{align*}
 \E\left|\frac{1}{N}\sum_{i=1}^NQ_{i}\left(\frac{1}{G_{ii}} \right)\right|^{2r}\le (Cp)^{3p}\sum_{\Gamma\in\caP_{2r}}(C \Phi)^{2p}\,.
\end{align*}
Finally, we recall that the number of partitions of $p$ elements is bounded by $(Cp)^{2p}$, thus
\begin{align*}
 \E\left|\frac{1}{N}\sum_{i=1}^NQ_{i}\left(\frac{1}{G_{ii}} \right)\right|^{2r}\le (Cp)^{5p} \Phi^{2p}\,.
\end{align*}
This proves the desired lemma.
\end{proof}

We will use the fluctuation Lemma~\ref{lemma.21} in a slightly generalized setting. Abbreviate
\begin{align}\label{small green function}
 g_i(z)\deq\frac{1}{\lambda v_i-z-\mfc(z)}\,,\quad\quad z\in\caD_L\,,\quad\lambda\in\caD_{\lambda_0}\,,\quad i\in\{1,\ldots,N\}\,,
\end{align}
and also recall that the random variables $(g_i)$ are bounded uniformly in $\lambda$ and $z$ as follows form the stability bound~\eqref{stability bound}. We will use the following corollary of the fluctuation lemma:
\begin{corollary}\label{generalized Z lemma}
Suppose $\xi$ satisfies~\eqref{eq.xi} and let $L\ge 12\xi$. Let $\Xi$  be the event defined in Lemma~\ref{lemma.21} and assume it has $(\xi,\nu)$-high probability. Then there exists a constant $C$, independent of $\lambda$ and $z$, such that, for $p\in\N$, even and satisfying $ p\le \nu(\log N)^{\xi-3/2}$, and $n=1,2,3$,
\begin{align}\label{generalied Z lemma estimate}
  \E\left|\frac{1}{N}\sum_{i=1}^NQ_i\left(g_i^n\frac{1}{G_{ii}}\right)\right|^{p}\le (Cp)^{5p}\left(\Phi(z)\right)^{2p}\,,
\end{align}
for $z\in \caD_L$, $\lambda\in\caD_{\lambda_0}$.

\end{corollary}
\begin{proof}
In the proof of Lemma~\ref{lemma.21}, we used the following two properties of $(Q_i)$:
\begin{itemize}
 \item [$i.$] For general random variables $A=A(W)$ and $B=B(W)$, $\E[ ({Q}_iA) B]=\E[B \E_i{Q}_i A]=0$, if $B$ is independent of the variables in the $i^{\mathrm{th}}$-column/row of $W$.
\item[$ii.$] For a general random variable $A=A(W)$, $q\in\N$, $\E|Q_iA|^q\le 2^q\E|A|^q$; see~\eqref{for general rv}.
\end{itemize}
Fix $n$ and define $\widetilde{Q}_i\deq Q_{i}g_i^n$. Since the random variables $(v_i)$ are independent of the random variables $(w_{ij})$, property $i$ holds true with $Q_i$ replaced by $\widetilde{Q}_i$. (Here $\E$ stands for the expectation with respect the $(w_{ij})$ and the~$(v_i)$ random variables, but, since the random variables $(g_i)$ are uniformly bounded, one could replace $\E$ by the conditional expectation with respect the $(w_{ij})$). Since the family of random variables $(g_i)$ is uniformly bounded and independent of $W$, property $ii$ holds now with $\E|\widetilde{Q}_i A|^q\le 2^q \E[|g_i|^{qn}] \E[|A|^q]\le C^q\E|A|^q$, for some constant~$C$, for any random variables $A=A(W)$ depending only on $W$. Thus the proof of Lemma~\ref{lemma.21} also applies to left side of~\eqref{generalied Z lemma estimate}: It suffices to multiply the bounds with $C^{p}$.

\end{proof}

\subsection{Strong self-consistent equation}\label{subsection strong self-consistent equation}
With Corollary~\ref{generalized Z lemma} at hand, it is easy to derive a stronger self-consistent equation for $m-\mfc$ than the one obtained in Lemma~\ref{lemma.15}. Recall the notation $[Z]=\frac{1}{N}\sum_{i=1}^NZ_i$.
\begin{lemma}\label{cor.french}
Suppose $\xi$ satisfies~\eqref{eq.xi}. Assume that there exists a deterministic function $\gamma(z)$ with \mbox{$\gamma(z)\le (\varphi_N)^{-2\xi}$} such that, for all $z\in \caD_L$, $\lambda\in\caD_{\lambda_0}$,
\begin{align*}
 \Lambda(z)\le\gamma(z)\,,
\end{align*}
with $(\xi,\nu)$-high probability. Then we have with $(\xi-2,\nu)$-high probability
\begin{align}\label{eq.251}
 |[Z]|\le C(\varphi_N)^{10\xi}\left(\frac{\im m_{fc}(z)+\gamma(z)}{N\eta}\right)\,,
\end{align}
and, for $n=1,2,3$,
\begin{align}\label{eq.250}
 \left|\frac{1}{N}\sum_{i=1}^N Q_i\left(g_i^n\frac{1}{G_{ii}}\right)\right|\le C(\varphi_N)^{10\xi}\left(\frac{\im m_{fc}(z)+\gamma(z)}{N\eta}\right)\,,
\end{align}
where the constant $C$ can be chosen uniformly in $z\in\caD_L$ and $\lambda\in\caD_{\lambda_0}$.

Moreover, the strong self-consistent equation
\begin{align}\label{self}
\left| (1-R_2) [\v]- R_3 [\v]^2\right|&\le \caO\left(\frac{\Lambda^2}{\log N}\right)+\caO\left((\varphi_N)^{10\xi}\,\left(\frac{\im m_{fc}(z)+\gamma(z)}{N\eta}\right)\right)+\caO\left(\frac{\lambda(\varphi_N)^{\xi}}{\sqrt{N}}\right)
\end{align}
holds with $(\xi-2,\nu)$-high probability, uniformly in $z\in\caD_{L}$, $\lambda\in\caD_{\lambda_0}$. 
\end{lemma}

Note that in the above lemma we have not changed the value of the parameter $\nu$, but replaced the $N$-dependent parameter $\xi$ by $\xi-2$. This is necessary at this point, since in the iteration procedure below we apply this lemma $\log\log N$ times.
\begin{proof}
We begin by proving \eqref{eq.251}. From Schur's complement formula we obtain
\begin{align}\label{eq.252}
 Q_i\left(\frac{1}{G_{ii}}\right)&=Q_i\left(\lambda v_i+w_{ii}-z-\sum_{k,l}^{(i)}h_{ik}G^{(i)}_{kl}h_{li}\right)\nonumber\\
&=w_{ii}-Q_i\left(\sum_{k,l}^{(i)}h_{ik}G^{(i)}_{kl}h_{li}\right)\nonumber\\
&=w_{ii}-Z_i\,.
\end{align}
Since $|w_{ii}|\le(\varphi_N)^{\xi}N^{-1/2}$, with high probability, we obtain from the large deviation estimate~\eqref{LDE1},
\begin{align*}
\left| \frac{1}{N}\sum_{i=1}^Nw_{ii}\right|\le \frac{(\varphi_N)^{2\xi}}{N}\,,
\end{align*}
with $(\xi,\nu)$-high probability. Hence it suffices to bound the average of the left side of~\eqref{eq.252} to get~\eqref{eq.251}.

Theorem~\ref{thm.weak} and Lemma~\ref{lemma.14} imply that assumptions $i$ and $ii$ of Lemma~\ref{lemma.21} hold with high probability. By Lemma~\ref{lemma.111}, we have $c\le |G_{ii}|\le C$, with high-probability. Finally, from the estimate on $Z_i$ in~\eqref{eq.162} and the bound on $w_{ii}$, we conclude that assumption $iii$ of Lemma~\ref{lemma.21}, i.e., Inequality~\eqref{eq.24}, holds with high probability. Hence the event~$\Xi$, as defined in Lemma~\ref{lemma.21}, being the intersection of several $(\xi,\nu)$-high probability events, has $(\xi-1/2,\nu)$-high probability. Thus we can apply Lemma~\ref{lemma.21}: Choosing $p$ in~\eqref{eq.25} as the largest even integer smaller than $\nu(\log N)^{\xi-2}$, Markov's inequality yields~\eqref{eq.251}. 
Note that we have not changed the parameter $\nu$ here, but have replaced $\xi$ by $\xi-2$. 

Similarly,~\eqref{eq.250} follows from Corollary~\ref{generalized Z lemma} and a high-moment Markov estimate.

 To derive the self-consistent Equation~\eqref{self}, we return to~\eqref{eq.149bis}, i.e.,
\begin{align}\label{self1}
(1-R_2)[\v]&=R_3[\v]^2+\frac{1}{N}\sum_{i=1}^N\frac{1}{(\lambda v_i-z-m_{fc})^2}\caY_i +\frac{1}{N}\sum_{i=1}^N\frac{1}{(\lambda v_i-z-m_{fc})^3}(\caY_i^2-2[\v]\caY_i)\nonumber\\
&\quad\quad+\caO(\Lambda^3)+\caO(\max_i|\caY_i|^3)+\caO\left(\frac{\lambda (\varphi_N)^{\xi}}{\sqrt{N}}\right)\,,
\end{align}
which holds with $(\xi,\nu)$-high probability. Recall that, on the event $\Xi$,
\begin{align}\label{bound for yi}
\caY_i=w_{ii}-Z_{i}-(m^{(i)}-m)&=w_{ii}-Z_{i}+\caO\left(\Phi^2\right)=Q_{i}\left(\frac{1}{G_{ii}}\right)+\caO\left(\Phi^2\right)  \,.
\end{align}
Using~\eqref{eq.250} to control the second term on the right side of~\eqref{self1} and the apriori bound~\eqref{eq.163} on $|\caY_i|$ to control the terms $\caO(\caY_i^2)$ in~\eqref{self1}, we obtain
\begin{align*}
(1-R_2)[\v]&=R_3[\v]^2 -2\frac{1}{N}\sum_{i=1}^N\frac{1}{(\lambda v_i-z-m_{fc})^3}[\v]\caY_i+\caO\left((\varphi_N)^{10\xi}\frac{\im m_{fc}(z)+\gamma(z)}{N\eta}\right)\\
&\quad\quad+\caO(\Lambda^3)+\caO\left(\frac{\lambda (\varphi_N)^{\xi}}{\sqrt{N}}\right)\,,
\end{align*}
with $(\xi-2,\nu)$-high probability. Arguing as in the proof of Lemma~\ref{lemma.15}, we obtain~\eqref{self}.

\end{proof}

\subsection{Proof of the strong deformed semicircle law}\label{proof of the strong deforemd semicircle law}
The proof of Theorem~\ref{thm.strong} is based on an iteration using the weak semicircle law, i.e., Theorem~\ref{thm.weak}, and Lemma~\ref{cor.french}. We start with an entirely deterministic lemma:

\begin{lemma}\label{lemma.22}
 Assume that $1\le\xi_1\le \xi_2$. Let $0<\tau<1$ and $L>40 \xi_2$. Suppose that there is a function $\gamma(z)$ satisfying
\begin{align}\label{eq.280}
 \gamma(z)\le(\varphi_N)^{11\xi_2}\left(\frac{\lambda^{1/2}}{{N}^{1/4}}+\frac{1}{N\eta} \right)^{1-\tau}\,,
\end{align}
such that $\Lambda(z)\le \gamma(z)$, for all $z\in \caD_L$, $\lambda\in\caD_{\lambda_0}$. We also assume that, for $z\in \caD_L$, $\lambda\in\caD_{\lambda_0}$,
\begin{align}\label{eq.281}
\left|(1-R_2) [\v]- R_3[\v]^2\right|=\caO\left(\frac{\Lambda^2}{\log N}\right)+\caO\left(\frac{\lambda(\varphi_N)^{\xi_1}}{\sqrt{N}}+(\varphi_N)^{10\xi_1}\frac{\alpha(z)+\gamma(z)}{N\eta}\right)\,,
\end{align}
where $\alpha=|1-R_2|$ was defined in~\eqref{eq.191}. Moreover, we assume that $\Lambda\ll 1$, if $\eta\sim1$. Then
\begin{align}\label{eq.282}
 \Lambda(z)\le (\varphi_N)^{11\xi_2}\left(\frac{\lambda^{1/2}}{{N}^{1/4}}+\frac{1}{N\eta} \right)^{1-\tau/2}\,,
\end{align}
for all $z\in \caD_L$, $\lambda\in\caD_{\lambda_0}$.
\end{lemma}
\begin{proof}
 The proof is based on a dichotomy argument. We set
\begin{align*}
 \alpha_0(z)\deq (\varphi_N)^{(10+3/4)\xi_2}\left(\frac{\lambda^{1/2}}{{N}^{1/4}}+\frac{1}{N\eta} \right)^{1-\tau/2}\,.
\end{align*}
Note that $\alpha_0\le\gamma$ and $\alpha_0 \ll 1$. Using~\eqref{eq.280} we find
\begin{align*}
C(\varphi_N)^{10\xi_1}\left(\frac{\lambda}{\sqrt{N}}+\frac{\gamma(z)}{N\eta}\right)&\le \alpha_0^2+(\varphi_N)^{21\xi_2}\frac{1}{N\eta}\left(\frac{\lambda^{1/2}}{{N}^{1/4}}+\frac{1}{N\eta} \right)^{1-\tau}\,\\
&\le(\varphi_N)^{22\xi_2}\left(\frac{\lambda^{1/2}}{{N}^{1/4}}+\frac{1}{N\eta} \right)^{2-\tau}\,.
\end{align*}
First consider the case $\alpha\le \alpha_0$: In this case, $\kappa \ll 1$ and $|R_3| > c$ for some constant $c$. From equation~\eqref{eq.281} we find that
\begin{align*}
 |[\v]^2| \le \left|[\v]^2- \frac{\alpha [\v]}{c} \right|+\frac{\alpha |[\v]|}{c} \le o(1)|[\v]|^2+(\varphi_N)^{22\xi_2}\left(\frac{\lambda^{1/2}}{{N}^{1/4}}+\frac{1}{N\eta}\right)^{2-\tau}+C(\varphi_N)^{10\xi_2}\frac{\alpha}{N\eta}+\frac{\alpha |[\v]|}{c}\,,
\end{align*}
and hence
\begin{align*}
 (|[\v]|-\frac{\alpha}{c})^2 \le \frac{\alpha^2}{c^2}+ (\varphi_N)^{22\xi_2}\left(\frac{\lambda^{1/2}}{{N}^{1/4}}+\frac{1}{N\eta}\right)^{2-\tau}+(\varphi_N)^{(10+3/4)\xi_2}\frac{\alpha}{N\eta}\,.
\end{align*}
(Note that we used here $\varphi_N$ to compensate for various constants, as we shall do below.) Thus taking the square root, recalling that $\alpha\le\alpha_0$ and using the definition of $\alpha_0$, we find
\begin{align*}
 |[\v]|\le C\alpha_0+(\varphi_N)^{11\xi_2}\left(\frac{\lambda^{1/2}}{{N}^{1/4}}+\frac{1}{N\eta} \right)^{1-\tau/2}\,,
\end{align*}
and the claim follows for $\alpha\le \alpha_0$.

Next, consider $\alpha>\alpha_0$: Recall that $|R_3| < C_3$ for some constant $C_3 > 0$. Assume first that $\Lambda\le\alpha/(2 C_3)$. Then in ~\eqref{eq.281} we can absorb the terms $R_3 [\v]^2$ and $\Lambda^2/\log N$ into the term $\alpha |[\v]|$ and we get
\begin{align}\label{eq.285}
 \Lambda&\le C(\varphi_N)^{10\xi_1}\left(\frac{\lambda}{\alpha \sqrt{N}}+\frac{1}{N\eta}+\frac{\gamma}{\alpha N\eta }\right)\nonumber\\
&\le \frac{1}{(\varphi_N)^{\xi_2/4}}\left(\frac{\alpha_0^2}{\alpha}+\alpha_0+\frac{\alpha_0^2}{\alpha}\right)\,,
\end{align}
where we used the definitions of $\gamma$ and $\alpha_0$. Since we assumed that $\alpha>\alpha_0$, we get $\Lambda\ll\alpha/(2 C_3)$ if $\Lambda\le \alpha/(2 C_3)$. Thus, if $\alpha\ge \alpha_0$, we either have $\Lambda>\alpha/(2 C_3)$ or $\Lambda\ll\alpha/(2 C_3)$. By the continuity of $\Lambda(z)$ in $\eta=\im z$, we must have $\Lambda\ll\alpha$, since we assume that $\Lambda(z)\ll 1=\caO(\alpha)$, for $\eta\sim 1$. Thus, the claim follows from~\eqref{eq.285}.
\end{proof}

\begin{proof}[Proof of Theorem~\ref{thm.strong}]
We prove~\eqref{strong1}. Let $\xi=\frac{A_0+o(1)}{2}\log\log N $ and set 
\begin{align*}
 \widetilde{\xi}\deq2(\log \log N/\log 2)+\xi\,.
\end{align*}
Note that $\widetilde{\xi}\le 3\xi/2\le A_0\log \log N$. Let $L\ge 40\widetilde{\xi}$.
To prove~\eqref{strong1} it suffices to prove
\begin{align}\label{eq.299}
 \bigcap_{\substack{z\in \caD_L\\ \lambda\in\caD_{\lambda_0}}}\left\{\left| m(z)-m_{fc}(z) \right| \le (\varphi_N)^{12\widetilde\xi}\left(\min\left\{\frac{(\varphi_N)^{12\widetilde{\xi}}}{\alpha}\frac{\lambda}{\sqrt{N}},\frac{\lambda^{1/2}}{N^{1/4}}\right\}+\frac{1}{N\eta}\right)\right\}\,,
\end{align}
with $(\xi,\nu)$-high probability.

The weak semicircle law, i.e., Theorem~\ref{thm.weak} with $\widetilde{\xi}$ replacing $\xi$, yields
\begin{align*}
 \Lambda\le (\varphi_N)^{2\widetilde{\xi}}\left(\frac{1}{N\eta}\right)^{1/3}\le (\varphi_N)^{2\widetilde{\xi}}\left(\frac{\lambda^{1/2}}{N^{1/4}}+\frac{1}{N\eta}\right)^{1-2/3}\,,
\end{align*}
for $z \in \caD_L$, $\lambda\in\caD_{\lambda_0}$, with $(\widetilde{\xi},\nu)$-high probability. Thus~\eqref{eq.21} holds with
\begin{align*}
 \gamma(z)\deq(\varphi_N)^{11\widetilde{\xi}}\left(\frac{\lambda^{1/2}}{N^{1/4}}+\frac{1}{N\eta}\right)^{1-2/3}\,.
\end{align*}
Since $L\ge 40\widetilde{\xi}$, we also have $\gamma(z)\le (\varphi_N)^{-2\widetilde{\xi}}$. Hence, by Lemma~\ref{cor.french} we have
\begin{align*}
\left|(1-R_2)[\v]- R_3 [\v]^2\right|&\le C\frac{\Lambda^2}{\log N}+C(\varphi_N)^{10\widetilde{\xi}}\left(\frac{\lambda}{\sqrt{N}}+\frac{\im m_{fc}(z)+\gamma(z)}{N\eta}\right)\,,
\end{align*}
for $z\in \caD_L$, $\lambda\in\caD_{\lambda_0}$, with $(\widetilde{\xi}-2,\nu)$-high probability. Since $\im m_{fc}\le C\alpha$, by Lemma~\ref{lemma.17}, this implies~\eqref{eq.281} with $\xi_1=\widetilde{\xi}$. Also, $\gamma$ satisfies~\eqref{eq.280} with $\xi_2=\widetilde{\xi}$ and $\tau=2/3$.  Moreover, since $\Lambda\le\gamma\le(\varphi_N) ^{-2\widetilde{\xi}}$, we have $\Lambda\ll1$, if $\eta\sim1$. Therefore, we can apply Lemma~\ref{lemma.22} with $\xi_1=\xi_2=\widetilde{\xi}$ to obtain
\begin{align*}
 \Lambda\le (\varphi_N)^{11\widetilde{\xi}}\left(\frac{\lambda^{1/2}}{N^{1/4}}+\frac{1}{N\eta}\right)^{1-1/3}\,,
\end{align*}
for $z\in \caD_L$, $\lambda\in\caD_{\lambda_0}$, with $(\widetilde{\xi}-2,\nu)$-high probability. Iterating this process $M$ times, we find that
\begin{align*}
 \Lambda\le (\varphi_N)^{11\widetilde{\xi}}\left(\frac{\lambda^{1/2}}{N^{1/4}}+\frac{1}{N\eta}\right)^{1-\frac{2}{3}\left(\frac{1}{2}\right)^M}\,,
\end{align*}
for $z\in \caD_L$, $\lambda\in\caD_{\lambda_0}$, holds with $(\widetilde{\xi}-2M,\nu)$-high probability. We choose $M=\lfloor\log \log N/ \log 2\rfloor-1$, here $\lfloor\cdot\rfloor$ denotes the integer part. Since $\lambda^{1/2}N^{-1/4}+N^{-1/2}+(N\eta)^{-1}\ge c N^{-1}$ on $\caD_L$, we get
\begin{align*}
 \left(\frac{\lambda^{1/2}}{N^{1/4}}+\frac{1}{N\eta}\right)^{-\frac{2}{3}\left(\frac{1}{2}\right)^M}\le C\le  (\varphi_N)^{\widetilde{\xi}}\,.
\end{align*}
Thus 
\begin{align}\label{eq.289}
 \Lambda\le (\varphi_N)^{12 \widetilde{\xi}}\left(\frac{\lambda^{1/2}}{N^{1/4}}+\frac{1}{N\eta}\right)\,,
\end{align}
for $z\in \caD_L$, $\lambda\in\caD_{\lambda_0}$, with $(\xi+2,\nu)$-high probability (the factor of $2$ comes from the $-1$ in $M$). This proves~\eqref{eq.299} when
\begin{align*}
\frac{(\varphi_N)^{12\widetilde{\xi}}}{\alpha}\frac{\lambda}{\sqrt{N}} \geq \frac{\lambda^{1/2}}{N^{1/4}}\,.
\end{align*}
In case
\begin{align*}
\frac{\lambda^{1/2}}{N^{1/4}} \leq (\varphi_N)^{-12\widetilde\xi}\alpha \leq \frac{\lambda^{1/2}}{N^{1/4}}+\frac{1}{N\eta}\,,
\end{align*}
we have
\begin{align*}
\min\left\{\frac{(\varphi_N)^{12\widetilde{\xi}}}{\alpha}\frac{\lambda}{\sqrt{N}},\frac{\lambda^{1/2}}{N^{1/4}}\right\}+\frac{1}{N\eta} \geq \frac{\lambda}{\sqrt{N}} \left(\frac{\lambda^{1/2}}{N^{1/4}}+\frac{1}{N\eta}\right)^{-1} + \frac{1}{N\eta} \geq \frac{1}{2} \left(\frac{\lambda^{1/2}}{N^{1/4}}+\frac{1}{N\eta}\right),
\end{align*}
and the proof for \eqref{eq.299} is similar to the above case. Finally, when
\begin{align}\label{eq.290}
 (\varphi_N)^{-12\widetilde\xi}\alpha\ge \frac{\lambda^{1/2}}{N^{1/4}}+\frac{1}{N\eta}\,,
\end{align}
set
\begin{align*}
 \ga(z)\deq  (\varphi_N)^{12\widetilde{\xi}}\left(\frac{\lambda^{1/2}}{N^{1/4}}+\frac{1}{N\eta} \right)\,.
\end{align*}
Then by Lemma~\ref{cor.french} we have
\begin{align*}
\alpha|[\v]|&\le C\Lambda^2+C(\varphi_N)^{10\widetilde{\xi}}\left(\frac{\alpha+\gamma(z)}{N\eta}\right)+C\frac{\lambda (\varphi_N)^{\xi}}{\sqrt{N}}\,,
\end{align*}
with $(\xi,\nu)$-high probability, where we used that $\im m_{fc}\le C\alpha$. Assuming that ~\eqref{eq.290} holds and using the definition of $\ga$, we have $\gamma(z)\le\alpha(z)$, and we get, using~\eqref{eq.289},
\begin{align}\label{eq.291}
 \left|[\v]\right|&\le C(\varphi_N)^{24\widetilde{\xi}}\frac{\lambda}{\alpha \sqrt{N}}+ C(\varphi_N)^{12\widetilde{\xi}}\left(\frac{1}{N\eta}+\frac{\gamma	}{\alpha {N}\eta}\right)\nonumber\\&\le (\varphi_N)^{12\widetilde{\xi}}\left((\varphi_N)^{12\widetilde{\xi}}\frac{\lambda}{\alpha\sqrt{N}}+\frac{1}{N\eta}\right)\,.
\end{align}

Hence, combining~\eqref{eq.289} and ~\eqref{eq.291} we find, using a simple lattice argument,~\eqref{eq.299}. Inequality~\eqref{strong2}  then follow from~\eqref{eq.299} combined with~\eqref{eq.161}.

\end{proof}

\section{Identifying the leading corrections in the bulk} \label{Identifying the leading corrections}
In this section, we identify the leading correction terms to $m-\mfc$ stemming from the diagonal random matrix~$V$. We define random variables $\zeta_0(z)\equiv\zeta_0^N(z)$, which only depends on the random variables $(v_i)$, such that, in the bulk of the spectrum, the leading correction term in the estimate on $|m(z)-m_{fc}(z)-\zeta_0|$ is of order $(N\eta)^{-1}$. This estimate is then used to prove Theorem~\ref{thm xi0}. 

In this section, we fix $\xi=\frac{A_0+o(1)}{2}\log\log N$ and choose $L\ge 40\xi$.
\subsection{Preliminaries}
Recall the notation $\Lambda=|m-\mfc|$ and the definition of $(R_n)$ in~\eqref{eq.171}. In Lemma~\ref{lemma.17}, we showed that $1-R_2\sim\sqrt{\kappa+\eta}$, $R_3 =\caO(1)$, for all $z\in\caD_L$, $\lambda\in\caD_{\lambda_0}$. We will need some more notation. For $z\in\caD_L$, $\lambda\in\caD_{\lambda_0}$, $n\in\N$, set
\begin{align*}
 r_n(z)\equiv r_n\deq\frac{1}{N}\sum_{i=1}^N\frac{1}{(\lambda v_i-z-\mfc)^n}-\int \frac{\dd\mu(v)}{(\lambda v-z-\mfc)^n}\,.
\end{align*}
Recall from~\eqref{lawlargenumber} that $|r_n(z)|\le (\varphi_N)^{\xi}\frac{\lambda}{\sqrt{N}}$, with high probability, uniformly in $z\in\caD_L$ and $\lambda\in\caD_{\lambda_0}$. Thus, combining the above observations, we obtain
\begin{align}\label{Xi0}
C^{-1}\sqrt{\kappa+\eta}\le& |1-R_2-r_2|\le C\sqrt{\kappa+\eta}\,,
\quad\quad |R_3+r_3|\le C\,,
\end{align}
for all $z\in\caD_L$, $\lambda\in\caD_{\lambda_0}$, with $(\xi,\nu)$-high probability, for some $C>1$.

\subsubsection{Definition of $\zeta_0$}
In order to define $\zeta_0\equiv \zeta_0(z)$, it is convenient to introduce a high-probability event $\Xi_0$, by requiring that~\eqref{Xi0} holds on it. We define $\zeta_0$ as the solution to the equation
\begin{align}\label{definition of zeta0}
 (1-R_2-r_2)\zeta_0(z)=r_1(z)+(R_3+r_3)\zeta_0(z)^2\,,\quad\quad z\in \C^+ \,,\quad\lambda\in\caD_{\lambda_0}\,,
\end{align}
such that $\zeta_0(z)\to 0$, as $\im z\to \infty$.

First, note that $\zeta_0(z)$, $z\in\caD_L$, is well-defined on $\Xi_0$. Second, note that $\zeta_0$ only depends on $(v_i)$, but is independent of the random entries $(w_{ij})$ of the Wigner matrix $W$. Third, from the discussion in the preceding subsection, we infer:
\begin{lemma}\label{lemma bound zeta_0}
 There is a constant $c>0$, such that, for $z\in\caD_L$, $\lambda\in\caD_{\lambda_0}$, we have on $\Xi_0$,
\begin{align*}
 |\zeta_0(z)|\le(\varphi_N)^{c\xi} \min \left\{\frac{\lambda^{1/2}}{N^{1/4}},\frac{\lambda}{\sqrt{\kappa+\eta}}\frac{1}{\sqrt{N}}\right\}\,,
\end{align*}
with $(\xi,\nu)$-high probability.
\end{lemma}
We  omit the proof and just remark that it is suffices to consider the cases ${\kappa+\eta}\sim |1-R_2|^2\ll |r_3|$ and $|r_3|\ll |1-R_2|^2$.

Recall the (strong) self-consistent equation for $m(z)-\mfc(z)$ in~\eqref{self}. The definition of $\zeta_0$ is natural in the sense that it embodies the leading correction to $m-\mfc$ stemming from the random matrix $V$: Subtracting the defining equation for $\zeta_0$ from the self-consistent Equation~\eqref{self}, we obtain, after some manipulations,
\begin{align*}
(1-R_2-r_2)(m-\mfc-\zeta_0)=(R_3+r_3)(m-\mfc)^2-(R_3+r_3)\zeta_0^2+\caO(\Lambda^3)+\caO\left((\varphi_N)^{c\xi}\frac{\im \mfc+\Lambda}{N\eta}\right)\,,
\end{align*}
on some high probability event $\Xi$. Theorem~\ref{thm xi0}, now follows easily from analyzing the stability of this equation in the variable $\zeta(z)\deq m(z)-\mfc(z)-\zeta_0(z)$. 

\subsection{Proof of Theorem~\ref{thm xi0}}
Next, we carry out the details of the proof of Theorem~\ref{thm xi0}.
\begin{proof}[Proof of Theorem~\ref{thm xi0}]
Recall the event $\Xi_0$ defined in~\eqref{Xi0}. Let 
\begin{align}\label{choice gamma}
\gamma(z)\deq(\varphi_N)^{c_1\xi}\left(\min\left\{\frac{\lambda^{1/2}}{N^{1/4}},\frac{\lambda}{\sqrt{\kappa+\eta}}\frac{1}{\sqrt{N}}\right\}+\frac{1}{N\eta}\right)\,,\quad\quad z\in\caD_L\,.
\end{align}
 Choosing $c_1$ sufficiently large in~\eqref{choice gamma}, we can achieve that $|\zeta_0|\le\gamma(z)$ on~$\Xi_0$.  Next, it follows from Theorem~\ref{thm.strong}, Lemma~\ref{cor.french} and Lemma~\ref{lemma.14}, that there is an event $\Xi_1$, having $(\xi,\nu)$-high probability, such that the following holds on it: There is a constant $c_0$ such that $|\Lambda(z)|\le\gamma(z)$,
\begin{align}\label{Xi1}
\max_{i}|\caY_i(z)|\le (\varphi_N)^{c_0\xi/2}\sqrt{\frac{\im\mfc(z)+\gamma(z)}{N\eta}}\,,
\end{align}
and, recalling~\eqref{small green function},
\begin{align}\label{Xi2}
\left|\frac{1}{N}\sum_{i=1}^N g_i^n\caY_i(z)\right|\le (\varphi_N)^{c_0\xi}\left(\frac{\im\mfc(z)+\gamma(z)}{N\eta}\right)\,,
\end{align}
for $n=1,2,3$, where $\caY_i=w_{ii}-Z_i-(m^{(i)}-m)$; see~\eqref{eq.149}.
Since both events $\Xi_0$ and $\Xi_1$ have high probability, the event $\Xi\deq\Xi_0\cap\Xi_1$ has $(\xi,\nu)$-high probability, with a slightly smaller $\nu>0$. Set \mbox{$\zeta(z)\deq m(z)-\mfc(z)-\zeta_0(z)$}. Subtracting the defining equation of $\zeta_0$, from Equation~\eqref{eq.173}, we obtain, using the bounds in~\eqref{Xi1} and~\eqref{Xi2},
\begin{align}\label{zeta stability}
 (1-R_2-r_2)\zeta=(R_3+r_3)(m-\mfc)^2-(R_3+r_3)\zeta_0^2+ \caO(\Lambda^3)+\caO\left((\varphi_N)^{c_0\xi}\frac{\im \mfc(z)+\gamma(z)}{N\eta}\right)\,,
\end{align}
on $\Xi$, for $z\in\caD_L$, $\lambda\in\caD_{\lambda_0}$. Let $c_2>\max\{c_0,c_1\}$ and set
\begin{align*}
 \alpha_0(z)\deq (\varphi_N)^{c_2\xi}\left(\min\left\{\frac{\lambda^{1/2}}{N^{1/4}},\frac{\lambda}{\sqrt{\kappa+\eta}}\frac{1}{\sqrt{N}}\right\}+\frac{1}{N\eta}\right)\,.
\end{align*}
Thus, on $\Xi$, we have $\Lambda\ll\alpha_0$ and $|\zeta_0|\ll \alpha_0$, for all $z\in\caD_L$, $\lambda\in\caD_{\lambda_0}$.

Recall that we defined the domain
\begin{align*}
 \caB_L = \caD_L\cap\{ z=E+\ii\eta\in\C\,:\, \sqrt{\kappa_E+\eta}\ge(\varphi_N)^{L\xi}N^{-1/4}\}\,.
\end{align*}
Note that we have on the domain $\caB_L$,
\begin{align*}
\min\left\{\frac{\lambda^{1/2}}{N^{1/4}},\frac{\lambda}{\sqrt{\kappa+\eta}}\frac{1}{\sqrt{N}}\right\} = \frac{\lambda}{\sqrt{\kappa+\eta}}\frac{1}{\sqrt{N}}\,.
\end{align*}
First, consider the case
\begin{align*}
\frac{\lambda}{\sqrt{\kappa+\eta}}\frac{1}{\sqrt{N}} \geq \frac{1}{N\eta}\,.
\end{align*}
In this case, we can easily see that
\begin{align*}
 \widetilde{\alpha}\deq\left|\frac{1-R_2-r_2}{R_3+r_3}\right|\ge \alpha_0\,,
\end{align*}
on $\Xi$. Then we obtain from~\eqref{zeta stability},
\begin{align*}
 |m-\mfc-\zeta_0|&\le\left|\frac{m-\mfc+\zeta_0}{\alpha_0}\right|\,|m-\mfc-\zeta_0|+C\frac{\Lambda^3}{\widetilde{\alpha}}+ C\frac{(\varphi_N)^{c_0\xi}}{\widetilde{\alpha}}\frac{\im\mfc+\gamma(z)}{N\eta} \nonumber \\
& \le o(1)|m-\mfc-\zeta_0|+C\frac{\gamma(z)^3}{\widetilde{\alpha}}+C\frac{(\varphi_N)^{c_0\xi}}{\widetilde{\alpha}}\frac{\im\mfc+\gamma(z)}{N\eta}\,,
\end{align*}
on $\Xi$. By Lemma~\ref{lemma.17}, we have  $K \widetilde\alpha\ge\sqrt{\kappa+\eta}\sim\im\mfc$ on $\Xi_0$, for some constant $K$, and we obtain, for some $c>c_2$, 
\begin{align*}
|m-\mfc-\zeta_0|\le (\varphi_N)^{c\xi}\left(\frac{\lambda^3}{(\kappa+\eta)^{2}}\frac{1}{N^{3/2}} +\frac{1}{N\eta}\right)\,,
\end{align*}
on $\Xi$, for all $\lambda\in\caD_{\lambda_0}$. Since the condition $\sqrt{\kappa+\eta}\ge(\varphi_N)^{L\xi}N^{-1/4}$ implies that
\begin{align*}
\frac{\lambda^3}{(\kappa+\eta)^{2}}\frac{1}{N^{3/2}} \ll \frac{1}{N(\kappa+\eta)} \leq \frac{1}{N\eta}\,,
\end{align*}
we conclude that
\begin{align*}
|m-\mfc-\zeta_0|\le \frac{(\varphi_N)^{c\xi}}{N\eta}\,,
\end{align*}
on $\Xi$, for $\lambda\in\caD_{\lambda_0}$. If
\begin{align*}
\frac{\lambda}{\sqrt{\kappa+\eta}}\frac{1}{\sqrt{N}} \leq \frac{1}{N\eta}\,,
\end{align*}
we may use the bound
\begin{align*}
|m-\mfc-\zeta_0|\le \Lambda+|\zeta_0|\le 2\gamma \leq \frac{(\varphi_N)^{c\xi}}{N\eta}\,,
\end{align*}
on $\Xi$, for $\lambda\in\caD_{\lambda_0}$. This finishes the proof.
\end{proof}

\section{Density of states} \label{Density of states}
In this section, we prove {Theorems~\ref{theorem density of states}, \ref{rigidity of eigenvalue spacing}, \ref{integrated density of states}, and \ref{rigidity of eigenvalues}.} Recall that we denote by $(\mu_\alpha)$ the eigenvalues of $H=\lambda V+W$. Define the {\it normalized eigenvalue counting function} of $H$ by
\begin{align}\label{counting measure}
 \rho(x)\deq\frac{1}{N}\sum_{\alpha=1}^N \delta (x-\mu_{\alpha})\,.
\end{align}
Then we can write
\begin{align*}
 m(z)=\frac{1}{N}\sum_{i=1}^NG_{ii}(z)=\int_{\R}\frac{\rho(x)\dd x}{x-z}\,,\qquad z\in\C^+\,.
\end{align*}
 For $E_1<E_2$, we defined in~\eqref{the counting functions} the counting functions
\begin{align*}
{\mathfrak{n}}(E_1,E_2)=\frac{1}{N}|\{\alpha\,:\,E_1<\mu_{\alpha}\le E_2\}|\,,\quad\quad{\mathfrak{n}}(E)=\frac{1}{N}|\{\alpha\,:\,\mu_{\alpha}\le E\}|\,.
\end{align*}
Similarly, we have denoted
\begin{align*}
n_{fc}(E_1,E_2)=\int_{E_1}^{E_2}\rho_{fc}(x)\,\dd x\,,\quad\quad n_{fc}(E)=\int_{-\infty}^{E}\rho_{fc}(x)\,\dd x\,,
\end{align*}
where $\rho_{fc}$ stands for the density of the free convolution measure $\mu_{fc}$.

 Throughout this section we fix \mbox{$\xi=\frac{A_0+o(1)}{2}\log\log N$} and choose $L\ge 40\xi$.
\subsection{Local density of states}\label{local density of states}
 Recall that $\kappa_E\deq\min\{| E-L_i|\,,\,i=1,2\}$. In the following, we set $\eta\deq N^{-1}$. The first part of Theorem~\ref{theorem density of states}, Inequality~\eqref{result density of states 1}, is an immediate consequence of the next two lemmas. Their proofs follow closely the proof of Lemma~8.1 and Lemma~8.2 in~\cite{EKYY1}.

\begin{lemma}\label{lemma helffer}
 Let $\eta\deq N^{-1}$. For any $E_1<E_2$  in $[-E_0,E_0]$, we define $f(x)\equiv f_{E_1,E_2,\eta}(x)$ to be an indicator function of the interval $[E_1,E_2]$, smoothed out on a scale $\eta$, i.e., $f(x)=1$, for $x\in[E_1,E_2]$, $f(x)=0$, for $x$ in $[E_1-\eta,E_2+\eta]^c$,  $|f'(x)|\le C \eta^{-1}$ and $|f''(x)|\le C\eta^{-2}$. Assume that the event
\begin{align}\label{eq.5.38}
 \bigcap_{\substack{z\in\caD_L\\ \lambda\in\caD_{\lambda_0}}}\left\lbrace |m(z)-m_{fc}(z)|\le (\varphi_N)^{c_0\xi}\left(\min\left\{\frac{\lambda^{1/2}}{N^{1/4}},\frac{\lambda}{\sqrt{\kappa_E+\eta}\sqrt{N}}\right\}+\frac{1}{N\eta}\right)\right\rbrace\,,
\end{align}
holds with $(\xi,\nu)$-high probability with $L\deq C_0\xi$, for some constant $C_0>0.$ Abbreviate 
\begin{align*}
 \kappa\deq\min\{\kappa_{E_1},\kappa_{E_2}\}\,,\quad\quad \caE\deq\max\{E_2-E_1,\,(\varphi_N)^{L}N^{-1}\}\,.
\end{align*}
Then, we have
\begin{align}\label{eq.5.40}
\left| \int_{\R} f(x)(\rho-\rho_{fc})(x)\dd x\right|\le (\varphi_N)^{c\xi}\left(\frac{1}{N}+\frac{\caE\lambda}{\sqrt{\kappa+\caE}}\frac{1}{\sqrt{N}}\right)\,,
\end{align}
with $(\xi,\nu)$-high probability, for some $c>c_0$.
\end{lemma}

\begin{proof}
 For convenience denote
\begin{align*}
 \rho^{\Delta}\deq\rho-\rho_{fc}\,,\quad\quad m^{\Delta}\deq m-m_{fc}\,.
\end{align*}
We use the Helffer-Sj\"ostrand formula. We set $y_0\deq(\varphi_N)^LN^{-1}$ and choose a smooth cut-off function~$\chi$ such that: 
\begin{align}\label{5.42}
 \chi(y)=1\,,\quad\textrm{on}\quad[-\caE,\caE]\,;\quad\quad \chi(y)=0\,,\quad\textrm{on}\quad[-2\caE,2\caE]^{c}\,;\quad\quad |\chi'(y)|\le\frac{C}{\caE}\,.
\end{align}
Starting from the Helffer-Sj\"ostrand formula, 
\begin{align}\label{5.43}
 f(w)=\frac{1}{2\pi}\int_{\R^2}\frac{\ii y f''(x)\chi(y)+\ii(f(x)+\ii y f'(x))\chi'(y)}{w-x-\ii y}\,\dd x\,\dd y\,,
\end{align}
we obtain
\begin{align}\label{eq.5.44}
 &\left|\int_{\R} f(w)\rho^{\Delta}(w)\dd w\right|\le C\int \dd x\int_0^{\infty}  \dd y (|f(x)|+|y| |f'(x)|)|\chi'(y)| |m^{\Delta}(x+\ii y)|\nonumber\\
&\quad\quad+C\left| \int \dd x\int_0^{\eta}\dd y f''(x)\chi(y) y\im m^{\Delta}(x+\ii y) \right|+C\left|\int \dd x\int_{\eta}^{\infty}\dd y f''(x)\chi(y) y\im m^{\Delta}(x+\ii y) \right|\,.
\end{align}
Using that $\chi'$ is supported on $[\caE,2\caE]$, we can bound the first term on the right side of the above inequality by
\begin{align}\label{eq.5.45}
 \frac{(\varphi_N)^{c\xi}}{\caE}\int\dd x\int_{\caE}^{2\caE}\dd y\, { (|f(x)|+ y|f'(x)|)} \left(\frac{\lambda}{\sqrt{\kappa_x+\caE}\sqrt{N}}+\frac{1}{\caE N}\right)\le (\varphi_N)^{c\xi}\left(\frac{\caE\lambda}{\sqrt{\kappa+\caE}\sqrt{N}}+\frac{1}{N}  \right)\,,
\end{align}
with $(\xi,\nu)$-high probability. In order to bound the two remaining terms in~\eqref{eq.5.44}, we first bound the imaginary part of $m^{\Delta}(x+\ii y)$. For $y\ge y_0$, we can use~\eqref{eq.5.38}. So assume that $0<y<y_0$. Using the spectral decomposition of $\lambda V+W$, it is easy to see that the function $y\mapsto y\im m(x+\ii y)$ is monotone increasing. Thus
\begin{align}\label{eq.5.46}
 y\im m(x+\ii y)\le y_0 \im m(x+\ii y_0)\le y_0\im m_{fc}(x+\ii y_0)+(\varphi_N)^{c\xi}y_0\left(\frac{\lambda^{1/2}}{N^{1/4}}+\frac{1}{N y_0}   \right)\,,\quad\quad (y\le y_0)\,.
\end{align}
Recalling that, by Lemma~\ref{lemma.12}, we have $\im m_{fc}(x+\ii y)\le C\sqrt{\kappa_x+y}$, we get
\begin{align}\label{eq.5.47}
  y\im m(x+\ii y)\le y_0 C\sqrt{\kappa_x+y}+(\varphi_N)^{c\xi}y_0\left(\frac{\lambda^{1/2}}{N^{1/4}}+\frac{1}{N y_0}   \right)\le\frac{(\varphi_N)^{c\xi}}{N}\,,\quad\quad(y\le y_0)\,,
\end{align}
with $(\xi,\nu)$-high probability. Using that $y\le y_0=(\varphi_N)^L N^{-1}$, we can now easily bound
\begin{align}\label{eq.5.48}
 |y \im m^{\Delta}(x+\ii y)|\le \frac{(\varphi_N)^{c\xi}}{N}\,,\quad\quad (y\le y_0)\,,
\end{align}
with $(\xi,\nu)$-high probability. Since by assumption we have $\eta\le y_0$, we can bound the second term on the right side of~\eqref{eq.5.44} by
\begin{align}\label{eq.5.49}
 \frac{(\varphi_N)^{c\xi}}{N}\int \dd x|f''(x)|\,\int_{{ 0}}^{\eta}\dd y\,\chi(y)\le \frac{(\varphi_N)^{c\xi}}{N}\,,
\end{align}
with $(\xi,\nu)$-high probability, where we used that the support of $f''$ has measure $\caO(\eta)$.
To bound the third term on the right side of~\eqref{eq.5.44}, we integrate by parts, first in $x$ then in $y$  to find the bound
\begin{align}\label{eq.5.60}
 C\left|\int\dd xf'(x)\eta \re m^{\Delta}(x+\ii \eta)  \right|&+C\left|\int\dd x\int _{\eta}^{\infty}\dd y f'(x)\chi'(y) { y} \re m^{\Delta} (x+\ii y)\right|\nonumber\\
&\quad\quad + C\left|\int\dd x \int_{\eta}^{\infty}\dd y f'(x)\chi(y)\re m^{\Delta}(x+\ii y)\right|\,.
\end{align}
The second term in~\eqref{eq.5.60} can be bounded similarly to the first term of~\eqref{eq.5.44} and we obtain
\begin{align}\label{eq.5.61}
 \left|\int\dd x\int _{\eta}^{\infty}\dd y f'(x)\chi'(y) { y} \re m^{\Delta} (x+\ii y)\right|\le (\varphi_N)^{c\xi}\left(\frac{\caE\lambda}{\sqrt{\kappa+\caE}\sqrt{N}}+\frac{1}{N}\right)\,,
\end{align}
with $(\xi,\nu)$-high probability. To bound the first and the third term in~\eqref{eq.5.60}, we write, for $y\le y_0$,
\begin{align}\label{eq.5.62}
 |m^{\Delta}(x+\ii y)|\le |m^{\Delta}(x+\ii y_0)|+\int_y^{y_0} { \dd u} \left(|\partial_u m(x+\ii u)|+|\partial_u m_{fc}(x+\ii u)|  \right)\,.
\end{align}
The first term on the right side of~\eqref{eq.5.62} can be estimated using~\eqref{eq.5.38}. For the others we observe that the Ward identity~\eqref{ward} implies, for $u\le y_0$,
\begin{align*}
 |\partial_u m(x+\ii u)|=|\frac{1}{N}\Tr G^2(x+\ii u)|\le \frac{1}{N}\sum_{i,j=1}^N|G_{ij}(x+\ii u)|^2=\frac{1}{u}\im m(x+\ii u)\le \frac{1}{u^2} y_0 \im m(x+\ii y_0)\,.
\end{align*}
Similarly, we obtain 
\begin{align*}
 |\partial_u m_{fc}(x+\ii u)|\le\int \frac{\rho_{fc}(t)\dd t}{|t-x-\ii u|^2}=\frac{1}{u}\im m_{fc}(x+\ii u)\le\frac{1}{u^2}y_0\im m_{fc}(x+\ii y_0)\,.
\end{align*}
From~\eqref{eq.5.62} we hence obtain
\begin{align}\label{eq.5.63}
 |m^{\Delta}(x+\ii y)|\le (\varphi_N)^{c\xi}\left(1+\int_y^{y_0} {\dd u \frac{y_0}{u^2}} \right)\le (\varphi_N)^{c\xi}\frac{y_0}{y}\,,\quad\quad (y\le y_0)\,,
\end{align}
with $(\xi,\nu)$-high probability. Thus we can bound the first term on the right side of~\eqref{eq.5.60} by
\begin{align}\label{eq.5.64}
 \left|\int\dd xf'(x)\eta \re m^{\Delta}(x+\ii \eta) \right|\le \frac{(\varphi_N)^{c\xi}}{N}\,,
\end{align}
with high probability. To bound the third term on the right side of~\eqref{eq.5.60}, we split the integration in the $y$ variable into the pieces $[\eta,y_0)$ and $[y_0,\infty)$. Using~\eqref{eq.5.63} we can bound the first piece by
\begin{align*}
 \int\dd x\,|f'(x)|\int_{\eta}^{y_0} \dd y\, |m^{\Delta}(x+\ii y)|\le \frac{(\varphi_N)^{c\xi}}{N}\,,
\end{align*}
with high probability. For the second integration piece we find
\begin{align*}
\int\dd x\,|f'(x)|\int_{y_0}^{{2\caE}} \dd y\, |m^{\Delta}(x+\ii y)|&\le {(\varphi_N)^{c\xi}}\int \dd x\,|f'(x)|\int_{y_0}^{2\caE}\dd y\,\left(\frac{{ \lambda}}{\sqrt{\kappa_x+y}}\frac{1}{\sqrt{N}}+\frac{1}{N y}\right)\nonumber\\
&\le (\varphi_N)^{c\xi}\left(\frac{1}{N}+\frac{1}{\sqrt{N}}\int_{y_0}^{2\caE}\dd y\,\frac{\lambda}{\sqrt{\kappa+y}}\right)\nonumber\\
&\le (\varphi_N)^{c\xi}\left(\frac{1}{N}+\frac{\caE\lambda}{\sqrt{\kappa+\caE}}\frac{1}{\sqrt{N}}\right)\,,
\end{align*}
with high probability. Adding all the contributions together, we have proven that
\begin{align*}
\left| \int_{\R} f(w)\rho^{\Delta}(w)\dd w\right|\le (\varphi_N)^{c\xi}\left(\frac{1}{N}+\frac{\caE\lambda}{\sqrt{\kappa+\caE}}\frac{1}{\sqrt{N}}\right)\,,
\end{align*}
with $(\xi,\nu)$-high probability.
\end{proof}
As a simple corollary, we obtain:
\begin{corollary}
 Under the assumptions of Lemma~\ref{lemma helffer}, there is $c>0$ such that, for any $-E_0\le E_1<E_2\le E_0$,
\begin{align}\label{eq.5.66}
 \left|\left(\frn(E_2)-\frn(E_1))-(n_{fc}(E_2)-n_{fc}(E_1)\right)\right|\le (\varphi_N)^{c\xi}\left(\frac{1}{N}+\frac{\caE\lambda}{\sqrt{\kappa+\caE}\sqrt{N}}\right)\,,
\end{align}
with $(\xi,\nu)$-high probability.
\end{corollary}
\begin{proof}
 Observe that, for $\eta=N^{-1}$,
\begin{align*}
 |\frn(x+\eta)-\frn(x-\eta)|\le C\eta \im m(x+\ii\eta)\le \frac{(\varphi_N)^{c\xi}}{N}\,,
\end{align*}
with $(\xi,\nu)$-high probability, where we used~\eqref{eq.5.46}. Hence,
\begin{align}\label{eq.5.68}
\left|\frn(E_1)-\frn(E_2)-\int_{\R} f(w)\rho(w)\dd w\right|\le C\sum_{i=1,2}
(\frn(E_i+\eta)-\frn(E_i-\eta))\le\frac{(\varphi_N)^{c\xi}}{N}\,,
\end{align}
with $(\xi,\nu)$-high probability. Moreover, since $\rho_{fc}$ is a bounded function, we find
\begin{align*}
 \left|n_{fc}(E_1)-n_{fc}(E_2)-\int_{\R} f(w)\rho_{fc}(w)\dd w\right|\le C\eta=\frac{C}{N}\,.
\end{align*}
Combination with the claims of Lemma~\ref{lemma helffer} yields the statements. 
\end{proof}
The first statement of Theorem~\ref{theorem density of states}, i.e.,~\eqref{result density of states 1}, now follows easily from the two preceding lemmas.

\subsection{Bulk fluctuations}\label{section bulk fluctuation}
The aim of this section is to prove the second part of Theorem~\ref{theorem density of states}, i.e., Inequality~\eqref{results density of states 2}. Recall the definition of the random variables $\zeta_0$ in~\eqref{definition of zeta0}. Since we will restrict the discussion to the bulk of the spectrum, we may use slightly modified random variables, $\widetilde\zeta_0(z)\equiv\widetilde\zeta^N_0(z)$, approximating~$\zeta_0$ in the bulk that are easier to handle in computations.

\subsubsection{Definition of $\widetilde{\zeta}_0$}\label{definition of widetildezeta0}
We define a random variable $\widetilde{\zeta}_0(z)\equiv\widetilde{\zeta}_0^N(z)$ by
\begin{align}\label{definition of zeta}
 \widetilde{\zeta}_0(z)=(1-R_2(z))^{-1}\left(\frac{1}{N}\sum_{i=1}^N\frac{1}{\lambda v_i-z-m_{fc}(z)}-m_{fc}(z)\right)\,,
\end{align}
for $z\in \C^+$, $\lambda\in\caD_{\lambda_0}$, where $(R_n)$ have been defined in~\eqref{eq.171}. Recall that, $1-R_2(z)\sim \sqrt{\kappa+\eta}$. Hence, by the large deviation estimate ~\eqref{lawlargenumber},
\begin{align}\label{bound zeta_0}
 |\widetilde{\zeta}_0(z)|\le \frac{(\varphi_N)^{c\xi}\lambda}{\sqrt{\kappa_E+\eta}\sqrt{N}}\,,\qquad \qquad(z=E+\ii\eta)\,,
\end{align}
with $(\xi,\nu)$-high probability, for some $c$, uniformly in $z\in\caD_L$ and $\lambda\in\caD_{\lambda_0}$.  Also note that $\widetilde\zeta_0$ approximates the random variables $\zeta_0$ in the bulk: Since $1-R_2\sim 1$ away from the spectral edge, it is straightforward to show that $|\zeta_0(z)-\widetilde\zeta_0(z)|=\caO(N^{-1})$, with high probability for such $z$.

\begin{lemma}\label{lemma bound on widetildexi0}
 Under the assumptions of Theorem~\ref{theorem density of states}, there is $c>0$ such that the event

\begin{align}\label{eq:bound zeta}
\bigcap_{\substack{z\in \caD_L\\ \lambda\in\caD_{\lambda_0}}}\left\{ |m(z)-m_{fc}(z)-\widetilde{\zeta}_0(z)|\le {(\varphi_N)^{c\xi}}\left(\min\left\{\frac{\lambda}{\sqrt{\kappa_E+\eta}}\frac{1}{\sqrt{N}},\frac{\lambda^2}{({\kappa_E+\eta})^{3/2}}\frac{1}{{N}}\right\}+\frac{1}{N\eta}  \right)\right\}\,,\qquad (z=E+\ii\eta)\,,
\end{align}
has $(\xi,\nu)$-high probability.
\end{lemma}
 We omit the proof of this lemma since it is similar to the proof of Theorem~\ref{thm xi0}. Note, however, that the estimate in~\eqref{eq:bound zeta}, deteriorates at the spectral edge and we have to restrict the discussion below mostly to the bulk of the spectrum.

Next, we show that the random variable $\widetilde\zeta_0(z)$, $z=E+\ii \eta$, is a slowly varying function of $E$ (for fixed $\eta$), in the bulk of the spectrum.
\begin{lemma}\label{derivative zeta0}
Under the assumptions of Theorem~\ref{theorem density of states}, the event
\begin{align*}
 \bigcap_{\substack{z\in \caD_L \\ \lambda\in\caD_{\lambda_0}}}\left\{\left|\frac{\partial \widetilde\zeta_o(E+\ii\eta)}{\partial E}\right|\le \frac{(\varphi_N)^{2\xi}\lambda}{(\kappa_E+\eta)^{3/2}}\frac{1}{\sqrt{N}}\right\}\,,
\end{align*}
has $(\xi,\nu)$-high-probability.
\end{lemma}

\begin{proof}
 Recalling the definition of $\widetilde{\zeta}_0$ in~\eqref{definition of zeta}, we compute, for $z\in\caD_L$, $\lambda\in\caD_{\lambda_0}$,
\begin{align*}
 \frac{\partial \widetilde{\zeta}_0(E+\ii \eta)}{\partial E}&=\left(\frac{\partial}{\partial E}\frac{1}{1-R_2(z)}\right)\left(\frac{1}{N}\sum_{i=1}^NQ_{v_i}\frac{1}{\lambda v_i-z-m_{fc}(z)}\right)\\
&\quad\quad+\frac{1}{1-R_2(z)}\left(\frac{1}{N}\sum_{i=1}^NQ_{v_i}\frac{1}{(\lambda v_i-z-m_{fc}(z))^2}\right)(1+m'_{fc}(E+\ii\eta))\,,
\end{align*}
where we abbreviate $m'_{fc}(E+\ii\eta)\equiv\frac{\partial m_{fc}(E+\ii\eta)}{\partial E}$ and $Q_{v_i}\deq\lone-\E_{v_i}$, where $\E_{v_i}$ denotes the partial expectation with respect the random variable $v_i$. Differentiating the functional Equation~\eqref{eq111} for $m_{fc}(z)$, we get
\begin{align*}
 m'_{fc}(E+\ii\eta)&=\int\frac{\dd\mu(v)}{(\lambda v-E-\ii\eta-m_{fc}(E+\ii\eta))^2}\left(1+m'_{fc}(E+\ii\eta)\right)\,,
\end{align*}
hence,
\begin{align*}
|1+m'_{fc}(E+\ii\eta)|=\frac{1}{|1-R_2(E+\ii\eta)|}\le\frac{K}{\sqrt{\kappa_E+\eta}}\,, 
\end{align*}
for some constant $K>1$, where we used Lemma~\ref{lemma.12}. Similarly,
\begin{align*}
 \frac{\partial}{\partial E}\frac{1}{1-R_2(E+\ii\eta)}&=\frac{2(1+m'_{fc}(E+\ii\eta))}{(1-R_2(E+\ii\eta))^2}\int\frac{\dd \mu(v)}{(\lambda v-z-m_{fc}(E+\ii\eta))^3}\\
&=\frac{2(1+m'_{fc}(E+\ii\eta))R_3(E+\ii\eta)}{(1-R_2(E+\ii\eta))^2}\,,
\end{align*}
and hence, by Lemma~\ref{lemma.12},
\begin{align*}
 \left|\frac{\partial}{\partial E}\frac{1}{1-R_2(E+\ii\eta)}\right|\le \frac{C}{(\kappa_E+\eta)^{3/2}}\,,
\end{align*}
for some constant $C$. The terms involving the $Q_{v_i}$ can be bounded by the large deviation estimates~\eqref{lawlargenumber}. Uniformity in $\lambda$ and $z$ follows from a lattice argument using the stability bound~\eqref{stability bound}.
\end{proof}

Next, let $f(x)\equiv f_{E_1,E_2,\eta}(x)$ be an indicator function of the interval $[E_1,E_2]$, smoothed out on scale $\eta=N^{-1}$. Let $\chi(y)$ be a smooth cut-off function as defined in~\eqref{5.42}. We set $m^{\Delta}\deq m-m_{fc}$. Appealing to the discussion in Section~\ref{local density of states}, we define
\begin{align}
\frX_0(E_1,E_2)&\deq \frac{1}{2\pi}\int_{\R^2} { \dd x\, \dd y}\, \big(\ii y f''(x)\chi(y)+\ii(f(x)+\ii y f'(x))\chi'(y)\big)\widetilde{\zeta}_0(x+\ii y)\label{eq.5.71}\,,
\end{align}
where we extend $\widetilde\zeta_0$ to the lower half-plane as $\widetilde\zeta_0(\overline{z})=\overline{\widetilde\zeta_0(z)}$, $z\in\C^+$.
\begin{lemma} \label{bound on frX0}
There is a constant $c>0$ such that, for $E_1<E_2$ with $E_1,E_2\in[-E_0,E_0]$ and for $\lambda\in\caD_{\lambda_0}$, we have
\begin{align}\label{FrX0}
|\frX_0(E_1,E_2)|\le(\varphi_N)^{c\xi}\frac{\caE^2\lambda}{{(\kappa+\caE)^{3/2}}}\frac{1}{\sqrt{N}}\,,
\end{align}
with $(\xi,\nu)$-high probability, where $\caE=\max\{E_2-E_1,\,(\varphi_N)^{L}N^{-1}\}$ and $\kappa=\min\{|E_i-L_i|\,:\, i=1,2\}$.
\end{lemma}
Choosing the energies $E_1$, $E_2$, such that $\min\{\kappa_{E_1},\kappa_{E_2}\}\ge\varkappa$, for some  fixed $\varkappa>0$, we obtain
\begin{align}\label{bound on frX02}
 |\frX_0(E_1,E_2)|\le C_{\varkappa}(\varphi_N)^{c\xi}\frac{\caE^2\lambda}{\sqrt{N}}\,,
\end{align}
with $(\xi,\nu)$-high probability, for some constant $C_\varkappa$, depending on $\varkappa$.
\begin{proof}
Starting from the definition of $\widetilde{\zeta}_0$, we find
\begin{align}\label{eq.10.50}
 |\frX_0(E_1,E_2)| &\le C\left|\int \dd x\int_0^{\infty}\dd y\, f(x)\chi'(y) \widetilde{\zeta}_0(x+\ii y)\right|+\left|\int \dd x\int_0^{\infty}\dd y\,yf'(x)\chi'(y) \widetilde{\zeta}_0(x+\ii y)\right|\\
&\quad\quad+C\left| \int \dd x\int_0^{2\caE}\dd y\, f''(x)\chi(y) y\im \widetilde{\zeta}_0(x+\ii y) \right|\nonumber\,.
\end{align}
To bound the first term on the right side of~\eqref{eq.10.50} we integrate by part in the variable $y$ to find, with $(\xi,\nu)$-high probability,
\begin{align}\label{eq.10.521}
 \left|\int \dd x\,f(x)\int_{\caE}^{2\caE}\dd y\, \chi'(y) \widetilde{\zeta}_0(x+\ii y)\right|&=\left|\int \dd x\,f(x)\int_{\caE}^{2\caE}\dd y\, \chi(y)\partial_y\widetilde{\zeta}_0(x+\ii y)  \right|\nonumber\\
&=\left|\int \dd x\,f(x)\int_{\caE}^{2\caE}\dd y\,\chi(y) \partial_x\widetilde{\zeta}_0(x+\ii y)\right|\nonumber\\
&\le (\varphi_N)^{c\xi}\caE\left|\int \dd x\,f(x)\frac{\lambda}{{(\kappa_x+\caE)^{3/2}}}\frac{1}{\sqrt{N}}  \right|\nonumber\\
&\le (\varphi_N)^{c\xi}\frac{\caE\lambda}{(\kappa+\caE)^{3/2}}\frac{1}{\sqrt{N}}\,,
\end{align}
where we used in the second line that $\widetilde{\zeta}_0(z)$ in an analytic function in the upper half plane, and in the third line we used Lemma~\ref{derivative zeta0}. In the fourth line we used that $\kappa_x\ge \kappa=\min\{\kappa_{E_1},\kappa_{E_2}\}$. Finally, we used that $f$ is supported on $[E_1-\eta,E_2+\eta]$, $\eta=N^{-1}$.

To bound the second term on the right side of~\eqref{eq.10.50}, we integrate by part in the variable $x$ and find, similarly to the computation above,
\begin{align}\label{eq.10.52}
 \left|\int_{E_1-\eta}^{E_2+\eta} \dd x\int_0^{\infty}\dd y\,yf'(x)\chi'(y) \widetilde{\zeta}_0(x+\ii y)\right|&=\left|\int_{E_1-\eta}^{E_2+\eta} \dd x\int_0^{\infty}\dd y\,yf(x)\chi'(y)\partial_x\widetilde{\zeta}_0(x+\ii y)\right|\nonumber\\
&\le(\varphi_N)^{c\xi}\int_{E_1-\eta}^{E_2+\eta} \dd x\int_{\caE}^{2\caE}\dd y\,yf(x)|\chi'(y)|\frac{\lambda}{(\kappa_x+\caE)^{3/2}}\frac{1}{\sqrt{N}}\nonumber\\
&\le(\varphi_N)^{c\xi}\frac{\lambda}{(\kappa+\caE)^{3/2}}\frac{1}{\sqrt{N}}\int_{E_1-\eta}^{E_2+\eta} \dd x\,f(x)\int_{\caE}^{2\caE}\dd y\,\frac{y}{\caE}\nonumber\\
&\le(\varphi_N)^{c\xi}\frac{\caE^2\lambda}{(\kappa+\caE)^{3/2}}\frac{1}{\sqrt{N}}\,,
\end{align}
with $(\xi,\nu)$-high probability.

Finally, the third term in~\eqref{eq.10.50} can be bounded by integrating by parts in $x$ to obtain
\begin{align}
 \left| \int \dd x\int_0^{2\caE}\dd y\, f''(x)\chi(y) y\im \widetilde{\zeta}_0(x+\ii y) \right|&=\left| \int \dd x\int_0^{2\caE}\dd y\, f'(x)\chi(y) y\im \partial_x\widetilde{\zeta}_0(x+\ii y) \right|\nonumber\\&\le (\varphi_N)^{c\xi}\left|\int_{0}^{2\caE}\dd y  \frac{\lambda y}{(\kappa+y)^{3/2}}\frac{1}{\sqrt{N}}  \right|\nonumber\\
&\le(\varphi_N)^{c\xi}\frac{\caE^2\lambda}{(\kappa+\caE)^{3/2}}\frac{1}{\sqrt{N}}\label{eq.10.53}\,,
\end{align}
with $(\xi,\nu)$-high probability. Adding up the estimates~\eqref{eq.10.521},~\eqref{eq.10.52} and~\eqref{eq.10.53} yields the claim.
\end{proof}

\subsubsection{Local eigenvalue density in the bulk}
In this subsection, we show that we can control the difference $\frn(E_2)-\frn(E_1)$ in terms of $n_{fc}(E_2)-n_{fc}(E_1)$ in the bulk of the spectrum up to an optimal error: Fix some $\varkappa>0$. We consider energies $E_1<E_2$, such that $\min\{\kappa_{E_1},\kappa_{E_2}\}\ge\varkappa$, $L_1<E_1<E_2<L_2$ and $E_2-E_1\ge(\varphi_N)^{L\xi}N^{-1}$. We denote with $C_{\varkappa}$ constants that only depend on $\varkappa$ (with $C_{\varkappa}\to\infty$, as $\varkappa\to 0$).

As above, let $f(x)\equiv f_{E_1,E_2,\eta}(x)$ be an indicator function of the interval $[E_1,E_2]$, smoothed out on scale $\eta=N^{-1}$. Let $\chi(y)$ be a smooth cut-off function as defined in~\eqref{5.42} and let $m^{\Delta}\deq m-m_{fc}$. Define
\begin{align}
\frX_1(E_1,E_2)&\deq \frac{1}{2\pi}\int_{\R^2}\big(\ii y f''(x)\chi(y)+\ii(f(x)+\ii y f'(x))\chi'(y)\big)m^{\Delta}(x+\ii y)\,,\label{eq.5.70}
\end{align}
and recall the definition of $\frX_0$ in~\eqref{eq.5.71},
\begin{align}
 \frX_0(E_1,E_2)&\deq \frac{1}{2\pi}\int_{\R^2}\big(\ii y f''(x)\chi(y)+\ii(f(x)+\ii y f'(x))\chi'(y)\big){ \widetilde{\zeta}_0}(x+\ii y)\,.
\end{align}
Here we implicitly assume that the functions $f$ and $\chi$ in both definitions agree.

Following the discussion in Section~\ref{local density of states}, one easily sees that
\begin{align}\label{estimate density frX}
 \left|\left(\frn(E_1,E_2)-{n_{fc}}(E_1,E_2)\right)-\frX_1(E_1,E_2)   \right|\le\frac{(\varphi_N)^{c\xi}}{N}\,,
\end{align}
with $(\xi,\nu)$-high probability. Recalling the estimate on $\frX_0$ in~\eqref{FrX0}, we observe that it suffices to bound $\frX_1-\frX_0$ in order to control the density of states.  
\begin{lemma}\label{lemma bound on frX10}
 Let $L_1<E_1<E_2<L_2$, with $\min\{\kappa_{E_1},\kappa_{E_2}\}\ge\varkappa$ and $E_2-E_1\ge (\varphi_N)^{L\xi}N^{-1}$. Then 
\begin{align}\label{eq.10.54}
|\frX_1(E_1,E_2)-\frX_0(E_1,E_2)|\le C_{\varkappa}(\varphi_N)^{c\xi} \frac{1}{N} \,,
\end{align}
with $(\xi,\nu)$-high probability. The constant $c>0$ can be chosen independent of $E_1,E_2$ and $\lambda\in\caD_{\lambda_0}$.
\end{lemma}
\begin{proof}
We set $y_0\deq(\varphi_N)^L N^{-1}$ and abbreviate $\widetilde\zeta(z)=m(z)-m_{fc}(z)-\widetilde{\zeta}_0(z)$. Using the definition of $\widetilde{\zeta}_0$, we find
\begin{align}\label{eq.10.70}
&|(\frX_1-\frX_0)(E_1,E_2)| \nonumber\\
&\le C\int \dd x\int_0^{\infty}(|f(x)|+|yf'(x)|)|\chi'(y)|\,|\widetilde\zeta(x+\ii y)|+C\left| \int \dd x\int_0^{y_0}\dd y f''(x)\chi(y) y\im \widetilde{\zeta}_0(x+\ii y) \right|\nonumber\\
&\quad+C\left| \int \dd x\int_0^{y_0}\dd y f''(x)\chi(y) y\im m^{\Delta}(x+\ii y) \right|+C\left|\int \dd x\int_{y_0}^{\infty}\dd y f''(x)\chi(y) y\im \widetilde\zeta(x+\ii y) \right|\,.
\end{align}
Using~\eqref{eq:bound zeta} we can bound the first term on the right side of~\eqref{eq.10.70} as
\begin{align}\label{eq.10.71}
\left|\int \dd x\int_{\caE}^{2\caE}\big( |f(x)|+y|f'(x)|\big)|\chi'(y)|\, |\widetilde\zeta(x+\ii y)| \right|&\le C_{\varkappa}\frac{(\varphi_N)^{c\xi}}{\caE}\int \dd x\int_{\caE}^{2\caE}\dd y (|f(x)|+ y |f'(x)|) \frac{1}{N y}\nonumber\\ 
&\le C_{\varkappa}(\varphi_N)^{c\xi} \frac{1}{N}\,,
\end{align}
with $(\xi,\nu)$-high probability.

The second term on the right side of~\eqref{eq.10.70} is, by~\eqref{bound zeta_0}, bounded by
\begin{align}\label{eq.10.72}
\left| \int \dd x\int_0^{y_0}\dd y f''(x)\chi(y) y\im \widetilde{\zeta}_0(x+\ii y) \right|&\le (\varphi_N)^{c\xi}\frac{\lambda}{\eta}\int_0^{y_0}\dd y \,\chi(y)y\frac{1}{\sqrt{\kappa+y}}\frac{1}{\sqrt{N}}\nonumber\\ &\le
(\varphi_N)^{c\xi}\frac{1}{\eta}\int_0^{y_0}\dd y \frac{\lambda\sqrt{y}}{\sqrt{N}}\nonumber\\
 &  \le \frac{(\varphi_N)^{c\xi}}{N}\,,
 \end{align}
with $(\xi,\nu)$-high probability.

To control the third term, we note that both functions $y\mapsto y \,\im m(x+\ii y)\,,y\,\im {m_{fc}}(x+\ii y)$, are monotone increasing. Thus we get from~\eqref{eq.5.38},
\begin{align*}
 y \im m(x+\ii y)\le y_0  \im m(x+\ii y_0)\le (\varphi_N)^{c\xi}y_0\left( \sqrt{\kappa_x+y_0}+\frac{\lambda^{1/2}}{N^{1/4}}+\frac{1}{N y_0}\right)\le\frac{ (\varphi_N)^{c\xi}}{N}\,,\quad\quad (y\le y_0)\,,
\end{align*}
and
\begin{align*}
y \im m_{fc}(x+\ii y)\le  y_0  \im {m_{fc}}(x+\ii y_0)\le C y_0 \sqrt{\kappa_x+y_0}\,,\quad\quad (y\le y_0)\,.
\end{align*}
Since $y_0=(\varphi_N)^LN^{-1}$, this yields
\begin{align}\label{eq.10.73}
 y\, |\im \widetilde\zeta(x+\ii y)|\le \frac{(\varphi_N)^{c\xi}}{{N}}\,,\quad\quad (y\le y_0)\,,
\end{align}
with $(\xi,\nu)$-high probability. The third term on the right side of~\eqref{eq.10.70} is thus bounded as
\begin{align}\label{eq.10.74}
\left| \int \dd x\int_0^{y_0}\dd y f''(x)\chi(y) y\im m^{\Delta}(x+\ii y) \right|\le \frac{(\varphi_N)^{c\xi}}{N}\int \dd x |f''(x)|\,\int_0^{y_0} \dd y\, \chi(y)\le \frac{(\varphi_N)^{c\xi}}{N}\,,
\end{align}
with $(\xi,\nu)$-high probability.

To bound the fourth term in~\eqref{eq.10.70}, we integrate first by parts in the variable $x$ and then in $y$, to find the bound
\begin{align}\label{eq 553}
 \left|\int \dd x\int_{y_0}^{2\caE}\dd y f'(x)\partial_y(\chi(y)y)\re\widetilde\zeta(x+\ii y)\right|+\left|\int \dd x f'(x)\chi(y_0)y_0\re\widetilde\zeta(x+\ii y_0) \right|\,.
\end{align}
Using the a priori high probability bounds
\begin{align*}
 |m^{\Delta}(x+\ii y_0)|\le (\varphi_N)^{c\xi}\left(\frac{\lambda^{1/2}}{{N}^{1/4}}+\frac{1}{{ Ny_0}   }\right)\le  (\varphi_N)^{c\xi}\,,\quad\quad |\widetilde{\zeta}_0(x+\ii y_0)|\le(\varphi_N)^{c\xi}\frac{\lambda}{\sqrt{\kappa_x+y_0}}\frac{1}{\sqrt{N}}\le (\varphi_N)^{c\xi}\,,
\end{align*}
we bound the second term on the right side of~\eqref{eq 553} as
\begin{align*}
 \left|\int\dd xf'(x)y_0 \widetilde\zeta(x+\ii \eta)  \right|\le (\varphi_N)^{c\xi} y_0\le \frac{(\varphi_N)^{c\xi}}{N}\,,
\end{align*}
with $(\xi,\nu)$-high probability. It remains to bound the first term in~\eqref{eq 553},
\begin{align}\label{bound nochmals}
 \left|\int \dd x\int_{y_0}^{2\caE}\dd y f'(x)\partial_y(\chi(y)y)\re\widetilde\zeta(x+\ii y)\right|&\le\left|\int \dd x\int_{y_0}^{2\caE}\dd y f'(x)\chi'(y)y\re\widetilde\zeta(x+\ii y)\right|\nonumber\\
&\quad\quad+ \left|\int \dd x\int_{y_0}^{2\caE}\dd y f'(x)\chi(y)\re\widetilde\zeta(x+\ii y)\right|\,.
\end{align}
For the first term on the right side, we use~\eqref{eq:bound zeta} to find 
\begin{align*}
\left|\int \dd x\int_{y_0}^{2\caE}\dd y f'(x)\chi'(y)y\re\widetilde\zeta(x+\ii y)\right|\le C_{\varkappa}(\varphi_N)^{c\xi} \frac{1}{N}\,,
\end{align*}
with $(\xi,\nu)$-high probability. Using once more~\eqref{eq:bound zeta}, we bound the second term on the right side of~\eqref{bound nochmals} as
\begin{align}\label{eq.10.76}
 \left|\int\dd x \int _{y_0}^{\infty}\dd y f'(x)\chi(y)\re \widetilde\zeta(x+\ii y)\right| \le C_{\varkappa} (\varphi_N)^{c\xi}\int_{y_0}^{2\caE}\dd y \frac{1}{\caE N} \leq (\varphi_N)^{c\xi} \frac{1}{N}\,,
\end{align}
with $(\xi,\nu)$-high probability. Adding up the different contributions, we find~ \eqref{eq.10.54}.
\end{proof}
To conclude this subsection, we prove~\eqref{results density of states 2} of Theorem~\ref{theorem density of states}:
\begin{proof}[Proof of~\eqref{results density of states 2}]
 Let $E_1<E_2$. Then we have from~\eqref{estimate density frX}
\begin{align*}
 |\frn(E_1,E_2)-n_{fc}(E_1,E_2)|&\le|\frX_1(E_1,E_2)|+C(\varphi_N)^{c\xi}\frac{1}{N}\\
&\le|\frX_0(E_1,E_2)|+|\frX_1(E_1,E_2)-\frX_0(E_1,E_2)|+C(\varphi_N)^{c\xi}\frac{1}{N}\,,
\end{align*}
with $(\xi,\nu)$-high probability. Using Lemma~\ref{bound on frX0} and Lemma~\ref{lemma bound on frX10}, we therefore get
\begin{align*}
 |\frn(E_1,E_2)-n_{fc}(E_1,E_2)|\le C_{\varkappa}(\varphi_N)^{c\xi}\left(\frac{1}{N} +\frac{\lambda^2\caE^2}{\sqrt{N}}\right)\,,
\end{align*}
with $(\xi,\nu)$-high probability. Inequality~\eqref{results density of states 2} follows by choosing $E_2-E_1\ge(\varphi_N)^{L\xi}N^{-1}$.
\end{proof}

\subsection{Eigenvalue spacing in the bulk}\label{Eigenvalue spacing in the bulk}
In this subsection, we prove Theorem \ref{rigidity of eigenvalue spacing}.

\begin{proof}[Proof of Theorem \ref{rigidity of eigenvalue spacing}]
Let $\lambda\in\caD_{\lambda_0}$. Starting from the identity
\begin{align*}
 \frac{i-j}{N}=\frn(\mu_i)-\frn(\mu_j)\,,
\end{align*}
we obtain from \eqref{results density of states 2} that
\begin{align*}
 n_{fc}(\mu_i)-n_{fc}(\mu_j)=\frac{i-j}{N}+\caO\left((\varphi_N)^{c\xi} \frac{(\mu_i-\mu_j)^2}{\sqrt{N}}\right)+\caO\left((\varphi_N)^{c\xi}\frac{1}{N} \right)\,,
\end{align*}
with $(\xi,\nu)$-high probability, for some $c$ large enough. Then, using $n_{fc}(\mu_i)-n_{fc}(\mu_j)=(\mu_i-\mu_j) n_{fc}'(\mu_i')$, for some $\mu_i' \in[\mu_i,\mu_j]$,
\begin{align} \label{eq.5.83}
\mu_i - \mu_j = \frac{i-j}{N \rho_{fc}(\mu_i')} + \caO\left((\varphi_N)^{c\xi} \frac{(\mu_i-\mu_j)^2}{\sqrt{N}}\right)+\caO\left((\varphi_N)^{c\xi}\frac{1}{N} \right)\,,
\end{align}
where we used that $n_{fc}'(\mu_i') = \rho_{fc}(\mu_i') > 0$ and $1/C' < \rho_{fc} < C'$ in the bulk for some constant $C' >1$, depending on $\lambda$ and $\mu$. Since $|\mu_i - \mu_j| =\caO(1)$, we have
\begin{align*}
(\varphi_N)^{c\xi}\frac{(\mu_i-\mu_j)^2}{\sqrt{N}} \ll |\mu_i - \mu_j|\,,
\end{align*}
which shows that the second term in the right side of \eqref{eq.5.83} can be absorbed into the left side. Similarly, the last term on the right side can be absorbed into the first term in the right side, as we can see from the condition $|i-j| \gg (\varphi_N)^{c\xi}$. Thus,
\begin{align} \label{eq.5.86}
C_1 \frac{|i-j|}{N} \leq |\mu_i - \mu_j| \leq C_2 \frac{|i-j|}{N}\,,
\end{align}
with $(\xi,\nu)$-high probability for some constants $C_1, C_2 > 0$. This proves the first part of the theorem.

If $|i-j| \leq(\varphi_N)^{c\xi} N^{1/2}$, we find from \eqref{eq.5.86} that $|\mu_i - \mu_j| \leq C_2 ( \varphi_N)^{c\xi} N^{-1/2}$. In this case,
\begin{align*}
(\varphi_N)^{c\xi}\frac{(\mu_i-\mu_j)^2}{\sqrt{N}} \ll\frac{1}{ N}\,,
\end{align*}
with high probability. Furthermore, since $|\mu_i' - \mu_i| \leq |\mu_i - \mu_j|$ and since $\rho_{fc}$ is Lipschitz continuous inside $\supp \mu_{fc}$ (see, e.g.,~\cite{B}), we get
\begin{align} \label{eq.rho_fc}
|\rho_{fc}(\mu_i') - \rho_{fc}(\mu_i)| \leq (\varphi_N)^{K \xi} \frac{1}{\sqrt{N}}\,,
\end{align}
hence
\begin{align*}
\left| \frac{i-j}{N \rho_{fc}(\mu_i')} - \frac{i-j}{N \rho_{fc}(\mu_i)} \right| \leq C \frac{|i-j|}{N} \frac{|\rho_{fc}(\mu_i') - \rho_{fc}(\mu_i)|}{\rho_{fc}(\mu_i') \rho_{fc}(\mu_i)} \leq (\varphi_N)^{K \xi} N^{-1}\,.
\end{align*}
with $(\xi,\nu)$-high probability, for some constant $K$. Thus, we obtain that
\begin{equation*}
\left| |\mu_i - \mu_j| - \frac{|i-j|}{N \rho_{fc} (\mu_i)} \right| \leq (\varphi_N)^{K \xi} N^{-1}\,,
\end{equation*}
with $(\xi,\nu)$-high probability, proving the second part the theorem.
\end{proof}
\newpage

\subsection{Integrated density of states and rigidity of eigenvalues}\label{section: Integrated density of states and rigidity of eigenvalues}
The goal of this subsection is to prove Theorems~\ref{integrated density of states} and~\ref{rigidity of eigenvalues}. The proofs follow closely~\cite{EKYY1}.
\subsubsection{Estimate on $\|H\|$}
 As a first step, we need an estimate on the operator norm of $H=\lambda V +W$. We have the following result:
\begin{lemma}\label{estimate norm H}
There is a constant $c_0>0$, such that for all $\lambda\in\caD_{\lambda_0}$, we have
\begin{align*}
 \| H\|\le \max\{|L_1|,L_2\}+(\varphi_N)^{c_0\xi}\left(\frac{\lambda}{\sqrt{N}}+\frac{1}{N^{2/3}}\right)\,,
\end{align*}
with $(\xi,\nu)$-high probability.
\end{lemma}

\begin{proof}
 We will only consider the largest eigenvalue $\mu_N$. A bound on the lowest eigenvalue $\mu_1$ is obtained in a similar way. From the strong local law~\eqref{strong1}, we  get
\begin{align*}
 \Lambda(z)\le (\varphi_N)^{c\xi}\left(\frac{\lambda^{1/2}}{N^{1/4}}+\frac{1}{N\eta}\right)\,,\quad\quad z\in\caD_L\,,\quad \lambda\in\caD_{\lambda_0}\,,
\end{align*}
with $(\xi+2,\nu)$-high probability. Then we can apply Lemma~\ref{lemma.22}, with
\begin{align*}
 \gamma(z)\deq(\varphi_N)^{c\xi}\left(\frac{\lambda^{1/2}}{N^{1/4}}+\frac{1}{N\eta}\right)\,,
\end{align*}
to get, for some sufficiently large constant $c_1$,
\begin{align*}
 \left|(1-R_2)[\v]- R_3 [\v]^2\right| \le C\frac{\Lambda^2}{\log N} +C( \varphi_N)^{c_1\xi}\left(\frac{\lambda}{\sqrt{N}}+\frac{\im\mfc(z)+\gamma(z)}{N\eta}\right)\,,
\end{align*}
with $(\xi,\nu)$-high probability, for any $z\in\caD_L$ and $\lambda\in\caD_{\lambda_0}$. Now, if $E>L_2$ and $\kappa\ge\eta$, we have $\im\mfc(z)\sim \eta/\sqrt{\kappa}$ and 
\begin{align*}
 \alpha\deq |1-R_2| \sim\sqrt{\kappa}\,,
\end{align*}
by the Lemmas~\ref{lemma.12} and~\ref{lemma.17}. Thus, we obtain, upon using Young's inequality,
\begin{align}\label{self outside}
 \left|(1-R_2)[\v]- R_3 [\v]^2\right| \le C\frac{\Lambda^2}{\log N} +C (\varphi_N)^{c_1\xi}\left(\frac{\lambda}{\sqrt{N}}+\frac{1}{(N\eta)^2}+\frac{1}{N\sqrt{\kappa}}\right)\,,
\end{align}
with $(\xi,\nu)$-high probability, for some $c_1$ sufficiently large.

 Given $c_1$, it is straightforward to check that there is a constant $c_2>2c_1$, such that, for any $E$ satisfying
\begin{align}\label{condition on E}
 L_2+(\varphi_N)^{c_2\xi}\left(\frac{\lambda}{\sqrt{N}}+\frac{1}{N^{2/3}}\right)\le E\le E_0\,,
\end{align}
we have
\begin{align}\label{conse of E}
 \min\{N^{-1/2}\kappa^{1/4}, N^{-1/2}\lambda^{-1}\kappa^{1/2},\kappa \}\ge(\varphi_N)^{c_1\xi+2}\frac{1}{N\sqrt{\kappa}}\,.
\end{align}
We assume now that $E$ satisfies~\eqref{condition on E} and set
\begin{align*}
 \eta\equiv \eta_E\deq (\varphi_N)^{c_1\xi+1}\frac{1}{N\sqrt{\kappa}}\,.
\end{align*}
Note that $z=E+\ii\eta\in\caD_{L}$. From~\eqref{conse of E}, we have $\kappa\ge\eta$. Similarly, we have
\begin{align}\label{kappa eta relation}
 \im\mfc(E+\ii\eta)\sim\frac{\eta}{\sqrt\kappa}\ll\frac{1}{N\eta}\,;\quad\quad \frac{\lambda}{\sqrt{\kappa N}}\ll\frac{1}{N\eta}\,.
\end{align}
Furthermore, since $\alpha\ge\sqrt{\kappa}/K$, for some $K>1$, we must have 
\begin{align}\label{alpha bound on lambda}
2|R_3| \Lambda(z)\le C(\varphi_N)^{c_1\xi+1}\left(\frac{\lambda^{1/2}}{{N}^{1/4}}+\frac{1}{N\eta}\right)\le\alpha\,,
\end{align} 
with $(\xi,\nu)$-high probability. Here, we used that $N\eta\sqrt{\kappa}=(\varphi_N)^{c_1\xi+1}$ and $\sqrt\kappa\ge (\varphi_N)^{c_2\xi/2}\lambda^{1/2}N^{-1/4}$. Since $\alpha\ge 2|R_3|\Lambda$, we get from~\eqref{self outside},
\begin{align}\label{self outside 2}
 \Lambda\le C(\varphi_N)^{c_1\xi}\left(\frac{\lambda}{\alpha \sqrt{N}}+\frac{1}{\alpha (N\eta)^2}+\frac{1}{\alpha N\sqrt{\kappa}}\right)\,,
\end{align}
with $(\xi,\nu)$-high probability. Since $\alpha\ge \sqrt{\kappa}/K$, we obtain from~\eqref{kappa eta relation}, 
\begin{align*}
 (\varphi_N)^{c_1\xi+1}\frac{1}{\alpha N\sqrt{\kappa}}\le C\frac{\eta}{\sqrt{\kappa}} \ll \frac{1}{N\eta}\,.
\end{align*}
The second term on the right side of~\eqref{self outside 2}, can be bounded by using~\eqref{alpha bound on lambda}. The first term on the right side of~\eqref{self outside 2} is estimated by using the second inequality in~\eqref{kappa eta relation}. Thus, for any $E$ satisfying~\eqref{condition on E}, $z\in\caD_L$, and any $\lambda\in\caD_{\lambda_0}$, we obtain that
\begin{align*}
 \Lambda(z)\ll\frac{1}{N\eta}\,,
\end{align*}
with $(\xi,\nu)$-high probability. Thus
\begin{align}\label{no eigenvalue here}
 \im m(z)\le\im\mfc(z)+\Lambda(z)\ll\frac{1}{N\eta}\,,
\end{align}
with $(\xi,\nu)$-high probability, for such $E$. By the spectral decomposition of $H$, we have
\begin{align*}
 \im m(z)=\frac{1}{N}\sum_{\alpha=1}^N\frac{\eta}{(\mu_\alpha-E)^2+\eta^2}\,,
\end{align*}
and we conclude that 
\begin{align}\label{condition for eigenvalue}
 \im m(z)\ge\frac{C}{N\eta}\,,
\end{align}
for some $C>0$, if there is an eigenvalue in the interval $[E-\eta,E+\eta]$. Thus~\eqref{no eigenvalue here}, implies, for any $E$, satisfying~\eqref{condition on E}, that there is no eigenvalue in the interval $[E-\eta,E+\eta]$, with $(\xi,\nu)$-high probability.

To cover energies $E\ge E_0$, we use the following result: For a Wigner matrix $W$ satisfying the assumptions in Definition~\ref{assumption wigner} we have
\begin{align}
 \|W\|\le 2+\frac{(\varphi_N)^{\xi}}{N^{1/4}}\,,
\end{align}
with $(\xi,\nu)$-high probability. We refer, e.g., to Lemma~4.3. in~\cite{EKYY1}. Spectral perturbation theory then implies  $\|H\|\le \|\lambda V\|+\|W\|\le 2+\frac{(\varphi_N)^{\xi}}{N^{1/4}}+\lambda$, with $(\xi,\nu)$-high probability, covering the regime $E\ge E_0$. This concludes the proof.

\end{proof}

\subsubsection{Integrated density of states}
In this subsection, we prove Theorem~\ref{integrated density of states}. Given the results on $\frn(E_1,E_2)$ in Theorem~\ref{theorem density of states} and the estimate on~$\| H\|$ this is straightforward:
\begin{proof}[Proof of Theorem \ref{integrated density of states}]
We assume that $E$ is such that $|E-L_1|\le |E-L_2|$. The other case is dealt with in the same way. Set
\begin{align}\label{the chosen E}
 E_1=L_1-(\varphi_N)^{c_1\xi}\left(\frac{\lambda}{\sqrt{N}}+\frac{1}{N^{2/3}}\right)\,,
\end{align}
with some $c_1$ large enough, such that ${ n_{fc}}(E_1)=0$ and $\frn(E_1)=0$ with $(\xi,\nu)$-high probability; see Lemma~\ref{estimate norm H}. 

Next, choose $E\ge E_1$, then from~\eqref{result density of states 1}, we get, setting $E_2=E$ and bounding $\caE\le E-E_1+(\varphi_N)^{L\xi}N^{-1}$,
\begin{align*}
 \left| \frn(E)-n_{fc}(E)\right|\le (\varphi_N)^{c\xi}\left(\frac{1}{N}+\frac{\lambda}{\sqrt{N}}\sqrt{E-E_1+(\varphi_N)^{L\xi}N^{-1}}\right)\,.
\end{align*}
with $(\xi,\nu)$-high probability. Using our assumption on $E$ and~\eqref{the chosen E}, we get
\begin{align*}
 \left| \frn(E)-n_{fc}(E)\right|&\le (\varphi_N)^{c\xi}\left(\frac{1}{N}+\frac{\lambda^{3/2}}{N^{3/4}}+\frac{\lambda}{N^{5/6}}+\frac{\lambda\sqrt{\kappa_E}}{\sqrt{N}} \right)\,,
\end{align*}
with $(\xi,\nu)$-high probability, for some $c_2$ large enough. This estimate holds for any $E$ and $\lambda$. Uniformity is obtained with a lattice argument, we omit the details.
\end{proof}

\subsubsection{Rigidity of eigenvalues}
In this subsection, we prove Theorem~\ref{rigidity of eigenvalues}. Recall the definition of the classical location $\gamma_\alpha$ of the eigenvalue $\mu_\alpha$ in~\eqref{classical location}. 
\begin{lemma}\label{rigidity lemma}
There exists a constant $C$, such that, for all $\lambda\in\caD_{\lambda_0}$, the following statements hold with $(\xi,\nu)$-high probability for some large enough $c$,
\begin{itemize}
 \item [$i.$] if $\max\{\kappa_{\gamma_{\alpha}},\kappa_{\mu_{\alpha}}\}\le(\varphi_N)^{c\xi}(\frac{\lambda}{\sqrt{N}}+\frac{1}{N^{2/3}})$, then
\begin{align*}
 |\mu_{\alpha}-\gamma_{\alpha}|\le(\varphi_N)^{C\xi}\left(\frac{\lambda}{\sqrt{N}}+\frac{1}{N^{2/3}}\right)\,;
\end{align*}
\item[$ii.$] if $\max\{\kappa_{\gamma_{\alpha}},\kappa_{\mu_{\alpha}}\}\ge(\varphi_N)^{c\xi}(\frac{\lambda}{\sqrt{N}}+\frac{1}{N^{2/3}})$, then
\begin{align*}
|\mu_{\alpha}-\gamma_{\alpha}|\le(\varphi_N)^{C\xi}\left(\frac{\lambda}{\sqrt{N}}+\frac{1}{\widehat\alpha^{1/3}N^{2/3}}+ \frac{\lambda^2}{N^{1/3} \widehat\alpha^{2/3}}\right)\,,
\end{align*}

\end{itemize}
where $\widehat\alpha\deq\min\{\alpha,N-\alpha\}$.
\end{lemma}

\begin{proof}
We will focus on the eigenvalues $\mu_1,\ldots,\mu_{N/2}$. The other eigenvalues can be treated in a similar way. Define an event $\Xi$ as the intersection of the events on which the estimates
\begin{align}\label{norm H nochmals}
 \|H\|\le \max\{|L_1|,L_2\}+(\varphi_N)^{C_0\xi}\left(\frac{\lambda}{\sqrt{N}}+\frac{1}{N^{2/3}}\right)\,,
\end{align}
(see Lemma~\ref{estimate norm H}), and
\begin{align*}
|\frn(E)-n_{fc}(E)|\le(\varphi_N)^{C_0\xi}\left(\frac{1}{N}+ \frac{\lambda^{3/2}}{N^{3/4}}+\frac{\lambda}{N^{5/6}}+\frac{\lambda\sqrt{\kappa_E}}{\sqrt{N}}\right)\,,
\end{align*}
(see Theorem~\ref{integrated density of states}), hold, for any $\lambda\le \lambda_0$ and $|E|\le E_0$. We note that on $\Xi$, we have $\mu_{N/2}\le K$, for some $K<L_2$.

Let $C_1>C_0$. We use the dyadic decomposition $$\{1,\ldots, N/2\}=\bigcup_{k=0}^{2\log N} U_{k}\,,$$ where
\begin{align*}
 U_0&\deq\left\{\alpha\le N/2\,:\, |L_1|+\max\{\mu_{\alpha},\gamma_{\alpha}\}\le 2(\varphi_N)^{C_1\xi} \lambda N^{-1/2}\right\}\,,\\
U_k&\deq\left\{\alpha\le N/2\,:\, 2^k\lambda(\varphi_N)^{C_1\xi} N^{-1/2}\le|L_1|+\max\{\mu_{\alpha},\gamma_{\alpha}\}\le 2^{k+1}(\varphi_N)^{C_1\xi}\lambda N^{-1/2}\right\}\,,\quad\quad (k\ge 1)\,.
\end{align*}
By the definition of $U_0$ and~\eqref{norm H nochmals}, we have
\begin{align*}
 |\mu_\alpha-\gamma_\alpha|\le (\varphi_N)^{C\xi}\frac{\lambda}{N^{1/2}}\,,
\end{align*}
on $\Xi$, for $\alpha\in U_0$. 

For $k\ge1$, we find on $\Xi$ that
\begin{align}\label{alpha N}
\frac{\alpha}{N}=n_{fc}(\gamma_\alpha)=\frn(\mu_{\alpha})=n_{fc}(\mu_\alpha)+(\varphi_N)^{{ C_0} \xi}\caO\left(\frac{1}{N}+ \frac{\lambda^{3/2}}{N^{3/4}}+\frac{\lambda}{N^{5/6}} +\frac{\lambda\sqrt{\kappa_E}}{\sqrt{N}}\right)\,.
\end{align}
On $\Xi$, and for $\alpha\in U_{k}$, we can bound the second term on the right side of the above equation as
\begin{align*}
&(\varphi_N)^{ C_0\xi}\caO\left(\frac{1}{N}+\frac{\lambda^{3/2}}{N^{3/4}}+\frac{\lambda}{N^{5/6}}+\frac{\lambda\sqrt{\kappa_E}}{\sqrt{N}}\right) \\
&\le C(\varphi_N)^{C_0\xi}\left(\frac{1}{N}+\frac{\lambda^{3/2}}{N^{3/4}}+\frac{\lambda}{N^{5/6}} \right)+C 2^{(k+1)/2}(\varphi_N)^{({C_0}+{C_1}/2)\xi}\frac{\lambda^{3/2}}{N^{3/4}}\,,
\end{align*}
where we used $\kappa_{\mu_\alpha}\le |L_1|+\mu_{\alpha}$. Furthermore, we have on $\Xi$, for $\alpha\in U_k$,
\begin{align*}
 n_{fc}(\gamma_{\alpha})+n_{fc}(\mu_{\alpha})\ge c 2^{3k/2}(\varphi_N)^{3{ C_1}\xi/2}\lambda^{3/2}N^{-3/4}\,,
\end{align*}
where we used $n_{fc}(L_1+x)\sim x^{3/2}$, for $0\le x\le |L_1|+K$. Thus
\begin{align*}
(\varphi_N)^{{ C_0}\xi}\caO\left(\frac{1}{N}+\frac{\lambda^{3/2}}{N^{3/4}}+\frac{\lambda}{N^{5/6}}+\frac{\lambda\sqrt{\kappa_E}}{\sqrt{N}}\right) \ll n_{fc}(\gamma_{\alpha})+n_{fc}(\mu_{\alpha})\,,
\end{align*}
which implies by~\eqref{alpha N} that
\begin{align*}
 n_{fc}(\mu_\alpha)=n_{fc}(\gamma_{\alpha})\left(1+\caO\left((\varphi_N)^{-({ C_1}-{ C_0})\xi}\right)\right)\,,
\end{align*}
on $\Xi$, for $\alpha\in U_k$. Using that $n_{fc}'(x)\sim (n_{fc}{ (x)})^{1/3} \sim (|L_1|+x)^{ 1/2}$, for $L_1\le x\le K$, we have $|L_1|+\gamma_\alpha\sim |L_1|+\mu_\alpha$. Hence
\begin{align*}
 n'_{fc}(x)\sim n_{fc}'(\gamma_{\alpha})\,,
\end{align*}
for any $x$ between $\gamma_\alpha$ and $\mu_\alpha$. Recalling that the density $\rho_{fc}(x)$ is continuous, we conclude that, on $\Xi$, for $\alpha\in U_k$,
\begin{align}
 |\mu_\alpha-\gamma_\alpha|&\le C\frac{|n_{fc}(\mu_\alpha)-n_{fc}(\gamma_\alpha) |}{n'_{fc}(\gamma_{\alpha})}\nonumber\\
&\le\frac{C(\varphi_N)^{{ C_0}\xi}}{(\alpha/N)^{1/3}} \left(\frac{1}{N}+\frac{\lambda^{3/2}}{N^{3/4}}+\frac{\lambda}{N^{5/6}}+\frac{\lambda\sqrt{\kappa_{\mu_{\alpha}}}}{\sqrt{N}} \right)\nonumber\\
&\le\frac{C(\varphi_N)^{{ C_0}\xi}}{\alpha^{1/3}}\left(\frac{1}{N^{2/3}}+\frac{\lambda^{3/2}}{N^{5/12}}+ \frac{\lambda\alpha^{1/3}}{\sqrt{N}}+\frac{\lambda\sqrt{|{ \mu_\alpha}-{ \gamma_\alpha}}|}{N^{1/6}} \right)\,,\label{last equation}
\end{align}
where we used $\kappa_{\mu_{\alpha}}\le\kappa_{\gamma_\alpha}+|\mu_\alpha-\gamma_\alpha|$ and $\kappa_{\gamma_\alpha}\sim(\alpha/N)^{2/3}$. Next, since $\alpha=N n_{fc}(\gamma_\alpha) \sim { N} (|L_1|+\gamma_{\alpha})^{3/2}$, we find for $\alpha\in U_k$, ($k\ge 1$),
\begin{align*}
 \alpha\ge cN\left(2^k(\varphi_N)^{{ C_1}\xi}\frac{\lambda}{\sqrt{N}}\right)^{3/2}\gg N^{1/4}\,,
\end{align*}
hence $\alpha^{-1/3}\ll N^{-1/12}$. Using Young's inequality, we can absorb the last term on the right side of~\eqref{last equation} into the left side and we obtain
\begin{align*}
  |\mu_\alpha-\gamma_\alpha|\le(\varphi_N)^{C\xi}\left( \frac{1}{\alpha^{1/3}N^{2/3}} +\frac{\lambda^2}{\alpha^{2/3}N^{1/3}} +\frac{\lambda}{\sqrt{N}}\right)\,,
\end{align*}
on $\Xi$, for $\alpha\in U_k$, some $C$ sufficiently large. The proof is completed by noticing that the event $\Xi$ has $(\xi,\nu)$-high probability.
\end{proof}

We conclude this section with the proof of Theorem~\ref{rigidity of eigenvalues}.

\begin{proof}[Proof of Theorem~\ref{rigidity of eigenvalues}]
 We restrict the discussion to eigenvalues with $\alpha\le N/2$, the other eigenvalues are dealt with in the same way. From $\alpha/N=n_{fc}(\gamma_\alpha)\sim(|L_1|+\gamma_\alpha)^{3/2}$, we find that
\begin{align}\label{alpha small}
 \alpha\le (\varphi_N)^{C\xi}(1+\lambda^{3/2}N^{1/4})\,,
\end{align}
if $\alpha$ is as in item~$i$ of Lemma~\ref{rigidity lemma}.  Combing the conclusions of items $i$ and $ii$ of Lemma~\ref{rigidity lemma} with~\eqref{alpha small} completes the proof of the theorem.
\end{proof}

\begin{appendix}

\section{Appendix: Free Convolution Measure and Stability Bounds}\label{appendixA}
\subsection{Introduction}
In this appendix, we discuss some properties of the (rescaled) free convolution measure,~$\mu_{fc}$, defined through the functional equation
\begin{align}\label{free convolution equation}
 m_{fc}(z)=\int_{-1}^{1}\frac{\dd\mu_{fc}(v)}{\lambda v-z-\mfc(z)}\,,\quad\quad z=E+\ii\eta\in\C^+\,,
\end{align}
such that $\im m_{fc}(z)>0$, for $\eta>0$; c.f., Equation~\eqref{eq111}. Here $\lambda\ge 0$ and we assume that $\mu$  is an absolutely continuous measure, with bounded and continuous density $\mu(v)$ such that $\mathrm{supp}\,\mu=[-1,1]$. For simplicity, we always assume that $\mu$ is centered, although this is not essential for our argument.

To see that Equation~\eqref{free convolution equation} has a unique solution such that $\im m_{fc}(E+\ii\eta)>0$, for $\eta>0$, one can choose $\eta>2$ first. Then it is straightforward to check that the right side of~\eqref{free convolution equation} is a contraction (in the sup-norm on the set of analytic function on the upper half plane with positive imaginary part). The fixed point equation~\eqref{free convolution equation} thus has a unique solution for $\eta>2$. By analytic continuation, the solution extends to the whole upper half plane. We leave the details aside and refer, e.g., to~\cite{PV}.

A deep study of the equation~\eqref{free convolution equation}, with slightly different conventions, can be found in~\cite{B}. One important result of~\cite{B} is the following: The measure $\mu_{fc}$ is absolutely continuous with respect to Lebesgue measure, in particular, we have $\pi \mu_{fc}(E)=\lim_{\eta\searrow 0}\im\mfc(E+\ii\eta)$. For general probability measures (of bounded support), the support of $\mu_{fc}$ may consist of several disjoint intervals, however, under our assumptions, the support of~$\mu_{fc}$ is a single interval, i.e., $\supp \mu_{fc}=[L_1,L_2]$, with $L_1<0<L_2$; see~Lemma~\ref{squareroot of mufc}. We refer to~\cite{B} for a discussion of the general case.

We are mainly interested in the behaviour of $\mu_{fc}(E)$ and $\im m_{fc}(E+\ii\eta)$, for $E\in\R$ close to $L_1$, $L_2$ respectively. We distinguish the cases $\lambda\le 1$ and $\lambda>1$:

For the former case, it was already pointed out in~\cite{B} (see also~\cite{ON, S2}) that $\mu_{fc}$ has a square root behaviour near $L_2$, i.e., $\mu_{fc}(L_2-\kappa)\sim\sqrt{\kappa}$, $\kappa\ge 0$, and similar for $L_1$. 

For $\lambda>1$, we will restrict our attention to Jacobi measures, a special class of measures whose densities are of the form
\begin{align*}
 \mu(v)= Z^{-1} (1+v)^{\a}(1-v)^{\b} d(v)\chi_{[-1,1]}(v)\,,
\end{align*}
where $\a,\b>-1$, $d\in C^1([-1,1])$ with $d(v)>0$, $v\in[-1,1]$ and the normalization constant $Z$ is appropriately chosen so that $\mu$ becomes a probability density. Again, for simplicity, we will always assume that $\mu$ is centered. Note that we also admit exponents $\a,\b$ smaller than zero, thus $\mu(v)\to\infty$ as $v\to\pm 1$ is allowed. As it turns out, the square root behaviour at the endpoint of the support persists for $\lambda>1$, in case we have $-1<\a,\b\le 1$, respectively. However, if $\a,\b>1$, there exists $\lambda_0>1$, such that for any $\lambda>\lambda_0$, we have $\mu_{fc}(L_1-\kappa)\sim \kappa^{\b}$; for a precise statement see Lemma~\ref{general case - large lambda}.

\subsection{Case $\lambda\le 1$}
In this subsection, we choose $\lambda\le 1$. Adopting the proof of Proposition 2 in~\cite{S2}, we have the following result:

\begin{lemma}\label{squareroot of mufc}
 Let $\mu$ be a centered probability measure supported on $[-1,1]$. Assume that $\mu$ has a continuous, strictly positive, bounded density $\mu(v)$ on $(-1,1)$. Suppose that $0\leq \lambda \leq 1$. Then, there exits $L_1,L_2\in\R$, with $L_1<0<L_2$, such that the (rescaled) free convolution of $\mu$ with the semicircle law, $\mu_{fc}$, satisfies
\begin{align*}
 \mathrm{supp}\,\mu_{fc}=[L_1,L_2]\,.
\end{align*}
Moreover, denoting by $\kappa_E$ the distance to the endpoints of the support of $\mu_{fc}$, i.e.,
\begin{align*}
 \kappa_E\deq\min\{|E-L_1|,|E-L_2|\}\,,
\end{align*}
we have
\begin{align}\label{edge wiedermal}
 C^{-1}\sqrt{\kappa_E}\le\mu_{fc}(E)\le C\sqrt{\kappa_E}\,,\quad\quad E\in[L_1,L_2]\,,
\end{align}
for some constant $C\ge1$.
\end{lemma}
We briefly outline how the proof in~\cite{S2} can be adopted to our setting: We denote by $\mfc(z)$, $z\in\C^+$, the Stieltjes transform of the free convolution measure $\mu_{fc}$. Define $\tau\deq z+\mfc(z)$ and consider instead of~\eqref{free convolution equation} the equation $F(\tau)=z$, where
\begin{align}\label{definition of F}
 F(\tau)\deq\tau-\int_{-1}^{1}\frac{\dd\mu(v)}{\lambda v-\tau}\,,\quad \tau\in\C^{+}\,.
\end{align}
Note that $\lim_{y\searrow 0} \im F(x+\ii y)=-\pi\mu(x)<0$, for $x\in(-\lambda,\lambda)$, since we have assumed that the density of $\mu$ is bounded and continuous, and strictly positive in the interval $(-1,1)$. Thus $F$ extends to a function on $\R$, which is continuous and bounded, except possibly at the point $\{\pm\lambda\}$.
As shown in~\cite{S2}, the endpoints, $(L_i)$, of the support of $\mu_{fc}$ are characterized as the real valued solutions, $\tau_i$, with $|\tau_i|\ge\lambda$, of the equation $F'(\tau)=0$ ($L_i$ are then obtained by solving $\tau_i=L_i+\mfc(L_i)$). Setting
\begin{align}
 H(\tau)\deq\int_{-1}^{1}\frac{\dd\mu(v)}{(\lambda v-\tau)^2}\,,\quad\quad\tau\in\C^+\,,
\end{align}
a point $E\in\R$ is an endpoint of the support of $\mu_{fc}$, if $H(\tau)=1$, $|\tau|\ge\lambda$, $\tau=E+\mfc(E)\in\R$. Since $\lambda\le 1$ and $\mu$ is centered, we have from Jensen's inequality
\begin{align}\label{jensen}
 H(\lambda)=\int_{-1}^{1}\frac{\dd\mu(v)}{(\lambda v-\lambda)^2} > \frac{1}{\lambda^2} \frac{1}{\left(\int\dd\mu(v)(v-1)\right)^2} = \frac{1}{\lambda^2}\ge 1\,.
\end{align}
Here, the first inequality is strict since $\mu$ is absolutely continuous. Since $H(\tau)$ is monotone decreasing (on $\R$) as $|\tau|\to\infty$, we conclude that there are only two real solutions $\tau_1,$ $\tau_2$. One then checks that the endpoints of the support of $\mu_{fc}$, $L_1$ and $L_2$ satisfy $L_1<-2$ and $L_2>2$. The square root behaviour of $\mu_{fc}$ at $L_i$, i.e.,~\eqref{edge wiedermal}, follows as in~\cite{S1}: It suffices to observe that $F''(\tau_i)\not=0$, thus by the inverse function theorem, we have, for $z\in\C$ in a neighborhood of $L_i$, $F^{-1}(z)=\tau_i+c_i\sqrt{z-L_i}(1+A_i(\sqrt{z-L_i}))$ (such that $\im F^{-1}(z)\ge 0$, for $z\in\C^+$), for real constants $c_i\not=0$ and analytic functions $A_i$, with $|A_i|\le 1$ in a neighborhood of zero. This concludes our discussion on the proof of Lemma~\ref{squareroot of mufc}.

As an important corollary of the proof of Lemma~\ref{squareroot of mufc}, we have the following stability bound already pointed out in~\cite{S2}:

\begin{corollary}\label{cor: stability bound}
 Under the assumptions of Lemma~\ref{squareroot of mufc} there exist constants $C,c>0$, such that
\begin{align*}
 c\le|\lambda v-z-\mfc(z)|\le C\,, \quad z=E+\ii\eta\,,
\end{align*}
for any $\lambda v\in(-\lambda,\lambda)$ and $|E|\le E_0$, $0<\eta\le 3$.
\end{corollary}
\begin{proof}
For the upper bound, we note that $|\mfc(z)|\le 1$, as follows from considering the imaginary part of $\mfc$ in~\eqref{free convolution equation}. For the lower bound, note that in a neighborhood of $L_i$, $\re (z+ \mfc(z))=\re \tau(z)=\re \tau_{i}+\caO(|z-L_i|^{1/2})$. Since $|\tau_i|>1$, $|\re(z+\mfc(z))|>1$, for $|z-L_i|<\epsilon$ for a sufficiently small $\epsilon>0$. For $|\re z|\ge |L_i|+ (\epsilon /2)$, the estimate is trivial. In the region not covered by the two preceding estimates, we must have $\im\tau>c$, thus $\im \mfc+\eta>c$. The claim follows.
\end{proof}

\subsection{Case $\lambda>1$}
In this subsection, we choose for simplicity $\mu$ as a Jacobi measure, i.e., $\mu$ is described in terms of its density
\begin{align}\label{jacobi measure}
 \mu(v)= Z^{-1} (1+v)^{\a}(1-v)^{\b}d(v)\chi_{[-1,1]}(v)\,,
\end{align}
where $\a,\b>-1$, $d\in C^1([-a,b])$ such that $d(v)>0$, $v\in[-a,b]$ and $Z$ is an appropriately chosen normalization constant such that $\mu$ is a probability measure. Below, we will assume, for simplicity of the arguments, that $\mu$ is centered, but this condition can easily be relaxed.

\begin{lemma}\label{squareroot of mufc - large lambda}
Let $\mu$ be a centered Jacobi measure. Suppose that $\lambda > 1$. If $-1<\a, \b \leq 1$, the results in Lemma~\ref{squareroot of mufc} and Corollary~\ref{cor: stability bound} hold true.
\end{lemma}

\begin{proof}
We can apply the same argument as in the proof of Lemma \ref{squareroot of mufc}. The only thing we need to prove is that
\begin{equation*}
H(\lambda + \epsilon)=\int_{-1}^{1}\frac{\dd\mu(v)}{(\lambda v- \lambda - \epsilon)^2} > 1\,,
\end{equation*}
for any sufficiently small $\epsilon > 0$, and a similar estimate for $H(-\lambda-\epsilon)$.

From the assumptions, we find that there exist constants $C,C_0 > 0$ such that $\mu(v)\geq C(1-v)^{\b} \geq C_0 (1-v)$ for any $v \in (0, 1)$. Let $n \deq e^{1 + \lambda^2 C_0^{-1}}$ and choose $\epsilon < 1/n$. Then, we have
\begin{equation*}
H(\lambda + \epsilon) \geq C_0 \int_{1 - (n-1)\epsilon / \lambda}^{1} \frac{(1-v) \dd v}{(\lambda v- \lambda - \epsilon)^2} = \frac{C_0}{\lambda^2} \int_{\epsilon / \lambda}^{n \epsilon / \lambda} \frac{t - (\epsilon / \lambda)}{t^2} \dd t = \frac{C_0}{\lambda^2} (\log n - 1 + \frac{1}{n} ) > 1\,.
\end{equation*}
From the continuity of $H$, we get the desired results. The same argument applies to $H(-\lambda-\epsilon)$.
\end{proof}

For $\a,\b>1$, we have the following result:

\begin{lemma}\label{general case - large lambda}

Let $\mu$ be a centered Jacobi measure with $\a,\b>1$. Define
\begin{align*}
\lambda_2\deq \left( \int_{-1}^1 \frac{\mu(v) \dd v}{(1-v)^2} \right)^{1/2}\,,\quad\quad \tau_2 \deq \int_{-1}^1 \frac{\mu(v) \dd v}{1-v} \,.
\end{align*}
Then, there exist $L_1<0<L_2$ such that the support of $\mu_{fc}$ is $[L_1,L_2]$. Moreover,
\begin{itemize}
 \item [$i.$] if $\lambda<\lambda_2$, then for $0\le\kappa\le L_2$,
\begin{align}
 C^{-1}\sqrt{\kappa}\le \mu_{fc}(L_2-\kappa)\le C \sqrt\kappa\,,
\end{align}
for some $C\ge1$.
\item[$ii.$] if $\lambda>\lambda_2$, then $L_2=\lambda+(\tau_2/\lambda)$ and, for $0\le\kappa\le L_2$,
\begin{align}\label{exponent beta}
 C^{-1}{\kappa}^{\b}\le \mu_{fc}(L_2-\kappa)\le C \kappa^{\b}\,,
\end{align}
for some $C\ge1$.
\end{itemize}
Moreover, for $0\le E\le E_0$, $0<\eta\le 2$, $z=E+\ii\eta$, $v\in[-1,1]$, 
\begin{align*}
|\lambda v-z-\mfc(z)|
\end{align*}
remains bounded from below in case $i$ uniformly in $z$ and $v$, but in case $ii$, it can be arbitrarily small as $v \to 1$, $E = L_2$, and $\eta \to 0$.

Similar statements hold for the lower endpoint $L_1$ of the support of $\mu_{fc}$, with $\tau_2$ and $\lambda_2$ replaced by
\begin{align}
 \lambda_1\deq \left( \int_{-1}^1 \frac{\mu(v) \dd v}{(1+v)^2} \right)^{1/2}\,,\quad\quad\tau_1 \deq \int_{-1}^1 \frac{\mu(v) \dd v}{1+v} \,.
\end{align}
\end{lemma}

\begin{proof}

We first note that $0 < \lambda_2, \tau_2 < \infty$, for $\b > 1$. Since $\mu(v)>0$, for $v\in(-1,1)$, $\mu_{fc}$ is supported on a single interval. Consider now
\begin{equation*}
H(\lambda)=\int_{-1}^{1}\frac{\mu(v) \dd v}{(\lambda v- \lambda)^2} = \frac{1}{\lambda^2} \int_{-1}^{1}\frac{\mu(v) \dd v}{(v-1)^2} = \left( \frac{\lambda_2}{\lambda} \right)^2\,.
\end{equation*}
When $\lambda < \lambda_2$, we may follow the proof of Lemma \ref{squareroot of mufc} to prove the claims in $i$.

We now choose $\lambda > \lambda_2$. We claim that there exists a unique continuous bounded curve $\gamma$ in $\C^{+}$ on which $\im F(\tau)=0$. For $z\in\C^+$,
\begin{align*}
 \im F(\tau)=\im\tau\left(1-\int\frac{\dd\mu(v)}{|\lambda v-\tau|^2} \right)\,.
\end{align*}
We know that the non-negative continuous function
\begin{align*}
\widetilde{H}(\tau)\deq\int\frac{\dd\mu(v)}{|\lambda v-\tau|^2} = \int\frac{\dd\mu(v)}{(\lambda v-\re\tau)^2+(\im\tau)^2}\,,\quad\quad\tau\in\C^{+} \cup (\R \backslash [-\lambda, \lambda]) \,,
\end{align*}
is monotonically decreasing in $\im\tau$. Let $\tau=x+\ii y$. For $x\in(-\lambda,\lambda)$, the continuity of $\mu$ implies that, as $y\searrow 0$, $y \widetilde{H}(x+\ii y)\to\pi \mu(x)>0$, hence $\widetilde{H}(x+\ii y)\to \infty$. Since $\widetilde{H}(x+\ii y)$ is monotonically decreasing as $y$ increases and $\widetilde{H}(x+\ii y) \to 0$ as $y \to \infty$, the equation $y=y\widetilde{H}(x+\ii y)$ has a unique solution $0<y<\infty$.
The analyticity of $F$ in the upper half plane then implies that on the interval $(-\lambda,\lambda)$ there exists a single bounded curve such that the imaginary part of $F$ vanishes on it.

The endpoints of the support of $\mu_{fc}$ are characterized as the points where the curve $\gamma$ approaches to the real line. Since $\widetilde{H}(\lambda) = H(\lambda) < 1$, the curve $\gamma$ does not connect with the real axis on $\R^{+} \backslash (0, \lambda)$. Since this curve cannot end at some point where $F$ is analytic, we can conclude that the curve approaches to $\lambda$ on $\R^{+}$. When $\tau = \lambda$, we have
\begin{equation*}
z = \tau - m_{fc} (z) = \lambda - \int_{-1}^{1} \frac{\mu(v) \dd v}{\lambda v- \lambda} = \lambda + \frac{\tau_2}{\lambda}\,,
\end{equation*}
which corresponds to the endpoint, $L_2$, of the support of $\mu_{fc}$ on $\R^{+}$.

To prove~\eqref{exponent beta}, let
\begin{equation*}
\tau = \lambda - \lambda k +\ii  \lambda y\,,\quad \quad z = \lambda + \frac{\tau_2}{\lambda} - \kappa +\ii \eta\,.
\end{equation*}
Considering the imaginary part of $m_{fc}$, we obtain
\begin{equation}\label{upp}
\lambda y - \eta = \im m_{fc}(z) = \im \int_{-1}^{1} \frac{\mu(v) \dd v}{\lambda v- \tau} = \frac{y}{\lambda} \int_{-1}^{1} \frac{\mu(v) \dd v}{(v-1+k)^2 + y^2}\,.
\end{equation}
We claim that in the limit $\eta\searrow 0$,
\begin{align}\label{bound y}
y \sim (k+y)^{\b}\,, 
\end{align}
for $\kappa,y\ll 1$.
 
For the upper bound, we consider first the case $y < k$: Let $\epsilon = \min \{ 1/2, ( \lambda^2 / \lambda_2^2)-1 \} $, then we have
\begin{equation} \label{up}
y \int_{-1}^{1} \frac{\mu(v) \dd v}{(v-1+k)^2 + y^2} = y \left( \int_{-1}^{1-8\epsilon^{-1} k} + \int_{1-8\epsilon^{-1} k}^{1-k-y} + \int_{1-k-y}^{1-k+y} + \int_{1-k+y}^1 \right) \frac{\mu(v) \dd v}{(1-v-k)^2 + y^2}\,.
\end{equation}
The first term in \eqref{up} can be estimated as
\begin{equation} \begin{split} \label{up_1}
y \int_{-1}^{1-8\epsilon^{-1} k} \frac{\mu(v) \dd v}{(1-v-k)^2 + y^2} &\leq y \int_{-1}^{1-8\epsilon^{-1} k} \frac{\mu(v) \dd v}{(1-v-k)^2} \leq y \left(1 + \frac{\epsilon}{2} \right) \int_{-1}^{1-8\epsilon^{-1} k} \frac{\mu(v) \dd v}{(1-v)^2} \\
& \leq \left(1 + \frac{\epsilon}{2} \right) \lambda_2^2 y\,.
\end{split} \end{equation}
Here, we used that $v \leq 1-8\epsilon^{-1} k$ implies that $1-v-k \geq (1 - \epsilon/8 )(1-v)$, hence
\begin{equation*}
\frac{1}{(1-v-k)^2} \leq \left( 1 - \frac{\epsilon}{8} \right)^{-2} \frac{1}{(1-v)^2} \leq \left(1 + \frac{\epsilon}{2} \right) \frac{1}{(1-v)^2}\,.
\end{equation*}
The second term in \eqref{up} can be estimated as
\begin{equation*}
y \int_{1-8\epsilon^{-1} k}^{1-k-y} \frac{\mu(v) \dd v}{(1-v-k)^2 + y^2} \leq Cy \int_{1-8\epsilon^{-1} k}^{1-k-y} \frac{k^{\b} \dd v}{(1-v-k)^2} \leq C k^{\b}\,.
\end{equation*}
The third term in \eqref{up} can be estimated as
\begin{equation*}
y \int_{1-k-y}^{1-k+y} \frac{\mu(v) \dd v}{(v-1+k)^2 + y^2} \leq C y \int_{1-k-y}^{1-k+y} \frac{(k+y)^{\b} \dd v}{y^2} \leq C (y+k)^{\b}\,.
\end{equation*}
The last term in \eqref{up} can be estimated as
\begin{equation*}
y \int_{1-k+y}^1 \frac{\mu(v) \dd v}{(1-v-k)^2 + y^2} \leq y \int_{1-k+y}^1 \frac{(k-y)^{\b} \dd v}{(1-v-k)^2} = y \int_y^k \frac{(k-y)^{\b} \dd w}{w^2} \leq C k^{\b}\,.
\end{equation*}

Thus, as $\eta \searrow 0$, we have that 
\begin{equation*}
\lambda y \leq \frac{1}{\lambda} \left(1 + \frac{\epsilon}{2} \right) \lambda_2^2 y + C (k+y)^{\b}\,.
\end{equation*}
Since
\begin{equation*}
\lambda - \frac{1}{\lambda} \left(1 + \frac{\epsilon}{2} \right) \lambda_2^2 = \frac{\lambda_2^2}{\lambda} \left( \frac{\lambda^2}{\lambda_2^2} - 1 - \frac{\epsilon}{2} \right) \geq \frac{\epsilon \lambda_2^2}{2 \lambda}\,,
\end{equation*}
we obtain
\begin{equation*}
y \leq C (k+y)^{\b}\,,
\end{equation*}
provided $y<k$. 

When $y \geq k$, we decompose the integral in \eqref{up} as
\begin{equation*}
y \int_{-1}^{1} \frac{\mu(v) \dd v}{(v-1+k)^2 + y^2} = y \left( \int_{-1}^{1-8\epsilon^{-1} k} + \int_{1-8\epsilon^{-1} k}^1 \right) \frac{\mu(v) \dd v}{(1-v-k)^2 + y^2}\,.
\end{equation*}
The first term is again estimated as in \eqref{up_1}. The second term can be estimated as
\begin{equation*}
y \int_{1-8\epsilon^{-1} k}^1 \frac{\mu(v) \dd v}{(1-v-k)^2 + y^2} \leq Cy \int_{1-8\epsilon^{-1} k}^1 \frac{k^{\b} \dd v}{y^2} \leq C k^{\b + 1} y^{-1} \leq C y^{\b}\,.
\end{equation*}
Following the argument we used for the case $y < k$, we find the relation $y \leq C y^{\b}$ in this case. For sufficiently small $y$, this is impossible, so this case does not happen.

To complete the proof of~\eqref{bound y}, we need a lower bound: Observe that
\begin{equation} \label{low}
y \int_{-1}^{1} \frac{\mu(v) \dd v}{(v-1+k)^2 + y^2} \geq y \int_{1-k-y}^{1-k} \frac{\mu(v) \dd v}{(v-1+k)^2 + y^2} \geq Cy \int_{1-k-y}^{1-k} \frac{\mu(v) \dd v}{y^2} \geq C k^{\b}\,,
\end{equation}
and~\eqref{bound y} follows from~\eqref{upp} and $k \geq y$. When $y, k \ll 1$, ~\eqref{bound y} implies that $k \gg y$ and since $y \to C \mu_{fc}(L_2-\kappa)$ as $\eta \to 0$, we have $\mu_{fc}(L_2-\kappa) \sim y \sim k^{\b}$.

To compare $k$ and $\kappa$, we consider the real part of $m_{fc}$ and get
\begin{equation*}
\kappa - \lambda k - \frac{\tau_2}{\lambda} = \re m_{fc}(z) = \re \int_{-1}^{1} \frac{\mu(v) \dd v}{\lambda v- \tau} = \frac{1}{\lambda} \int_{-1}^{1} \frac{(v-1+k) \mu(v) \dd v}{(v-1+k)^2 + y^2}\,.
\end{equation*}
From the definition of $\tau_2$, we find that
\begin{equation*}
\kappa - \lambda k = \frac{1}{\lambda} \int_{-1}^{1} \left( \frac{(v-1+k) \mu(v) \dd v}{(v-1+k)^2 + y^2} + \frac{\mu(v) \dd v}{1-v} \right) = \frac{1}{\lambda} \int_{-1}^{1} \frac{\mu(v) \dd v}{1-v} \cdot \frac{k(v-1) + k^2 + y^2}{(v-1+k)^2 + y^2}\,.
\end{equation*}
We now separate the integral and estimate each term as in \eqref{low} and \eqref{up}. We then get
\begin{equation*}
\int_{-1}^{1} \frac{k \mu(v) \dd v}{(v-1+k)^2 + y^2} \sim \frac{k}{y} (k+y)^{\b}
\end{equation*}
and
\begin{equation*}
\frac{1}{\lambda} \int_{-1}^{1} \frac{\mu(v) \dd v}{1-v} \cdot \frac{k^2 + y^2}{(v-1+k)^2 + y^2} \sim \frac{k^2 + y^2}{y} (k+y)^{\b - 1}\,.
\end{equation*}
Recalling that $y \sim k^{\b}$, when $y, k \ll 1$, we find that $\kappa - \lambda k = O(k)$. Therefore we get
\begin{equation*}
\mu_{fc}(L_2-\kappa) \sim y \sim k^{\b} \sim \kappa^{\b}\,,
\end{equation*}
as $\kappa\searrow0$. Finally, it is easy to see that $|\lambda v-z-\mfc(z)|$ is not bounded from below: Choosing $z=L_2$, we have $\im(\mfc(L_2))=0$, but $\re (\lambda v-L_2-\mfc(L_2))=\lambda v-\lambda$. This proves the claims in $ii$.

\end{proof}

\subsection{Square root behaviour of $\mfc$ and further stability bounds}
In this subsection, we prove that the Stieltjes transform $m_{fc}$ inherits the square root behavior from $\mu_{fc}$: 
\begin{lemma}\label{lem:m_fc_bound}
 Assume that $\mu_{fc}$ has support $[L_1,L_2]$ and satisfies
\begin{align}\label{another squareroot}
 C^{-1}\sqrt{\kappa}\le\mu_{fc}(L_2-\kappa)\le C \sqrt{\kappa}\,,
\end{align}
$0\le\kappa\le L_2$, $C\ge 1$. Then,
\begin{itemize}
 \item[$i.$] 
for $z=L_2-\kappa+\ii\eta$, with $0\le\kappa\le L_2$ and $0<\eta\le 2$, we have, $C\ge 1$,
\begin{align*}
 C^{-1}\sqrt{\kappa+\eta}\le \im\mfc(z)\le C \sqrt{\kappa+\eta}\,;
\end{align*}
\item[$ii.$]
for $z=L_2+\kappa+\ii\eta$, with $0\le\kappa\le 1$ and $0<\eta\le 2$, we have, $C\ge 1$,
\begin{align*}
 C^{-1}\frac{\eta}{\sqrt{\kappa+\eta}}\le \im\mfc(z)\le C\frac{\eta}{\sqrt{\kappa+\eta}}\,.
\end{align*}
\end{itemize}
The analogous statements hold for $z=L_1\pm\kappa+\ii\eta$.
\end{lemma}

\begin{proof}
We start with the claim $i$:
Notice that
\begin{equation*}
\im m_{fc}(z) = \im \int \frac{\dd\mu_{fc}(x)}{x-z} = \int \frac{\eta \:\dd\mu_{fc}(x) }{(x-L_2+\kappa)^2 + \eta^2}\,.
\end{equation*}
To prove the lower bound, consider the following cases:

\textit{Case 1.} When 
$\kappa,\eta<1/2$, computing the integral from $x=L_2 - \kappa - 2 \eta$ to $x=L_2 - \kappa - \eta$, we find from~\eqref{another squareroot} that
\begin{equation*}
\im m_{fc}(z) = \int \frac{\eta \: \dd\mu_{fc}(x) }{(x-L_2+\kappa)^2 + \eta^2} \geq C \int_{L_2 - \kappa - 2 \eta}^{L_2 - \kappa - \eta} \frac{\eta \sqrt{\kappa + \eta}}{\eta^2} \dd x \geq C\sqrt{\kappa+\eta}\,.
\end{equation*} 

\textit{Case 2.} When $\kappa \geq 1/2$, $\eta<1/2$, we obtain from~\eqref{another squareroot} that
\begin{equation*}
\im m_{fc}(z)\geq C\int_{L_2-\kappa+\eta/8}^{L_2-\kappa+\eta/4}\frac{\eta \: \dd\mu_{fc}(x) }{(x-L_2+\kappa)^2 + \eta^2} \geq C\sqrt{\kappa}\int_{L_2-\kappa+\eta/8}^{L_2-\kappa+\eta/4}\frac{\eta \: \dd x }{ \eta^2} 
\ge C \sqrt{\kappa} \geq C \sqrt{\kappa+\eta}\,.
\end{equation*}

\textit{Case 3.} When $\eta \geq 1/2$, we have the bound
\begin{equation*}
\im m_{fc}(z) = \int \frac{\eta \: \dd\mu_{fc}(x) }{(x-L_2-\kappa)^2 + \eta^2} \geq C \int \frac{\eta \: \dd\mu_{fc}(x) }{\eta^2} = \frac{C}{\eta} \geq C \sqrt{\kappa+\eta}\,.
\end{equation*}

This proves the lower bound. To prove the upper bound, we consider the following cases:

\textit{Case 1.} When $\eta < \kappa < 1/2$, from~\eqref{another squareroot} we have
\begin{align*} 
\im m_{fc}(z)& = \int \dd\mu_{fc}(x) \frac{\eta}{(x-L_2+\kappa)^2 + \eta^2} \\
&\leq C \eta \int_{-L_1}^{L_2 - \kappa - \eta} \frac{\sqrt{L_2 - x}}{(x-L_2+\kappa)^2} \dd x + C \eta \int_{L_2 - \kappa - \eta}^{L_2 - \kappa + \eta} \frac{\sqrt{\kappa + \eta}}{\eta^2} \dd x + C \eta \int_{L_2 - \kappa + \eta}^{L_2} \frac{\sqrt{\kappa}}{(x-L_2+\kappa)^2} \dd x \\
&\leq C \eta \int_{\eta}^{L_1 + L_2 - \kappa} \frac{\sqrt{y + \kappa}}{y^2} \dd y + C \sqrt{\kappa + \eta} + C \eta \int_{\eta}^{\kappa} \frac{\sqrt{\kappa}}{y^2} \dd y \leq C \sqrt{\kappa + \eta}\,.
 \end{align*}

\textit{Case 2.} When $\kappa < \eta < 1/2$, a calculation similar to $\textit{Case 1}$ proves the same bound.

\textit{Case 3.} When $\kappa \geq 1/2$,  we have
\begin{equation*}
\im m_{fc}(z) \leq C\int\frac{\eta\:\dd x}{(x-L_2+\kappa)^2+\eta^2}\leq C  \leq C \sqrt{\kappa + \eta}\,.
\end{equation*}

\textit{Case 4.} When $\eta \geq 1/2$, we have the trivial bound
\begin{equation*}
\im m_{fc}(z) \leq |m_{fc}(z)| \leq \frac{1}{\eta} \leq C \sqrt{\kappa + \eta}\,.
\end{equation*}

This completes the proof of statement $i$. To prove $ii$, we proceed similarly:

\textit{Case 1.} When $\kappa > \eta$, computing the integral from $x=L_2 - \kappa$ to $x=L_2 - 2 \kappa$, we get
\begin{equation*}
\im m_{fc}(z) = \int \frac{\eta \: \dd\mu_{fc}(x) }{(x - L_2 - \kappa)^2 + \eta^2} \geq C \int_{L_2 - \kappa}^{L_2 - 2 \kappa} \frac{\eta \sqrt{\kappa}}{\kappa^2} \dd x \geq \frac{C \eta}{\sqrt{\kappa}} \geq \frac{C \eta}{\sqrt{\eta + \kappa}}\,.
\end{equation*}
For the upper bound, we find that
\begin{equation*} \begin{split}
 \im m_{fc}(z)& = \int \frac{\eta \: \dd\mu_{fc}(x) }{(x - L_2 - \kappa)^2 + \eta^2} \leq C \eta \int_{L_2 - \kappa}^{L_2}  \frac{\sqrt{\kappa}}{\kappa^2} \dd x + C \eta \int_{-L_1}^{L_2 - \kappa} \frac{\sqrt{L_2 -x}}{(x - L_2)^2} \dd x \\
&\leq \frac{C \eta}{\sqrt{\kappa}} + C \eta \int_{\kappa}^{L_1 + L_2} \frac{\sqrt{y}}{y^2} \leq \frac{C \eta}{\sqrt{\kappa}} \leq \frac{C \eta}{\sqrt{\eta + \kappa}}\,.
\end{split} \end{equation*}

\textit{Case 2.} When $\kappa \leq \eta$, computing the integral from $x=L_2 - (\eta /2)$ to $x=L_2 - \eta$, we obtain
\begin{equation*}
\im m_{fc}(z) = \int \frac{\eta \: \dd\mu_{fc}(x) }{(x - L_2 - \kappa)^2 + \eta^2} \geq C \int_{L_2 - (\eta /2)}^{L_2 - \eta} \frac{\eta \sqrt{\eta}}{\eta^2} \dd x \geq C \sqrt{\eta} \geq \frac{C \eta}{\sqrt{\eta + \kappa}}\,.
\end{equation*}
For the upper bound, we find that
\begin{equation*} \begin{split}
 \im m_{fc}(z) &= \int \frac{\eta \:\d d\mu_{fc}(x) }{(x - L_2 - \kappa)^2 + \eta^2} \leq C \eta \int_{L_2 - \eta}^{L_2}  \frac{\sqrt{\eta}}{\eta^2} \dd x + C \eta \int_{-L_1}^{L_2 - \eta} \frac{\sqrt{L_2 -x}}{(x - L_2)^2} \dd x \\
&\leq C \sqrt{\eta} + C \eta \int_{\eta}^{L_1 + L_2} \frac{\sqrt{y}}{y^2} \leq C \sqrt{\eta} \leq \frac{C \eta}{\sqrt{\eta + \kappa}}\,.
\end{split} \end{equation*}
This completes the proof of the lemma.
\end{proof}

Finally, we show that $|1-R_2(z)| \sim \sqrt{{\kappa_E+\eta}}$ and $R_3(z) = \caO(1)$; see \eqref{eq.171} for the definitions.

\begin{lemma} \label{lem:stability_bound}
Assume that $\mu_{fc}$ has support $[L_1,L_2]$ and satisfies
\begin{align*}
 C^{-1}\sqrt{\kappa_E}\le\mu_{fc}(E)\le C \sqrt{\kappa_E}\,,
\end{align*}
with $C\ge1$, where $\kappa_E\deq\min\{|E-L_1|,|E-L_2|\}$, denotes the distance to the endpoints of the support of $\mu_{fc}$. Moreover, assume the stability bound
\begin{align*}
 c<|\lambda v-z-\mfc(z)|\le C_0\,,
\end{align*}
with $C_0,c>0$, for any $|E|\le E_0$, $0<\eta\le 3$ and $|v| \leq 1$. Then, we have the followings:
\begin{itemize}
\item[i.]
There exists a constant $C\ge1$ such that for any $|E|\le E_0$, $0<\eta\le 3$,
\begin{equation*}
C^{-1} \sqrt{\kappa_E + \eta} \leq \left| 1 - \int \frac{\dd \mu (v)}{(\lambda v - z - m_{fc}(z))^2} \right|  \leq C \sqrt{\kappa_E + \eta}\,.
\end{equation*}

\item[ii.]
There exists a constant $C$ such that $|R_3(z)| \leq C$ uniformly in $z\in \caD_L$ and $\lambda\in \caD_{\lambda_0}$. Moreover, there exist constants $c$ and $\epsilon_0$ such that $|R_3(z)| \geq c$ whenever $z\in \caD_L$ satisfies $|z-L_i| < \epsilon_0$, $i=1,2$.

\end{itemize}
\end{lemma}

\begin{proof}
Since $c \leq |\lambda v-z - m_{fc}(z)|$, it is easy to see that $|R_3| < C$. Furthermore, it is proved in \cite{S2} that $R_3 (L_2) > 0$. Since $R_3(z)$ is an analytic function of $z$ in a neighborhood of $L_i$, $i=1,2$, this proves the second part of the lemma.

In order to prove the first part of the lemma, we first consider the following decomposition:
\begin{equation*} \begin{split} \label{eq:stability_decomp}
&\left| 1 - \int \frac{\dd \mu (v)}{(\lambda v - z - m_{fc}(z))^2} \right| \\
&\qquad\leq \left| 1 - \int \frac{\dd \mu (v)}{|\lambda v - z - m_{fc}(z)|^2} \right| + \left| \int \frac{\dd \mu (v)}{|\lambda v - z - m_{fc}(z)|^2} - \int \frac{\dd \mu (v)}{(\lambda v - z - m_{fc}(z))^2} \right|\,.
\end{split} \end{equation*}
For the first term in the right side of the decomposition \eqref{eq:stability_decomp}, we have
\begin{equation*}
1 - \frac{\im m_{fc}(z)}{\im (z + m_{fc}(z))} = \frac{\eta}{\im (z + m_{fc}(z))} \leq \frac{\eta}{C \sqrt{\kappa_E + \eta}} \leq C \sqrt{\kappa_E + \eta}\,,
\end{equation*}
if $E\in[L_1,L_2]$ and
\begin{align*}
 1 - \frac{\im m_{fc}(z)}{\im (z + m_{fc}(z))} = \frac{\eta}{\im (z + m_{fc}(z))} \leq \frac{\eta}{C \eta/\sqrt{\kappa_E + \eta}} \leq C \sqrt{\kappa_E + \eta}\,,
\end{align*}
if $E\in[L_1,L_2]^{c}$.
Since
\begin{equation*}
|\re (\lambda v - z - m_{fc}(z))|\,, \im(z + m_{fc}(z)) \leq |z + m_{fc}(z)| + \lambda < C\,,
\end{equation*}
we also find that
\begin{align}  \label{eq:stability_mid}
&\left| \int \frac{\dd \mu (v)}{|\lambda v - z - m_{fc}(z)|^2} -\int \frac{\dd \mu (v)}{(\lambda v - z - m_{fc}(z))^2} \right|\nonumber  \\
&\qquad\qquad= 2 \left| \int \frac{(\im(z + m_{fc}(z)) )^2 +\ii \re (\lambda v - z - m_{fc}(z)) \cdot \im(z + m_{fc}(z))}{|\lambda v - z - m_{fc}(z)|^4} \dd \mu (v) \right|\nonumber \\
&\qquad\qquad\leq C \: \im(z + m_{fc}(z)) \le C \sqrt{\kappa_E + \eta}\,.
\end{align} 
Thus,
\begin{equation*} \begin{split}
&\left| 1 - \int \frac{\dd \mu (v)}{(\lambda v - z - m_{fc}(z))^2} \right|  \leq C\sqrt{\kappa_E + \eta}\,,
\end{split} \end{equation*}
which proves the upper bound.

For the lower bound, we first consider the case $E\in[L_1,L_2]$: If $|\re (z+m_{fc}(z))| > \lambda$, we get
\begin{align*}
\left| \im \int \frac{\dd\mu (v)}{(\lambda v - z - m_{fc}(z))^2} \right| &= 2 \left| \int \frac{\re (\lambda v - z - m_{fc}(z)) \cdot \im(z + m_{fc}(z))}{|\lambda v - z - m_{fc}(z)|^4} \dd\mu (v) \right| \\
&= 2 \int \frac{ | \re (\lambda v - z - m_{fc}(z))| \cdot \im(z + m_{fc}(z))}{|\lambda v - z - m_{fc}(z)|^4} \dd\mu (v) \\
&\geq C \im(z + m_{fc}(z)) \geq C \sqrt{\kappa_E + \eta}\,.
\end{align*}
Hence,
\begin{equation*}
\left| 1 - \int \frac{\dd \mu (v)}{(\lambda v - z - m_{fc}(z))^2} \right| \geq \left| \im \int \frac{\dd \mu (v)}{(\lambda v - z - m_{fc}(z))^2} \right| \ge C \sqrt{\kappa_E + \eta}\,.
\end{equation*}
If $|\re (z+m_{fc}(z))| < \lambda$, then the stability bound $|\lambda v - z - m_{fc}(z)| > c$ implies that $\im (z + m_{fc}(z)) > c$. Then, we get
\begin{align*}
\re \int \frac{\dd \mu (v)}{(\lambda v - z - m_{fc}(z))^2} &= \int \frac{[\re (\lambda v - z - m_{fc}(z))]^2 - [\im(z + m_{fc}(z))]^2 }{|\lambda v - z - m_{fc}(z)|^4} \dd \mu (v) \\
&\leq \int \frac{\dd \mu (v)}{|\lambda v - z - m_{fc}(z)|^2} - C \,.
\end{align*}
Thus,
\begin{align*}
\left| 1 - \int \frac{\dd \mu (v)}{(\lambda v - z - m_{fc}(z))^2} \right| &\geq 1 - \re \int \frac{\dd\mu (v)}{(\lambda v - z - m_{fc}(z))^2} \\
&\geq \int \frac{\dd \mu (v)}{|\lambda v - z - m_{fc}(z)|^2} - \re \int \frac{\dd\mu (v)}{|\lambda v - z - m_{fc}(z)|^2}\\ & \geq C \geq C \sqrt{\kappa_E + \eta}\,.
\end{align*}

In case $E\in[L_1,L_2]^c$, we obtain a lower bound from
\begin{equation*}
\left| 1 - \int \frac{\dd \mu (v)}{(\lambda v - z - m_{fc}(z))^2} \right| \geq 1 - \int \frac{\dd\mu (v)}{|\lambda v - z - m_{fc}(z)|^2}
\end{equation*}
and
\begin{equation*}
1 - \int \frac{\dd \mu (v)}{|\lambda v - z - m_{fc}(z)|^2} = 1 - \frac{\im m_{fc}(z)}{\im (z + m_{fc}(z))} = \frac{\eta}{\im (z + m_{fc}(z))} \geq \frac{C \eta}{\eta / \sqrt{\kappa_E + \eta}} = C \sqrt{\kappa_E + \eta}\,.
\end{equation*}
This completes the proof.

\end{proof}

\end{appendix}

\end{document}